\documentclass{amsart}
\usepackage{color}

\usepackage{mathtools}
\usepackage{amssymb, amsmath, amsthm, amsfonts, amscd, graphicx, subfigure, array, multicol, setspace, caption,tikz, mathtools, multicol}

\usepackage{amsmath,amsthm,pb-diagram,amssymb,comment}
\usepackage{amsfonts,graphicx,color}

\usepackage{amsmath} 
\usepackage{amssymb}
\usepackage{mathrsfs}

\newtheorem{theorem}{Theorem}[section] 
\newtheorem{claim}[theorem]{Claim}
\newtheorem{subc}[theorem]{Subclaim}
\newtheorem{lemma}[theorem]{Lemma} 
 
\newtheorem{conclusion}[theorem]{Conclusion}
\newtheorem{observation}[theorem]{Observation}

\theoremstyle{definition}
\newtheorem{definition}[theorem]{Definition}

\newtheorem{convention}[theorem]{Convention}
\newtheorem{fact}[theorem]{Fact}
\newtheorem{subf}[theorem]{Subfact}

\theoremstyle{remark}
\newtheorem{remark}[theorem]{Remark}

\newtheorem{notation}[theorem]{Notation}
\newtheorem{context}[theorem]{Context}

\newcommand{\tp}{{\rm tp}}
\newcommand{\pp}{{\rm pp}}
\newcommand{\oor}{{\rm or}}
\newcommand{\pcf}{{\rm pcf}}
\newcommand{\stp}{{\rm stp}}

\newcommand{\lub}{{\rm lub}}
\newcommand{\acl}{{\rm acl}}
\newcommand{\otp}{{\rm otp}}
\newcommand{\acc}{{\rm acc}}
\newcommand{\tr}{{\rm tr}}

\newcommand{\cd}{{\rm cd}}
\newcommand{\Ord}{{\rm Ord}}

\newcommand{\Suc}{{\rm Suc}}

\newcommand{\cbe}{{\rm cbe}}

\newcommand{\card}{{\rm card}}
\newcommand{\finite}{{\rm finite}}

\newcommand{\bd}{{\rm bd}}
\newcommand{\eq}{{\rm eq}}

\newcommand{\PC}{{\rm PC}}
\newcommand{\EM}{{\rm EM}}

\newcommand{\Min}{{\rm Min}}
\newcommand{\Dom}{{\rm Dom}}
\newcommand{\Rang}{{\rm Rang}}
\newcommand{\rang}{{\rm rang}}

\newcommand{\rest}{{\restriction}}

\newcommand{\set}{{\rm set}}
\newcommand{\wilog}{{\rm without loss of generality}}
\newcommand{\Wilog}{{\rm Without loss of generality}}

\newcommand{\then}{{\underline{then}}}
\newcommand{\when}{{\underline{when}}}
\newcommand{\Then}{{\underline{Then}}}

\newcommand{\If}{{\underline{if}}}
\newcommand{\Iff}{{\underline{iff}}}
\newcommand{\mn}{{\medskip\noindent}}
\newcommand{\sn}{{\smallskip\noindent}}

\newcommand{\cA}{{\mathscr A}}

\newcommand{\seteq}{{\mathbf \Upsilon}}

\newcommand{\modG}{{\dot{\bbG}}}

\newcommand{\primep}{{\dot{p}}}

\newcommand{\isoI}{{\dot{\mathbb I}}}

\newcommand{\numbIE}{{\dot{I}\dot{E}}}

\newcommand{\numbp}{{\dot{p}}}

\newcommand{\filterD}{{\dot{D}}}

\newcommand{\gb}{{\mathfrak b}}
\newcommand{\ga}{{\mathfrak a}}
\newcommand{\gC}{{\mathfrak C}}

\newcommand{\cH}{{\mathscr H}}

\newcommand{\bF}{{\bf F}}
\newcommand{\bfF}{{\mathbf{F}}}

\newcommand{\bfc}{{\mathbf{c}}}
\newcommand{\bfd}{{\mathbf{d}}}
\newcommand{\cF}{{\mathscr F}}
\newcommand{\cG}{{\mathscr G}}

\newcommand{\bbL}{{\mathbb L}}
\newcommand{\bbI}{{\mathbb I}}

\newcommand{\cM}{{\mathscr M}}

\newcommand{\cP}{{\mathscr P}}

\newcommand{\gs}{{\mathfrak s}}

\newcommand{\bbG}{{\mathbb G}}
\newcommand{\bbZ}{{\mathbb Z}}

\newcommand{\bfn}{{\bf n}}
\newcommand{\cS}{{\mathscr S}}
\newcommand{\cT}{{\mathscr T}}
 
\newcommand{\cU}{{\mathscr U}}

\newcommand{\cf}{{\rm cf}}

\newcount\skewfactor
\def\mathunderaccent#1#2 {\let\theaccent#1\skewfactor#2
\mathpalette\putaccentunder}
\def\putaccentunder#1#2{\oalign{$#1#2$\crcr\hidewidth
\vbox to.2ex{\hbox{$#1\skew\skewfactor\theaccent{}$}\vss}\hidewidth}}

\newenvironment{PROOF}[2][\proofname.]
   {\begin{proof}[#1]\renewcommand{\qedsymbol}{\textsquare$_{\rm #2}$}}
   {\end{proof}}

\usepackage[hidelinks]{hyperref}

\begin{document}
\makeatletter\def\shfiuwefootnote{\gdef\@thefnmark{}\@footnotetext}\makeatother\shfiuwefootnote{Version 2026-04-21\_2. See \url{https://shelah.logic.at/papers/331/} for possible updates.}

\title[A Complicated Family of trees]
{A complicated family of trees with $\omega +1$ levels  \\
Sh331}

\author{Saharon Shelah}

\address{Einstein Institute of Mathematics\\
Edmond J. Safra Campus, Givat Ram\\
The Hebrew University of Jerusalem\\
Jerusalem, 91904, Israel\\
and \\
 Department of Mathematics\\
 Hill Center - Busch Campus \\ 
 Rutgers, The State University of New Jersey \\
 110 Frelinghuysen Road \\
 Piscataway, NJ 08854-8019 USA}
 
\email{shelah@math.huji.ac.il}

\urladdr{http://shelah.logic.at}

\thanks{Was mostly ready in the early nineties, and public to some extent. This was written as Chapter VI of the book \cite{Sh:e}, which hopefully will materialize someday, but in meanwhile is it \cite{Sh:E59}. The intentions were: \cite{Sh:E58} (revising \cite{Sh:229} for Chapter VI), \cite{Sh:421} for Ch. II, \cite{Sh:E59} for Ch.III, \cite{Sh:309} for Ch.IV, \cite{Sh:363} for Ch.V, \cite{Sh:331} for Ch.VI, \cite{Sh:511} for Ch.VII, \cite{Sh:E60}, revision of \cite{Sh:128} for Ch.VIII, \cite{Sh:384} for Ch.IV \cite{Sh:482} for Ch.X, \cite{Sh:E62} for the appendix, and probably \cite{Sh:757}, and \cite{Sh:800}. References like \cite[q17 = Lc2]{Sh:E62} means that c2 is the label of 3.19 in \cite{Sh:E62}, will only help the author if changes in the paper \cite{Sh:E62} will change the number. The reader should note that the version in author's website is usually more updated than the one in the mathematical archive.\\
For versions up to 2019, the author thanks Alice Leonhardt for the beautiful typing. In the latest version, the author would like to thank the typist for his work and is also grateful for the generous funding of typing services donated by Craig Falls. The author would like to thank for partially supporting this research an NSF-BSF 2021: grant with M. Malliaris, NSF 2051825, BSF 3013005232 (2021/10-2026/09). ISF 2320/23: The Israel Science Foundation (ISF) (2023-2027). ACAD: Grant(s) from the Israeli Academy of Sciences. BSF: United States-Israel Binational Science Foundation (several projects).  \\
This paper is number 331 in the author's publication list.}



% Last revision: Jan. 11, 2016

\date{April 21, 2026}

\begin{abstract}
    Our main aim is to prove that if $T$ is a complete first-order theory, which is not
    superstable (no knowledge on this notion is required), included in a first-order theory $T_1$ then for any $\lambda > |T_1|$ there are $2^\lambda$ models of $T_1$ such
    that for any two of them,  the $\tau(T)$-reducts of one is not 
    elementarily embeddable into the $\tau(T)$-reduct of the other, thus 
    completing the investigation of the $1978$ author's book ``\emph{Classification Theory and the Number of Non-Isomorphic Models}''.  Note the difference
    with the case of unstable $T$: there $\lambda \ge |T_1| + \aleph_{1}$ suffices for the existence of $2^{\lambda}$ pairwise non-isomorphic such models. 
    
    As earlier, it suffices for every such $\lambda$ to find a 
    complicated enough family of trees with $\omega + 1$ levels of cardinality
    $\lambda$.  If $\lambda$ is regular this is done already in Chapter VIII of the author's book.  The proof here (in sections 1 and 2) goes by dividing 
    into cases, each with its own combinatorics. In particular, we have
    to use guessing clubs which was discovered for this aim.
    
    In \S3 we improve the combinatorics, an aim is to consider strongly $\aleph_\epsilon $-saturated models of 
    stable $T$ (so if you do not know stability better just ignore this).
     We also deal with separable reduced Abelian $p$-groups.  We then deal
     with various improvements of the earlier combinatorial results.
\end{abstract}

\maketitle
\numberwithin{equation}{section}
\setcounter{section}{-1}
\newpage

\section{Introduction}

In \cite[Ch.VIII,\S2]{Sh:a} for unsuperstable (complete first order) theory $T$, it was proved that $\lambda > |T| + \aleph_1 \Rightarrow  \isoI(\lambda,T) = 2^\lambda$, in fact for every $T_{1}$ extending  $T,\lambda > \mu  := |T_{1}| + \aleph_0 \Rightarrow \isoI(\lambda,T_{1},T) = 2^\lambda$, and we have gone to considerable  troubles to prove it for all cases, in ZFC (recall that  $\isoI(\lambda,T) = \isoI(\lambda,T,T)$ where $\isoI(\lambda,T_1,T)$  is the number of $\tau(T)$-reducts of models of $T_1$ of cardinality $\lambda$ up to isomorphism, where $T\subseteq T_1$, and now both are first-order complete theories). $\numbIE(\lambda,T_1,T)$ is the maximal number of such models no one elementarily embedded into another (pedantically, the supremum); see Definition \cite[1.4=L1.4new]{Sh:E59}.

Now \cite[Ch.VIII,\S2]{Sh:a} gets results of the form  $\numbIE(\lambda,T_{1},T) = 2^\lambda $ under some constraints on $\lambda > |T_{1}|$  but have not tried to exhaust.  Later this
was put in a more general framework (see \cite{Sh:136} or \cite[\S2]{Sh:E59}) with several applications and more cases for  $\lambda > |T_{1}|$;  the cases left open were:
\mn
\begin{enumerate}
\item[$(\alpha)$]  $\lambda$ strong limit of cofinality $\aleph_{0}$
\end{enumerate}
\sn
and
\mn
\begin{enumerate}
\item[$(\beta)$]  $\lambda$ not strong limit, $\neg(\exists \chi)[\mu \le
\chi = \chi^{\aleph_{0}} < \lambda \le 2^\chi]$ (for example $\lambda <
2^{\aleph_{0}}$).
\end{enumerate}
\mn
Looking through \cite[Ch.VIII]{Sh:a} you may get the impression that the general case $(\lambda > |T_{1}|)$ is obviously true, just needs a
proof (as this holds in so many cases with diverse proofs).  Now in  addition to the accepted wisdom (at least among mathematicians) that such arguments are not proofs, there was until recently 
(i.e. before this was done in 1988) a reasonable argument for the other side: For most of the cases which were left open in \cite[Ch.VIII]{Sh:a}, their negations have been proved consistent 
(by \cite{Sh:100}, \cite{Sh:262}). However, here we prove this in all the cases.

Here we replace the properties from \cite[\S2]{Sh:E59} with stronger ones (variants of ``super unembeddable"), prove they imply the ones  from \cite[\S2]{Sh:E59}, look at their interrelations and mainly prove the existence of such families of trees for the various cardinalities. In \ref{7.5A}--\ref{7.7} we have the parallel of old theorems in the present frame see \ref{7.1} and \ref{7.2}; in \S2 new ones.  Lastly, in \ref{7.11} we draw the conclusions.

For this we prove in ZFC theorems of the form ``there is a club-guessing sequence'' (continued, see \cite[Ch.III]{Sh:e}, \cite{Sh:572} and more
current summary in \cite{Sh:E12}). Our main theorem is \ref{7.11}: for $\lambda > \mu, K^\omega_{\tr}$ has the full $(\lambda,\lambda,\mu,\aleph_0)$-super bigness property (defined in \ref{7.1}, \ref{7.2} below).  As a consequence, we here shall get that for $\lambda > \mu,K^\omega_{\tr}$ has the full strong
$(\lambda,\lambda,\mu,\aleph_0)$-bigness property (see definition in \cite[2.2(3)=f5(3)]{Sh:E59}, by \ref{7.12}(2)).

Lastly in \ref{7.11} we sum up what we get for $K^\omega_{\tr}$ for every
$\lambda>\mu$. The proof of \ref{7.11} is split into cases (each being an earlier claim) using several combinatorial ideas (in
some we get stronger combinatorial results than in others). 

The results are phrased such that they apply to many non-first-order classes. 

We conclude by deriving some further results dealing with some specific cases
in \S3. There, we begin by deducing the results on  $\numbIE(\lambda,T_1,T)$ being $2^\lambda$
when $\lambda>|T_1|,T\subseteq T_1$ are first order complete theories, $T$ unsuperstable  (in \ref{7.12}(1)). Then we get similar results for the number of strongly $\aleph_\varepsilon$-saturated models: the case which requires work is $T$ stable not superstable $\lambda = \lambda(T)+\aleph_1,T=T_1$ this requires some knowledge of stability  theory but is not used elsewhere; naturally this requires the so-called ${\bf F}^f_{\aleph_0}$-construction from \cite[Ch.~IV]{Sh:e}. We then deal with  the number of reduced separable abelian $\dot p$-groups on $\lambda$ no one
embeddable in another (not necessarily purely). We prove it by assuming $\lambda>2^{\aleph_0}$ (in \ref{7.15}) for this we need to improve the main
conclusion of \S2 (in \ref{7.14}).

In \S1 we, in a sense, redo results from \cite[\S2,VIII]{Sh:c} and 
\cite{Sh:136} restated in terms of super unembeddability 
in particular in \ref{7.6} (so we prove somewhat more). 

The results in \S2 were presented in a mini course in Rutgers, fall 88; it
contains ``guessing clubs in ZFC'', which because of the delay in 
publication was also represented and continued in
\cite[Ch.III]{Sh:g}, see more \cite{Sh:572}; printed version exists
since the early nineties. 

The results on the number of strongly $\aleph_\varepsilon$-saturated 
model \ref{7.13} improve Theorem \cite[2.1]{Sh:225} and
\cite[2.1]{Sh:225a} (see explanations below in \ref{7.13a}),
they assume knowledge of \cite{Sh:a} or \cite{Sh:c} but the reader can skip it as this theorem is not used later and move to
\ref{7.14}; some definitions are recalled in \ref{7.13a} below.  We
refer to \cite{Sh:E62} for various combination facts, see history
there (this will help if the book on non-structure will materialize).

We thank Haim Horowitz and Thilo Weinert for their help with the proofs, as well as Santiago Pinz\'on  and the referee for meticulously adding many helpful comments. 

\begin{convention}\label{z2}
1) $K_{\tr} = K^\omega_{\tr}$ (defined in \cite[1.9(4)=Lb11(4)]{Sh:E59} for $\kappa = \aleph_{0}$) is  restricted (in this
section) to the cases each $\Suc(\eta)$ is well ordered, so $K^\omega_{\tr}$ is the class of trees $I$ with $\omega +1$ levels expanded
by a well ordering on each $\Suc_I(\eta)$ (see \cite[1.9 = Lb11(2)]{Sh:E59}), so getting a lexicographic order $<_{lx}^{I}$ on $I$, and $P_{\alpha}^{I}$ for $\alpha \leq \omega$ is the set of elements of level $\alpha$ and if $\beta < \alpha \wedge \eta \in P^{I}$, then $\nu = \eta \rest \beta$ is the unique $\nu \in P_{\beta}^{I}$, which is an initial segment of $\eta$.

\noindent
2) Also, $\cM_{\mu, \kappa}(J)$ from \cite[2.1 = Lf2]{Sh:E59} and $\kappa$ regular for transparency. 

\noindent
3) Strong finitary, see Definition \cite[2.2 = Lf5]{Sh:E59}, that is: 

\begin{itemize}
    \item a member $a$ of $\cM_{\mu, \kappa}(J)$ is called \emph{strongly finitary} in $\cM_{\mu, \kappa}(J)$ if there are $n$, a $\tau_{\mu, \kappa}$-term $\sigma(x_{0}, \dots, x_{n-1})$ with finitely-many sub-terms using only functions with finite arity are $\tau_{0}, \dots, \tau_{n-1} \in J$ such that $\cM_{\mu, \kappa}(J) \models$``$a = \sigma(\tau_{0}, \dots, \tau_{n-1})$''. Recall $\tau_{\mu, \kappa}$ is the vocabulary of $\cM_{\mu, \kappa}(J)$. 
\end{itemize}

\begin{notation}\label{z5}
    1) Let $\cH(\chi)$ be the family of sets with transitive closure of cardinality $< \chi$ (so  of cardinality $2^{< \chi}$) and we let $<_{\chi}^{\ast}$ be a well-ordering of $\cH(\chi)$. 

    2) We use $M, N$ for models, $\tau_{M} = \tau(M)$ is the vocabulary of $M$, and $T$ for a theory in the vocabulary $\tau_{T} = \tau(T)$, complete if not said otherwise.  

    3) For a class $K$ of models and a cardinal $\lambda$, let $\dot{\bbI}(\lambda, K)$ be the number of models $M \in K$ up to isomorphisms. We may write $\dot{\bbI}(\lambda, T)$ for a theory $T$ (usually of first-order) instead of $K = \mathrm{Mod}_{T}$ and $\dot{\bbI}(\lambda, T_{1}, T)$ for theories $T_{1} \supseteq T$ (usually first-order complete), instead of $K = \PC(T_{1}, T) = \{ M \rest \tau_{T} \colon M \text{ is a model of } T_{1} \}$. 
\end{notation}

Recall,

\begin{definition}\label{z8}
    A complete first-order theory $T$ is called \emph{unsuperstable} \underline{when} there is some $\varphi$ witnessing it, which means: 

    \begin{enumerate}
        \item[(a)] $\varphi = \langle \varphi_{n}(x, \bar{y}_{n}) \colon n < \omega \rangle$, 

        \item[(b)] $\varphi_{n}(x, \bar{y}_{n}) \in \bbL(\tau_{T})$, i.e., it is a first-order formula in the vocabulary $\tau_{T}$ of $T$, and

        \item[(c)] for every $\lambda$ there are a model $M$ of $T$, sequences $\bar{a}_{\eta} \in {}^{\lg(\bar{y}_{m})} M$ for $\eta \in {}^{n} \lambda$ and elements $b_{\eta}$ of $M$ for $\eta \in {}^{\omega} \lambda$ such that if $\eta \in {}^{\omega} \lambda$, and $\nu \in {}^{n} \lambda$ then $M \models \varphi_{n}[b_{\eta}, \bar{a}_{\nu}] \Leftrightarrow \nu = \eta \rest n$. 
    \end{enumerate}
\end{definition}

% \noindent
% 2) Also $\cM_{\mu,\kappa}(J)$ from \cite[2.4(b)=L2.2(b)]{Sh:E59}.

% \noindent
% 3) Strong finitary, see Definition \cite[2.5(5)=L2.3(5)]{Sh:E59} or
%    here \ref{7.3}(A).
\end{convention}

\newpage

\section{Properties saying trees are complicated}

See also with \cite[\S1]{Sh:511}, \ref{7.5}. 

\begin{definition}
\label{7.1}
We say $I \in K^\omega_{\tr}$ is $(\mu,\kappa)$-super-unembeddable
into $J \in K^\omega_{\tr}$ \If \, : for every regular large enough  
$\chi^*$ (for which $\{I,J,\mu,\kappa\} \in \cH(\chi^*)$, recalling  \ref{z5}(1)), and for
simplicity $<^*_{\chi^*}$ is a well-ordering  of $\cH(\chi^*)$ and
$x \in {\cH}(\chi^*)$ we have:
\mn
\begin{enumerate}
\item[$(*)$]  there are $\eta, M_{n},N_{n}$ (for $n<\omega$) such that:
\sn
\begin{enumerate}
\item[$(i)$]  $M_{n} \prec N_{n} \prec M_{n+1} \prec ({\cH}(\chi^*),
\in,<^*_{\chi^*})$, 
\sn
\item[$(ii)$]  $M_{n} \cap \mu  = M_{0} \cap \mu$ and $\kappa \subseteq M_0$,
\sn
\item[$(iii)$]  $I,J,\mu,\kappa$ and $x$ belong to $M_0$,
\sn
\item[$(iv)$]  $\eta \in P^I_{\omega}$, i.e. is of level $\omega$ in $I$, 
\sn
\item[$(v)$]   for each $n$  for some $k,\eta \rest k \in M_n,
\eta \rest (k+1) \in N_n \backslash M_n$, 
\sn
\item[$(vi)$]  for each $\nu \in P^J_{\omega}$, for every large enough $n < \omega$, and
\[
\{\nu \rest \ell:\ell < \omega\} \cap N_n \subseteq M_n.
\]
\end{enumerate}
\end{enumerate}
\end{definition}

\begin{notation}
\label{7.1A}
We may write $\mu$ instead $(\mu,\mu)$ and may omit it if $\mu = \aleph_0$.
\end{notation}

\begin{remark} 
    \label{7.1f}
    \mn
    \begin{enumerate}
    \item In Definition \ref{7.1} the ``and $x \in \cH(\chi^{\ast})$'' and ``and $x$'' in clause (iii) can be omitted (and we get equivalent definition
    using a bigger $\chi^*$). However, in using the definition, with $x$ it is  more natural: we construct something from a sequence of $I$'s,  we would like to show that ``there are no objects such that...'' and 
    $x$ will be such an undesirable object in a proof by contradiction.
    
    \sn
    \item   We can also omit $<^*_{\chi^{\ast}}$ at the price of increasing $\chi^*$.
    \end{enumerate}
\end{remark}

Recall from \cite{Sh:E59}. 

\begin{definition}
    \label{7.2}
    \mn
    \begin{enumerate}
    \item   $K^\omega_{\tr}$ has the $(\chi,\lambda,\mu, \kappa)$-super-bigness property \If \,: there are $I_{\alpha} \in K^\omega_{\tr}$ of cardinality $\lambda$ for $\alpha < \chi$ such that for  $\alpha \ne \beta,I_{\alpha}$ is $(\mu,\kappa)$-super unembeddable into $I_{\beta}$.
    
    \sn
    \item $K^\omega_{\tr}$ has the full
    $(\chi,\lambda,\mu,\kappa)$-super-bigness property \If \,: \\
    there are $I_{\alpha} \in K^\omega_{\tr}$ of cardinality $\lambda$ for $\alpha <
    \chi$ such that for $\alpha < \chi,I_{\alpha}$ is $(\mu,\kappa)$-super unembeddable into $J_{\alpha} := \sum\limits_{\beta < \chi,\beta \ne \alpha} I_{\beta}$, which means, assuming for simplicity that the $I_{\beta}$-s are pairwise disjoint (except the common root $\mathrm{rt} = \mathrm{rt}_{I_{\beta}}$) that (recalling Definition \cite[1.9 = Lb11(4)]{Sh:E59} of $K_{\tr}^{\omega})$: 

     \begin{itemize}
         \item the set of elements of $J_{\alpha}$ is $\{ s \colon s \in I_{\beta}$ for some $\beta < \chi,$ $\beta \neq \alpha\}$,

         \item if $s_{\ell} \in I_{\beta_{\ell}}$ for $\ell = 1, 2$ (and $\beta_{1}, \beta_{2} \in \chi \setminus \{ \alpha \}$) \underline{then}:
         \[
         s_{1} <_{ex}^{J_{\alpha}} s_{2} \Leftrightarrow \beta_{1} < \beta_{2} \vee (\beta_{1} = \beta_{2} \wedge s_{1} <_{ex}^{I_{\beta_{1}}} s_{2}).
         \]
     \end{itemize}
       
    \sn
    \item   We may omit $\kappa$ if $\kappa = \aleph_0$.
    \end{enumerate}
\end{definition}

\bigskip
\centerline {$* \qquad * \qquad *$}
\bigskip

\noindent
The next definition gives many variants of Definition \ref{7.1} (clause (D)$^{+}$ repeat Definition \ref{7.1}); but the
reader may understand the rest of the section without it; just 
ignore \ref{7.3}, \ref{7.3A}, \ref{7.4}, \ref{7.5}(1); and from 
\ref{7.5} on, ignore the superscript to ``super" (we are getting 
stronger results).

\begin{definition}
    \label{7.3}
    We say $I \in K^\omega_{\tr}$ is
    $(\mu,\kappa)$-super$^\ell$-unembeddable into $J \in K^\omega_{\tr}$
    \If \, one of the following holds (and in this paper, $\ell$ is always one of those):
    
    \mn
    \begin{enumerate}
    \item[$(A)$]  $\ell = 1$ and for every regular large enough cardinal $\chi^*$
    and $x \in {\cH}(\chi^*)$ where $\{I,J,\mu,\kappa\} \in {\cH}(\chi^*)$
    and $f:I \rightarrow {\cM}_{\mu,\kappa}(J)$, which is strongly finitary on
    $P^I_{\omega}$ [i.e. for $\nu \in P^I_{\omega},f(\nu)$ is a strongly 
    finitary member of  ${\cM}_{\mu,\kappa}(J)$; see \ref{z2}(3)] and $g$ a function from $I$ 
    (really $P^I_{\omega}$) to finite sets of ordinals \underline{there is}
    $\eta \in P^I_{\omega}$ such that,

    \begin{enumerate}
        \item[$(\ast)$]  letting $f(\eta ) = \sigma(\nu_{0},
    \dotsc,\nu_{n-1})$, for infinitely many $k < \omega$ there are $M,N$  
    satisfying:
    \sn
    \begin{enumerate}
    \item[$(i)$]  $M \prec N \prec ({\cH}(\chi^*),\in,<^*_{\chi^*})$,
    \sn
    
    \item[$(ii)$] $M \cap \mu  = N \cap \mu,$ and $\kappa \subseteq M$,
    \sn
    \item[$(iii)$]  $\{I,J,\mu,\kappa,x\} \in M$,
    \sn
    \item[$(iv)$]  $\eta \rest k \in M$
    \sn
    \item[$(v)$]  $\eta \rest (k+1) \in N \setminus M$,
    \sn
    \item[$(vi)$]   for each $m < n$:
    \sn
    \item[${{}}$]  $(a) \quad \nu_{m} \in M$ or,
    \sn
    \item[${{}}$]  $(b) \quad$ for 
    some $k_{m} < \lg(\nu_{m}), \nu_{m} \rest k_m \in M,\nu_m(k_m) \notin N$ or,
    \sn
    \item[${{}}$]  $(c) \quad \ell g(\nu_{m}) = \omega,\nu_{m} \notin N,
    (\forall \ell < \omega) [\nu_{m} \rest \ell \in M]$,
    \sn
    \item[$(vii)$] $m < n$ and if $\alpha = \nu_m(k_m)$ (where $(vi)(b)$ holds for 
    $\nu_m,k_m$) or $\alpha \in g(\eta)$ then:
    \[
    \Min[(\chi^* \setminus \alpha) \cap M] =
    \Min[(\chi^* \setminus \alpha) \cap N].
    \]
    \end{enumerate}
    \sn
    \end{enumerate}

    \item[$(B)$]  $\ell = 2$ and for every regular large enough $\chi^*$
    satisfying $\{I,J,\mu,\kappa\} \in {\cH}(\chi^*)$ and $x \in
    {\cH}(\chi^*)$  there is $\eta \in P^I_{\omega}$ such that:

    \begin{enumerate}
        \item[$(\ast)$]  for any finite $w \subseteq \chi^*,n <
      \omega$ and $\nu_{0},...,\nu_{n-1} \in J$, for infinitely many 
    $k < \omega$ there are models $M,N$ satisfying:

    \begin{enumerate}
        \item[$(i)$]  $M \prec N \prec ({\cH}(\chi^*),\in,<^*_{\chi^*})$, 

        \item[$(ii)$]  $M \cap \mu = N \cap \mu, $ and $\kappa \subseteq M$,

        \item[$(iii)$] $ \{ I,J,\mu,\kappa,x\} \in M$,

        \item[$(iv)$] $\eta \rest k \in M$,

        \item[$(v)$]  $\eta \rest (k+1) \in N \backslash M$, 

        \item[$(vi)$] for each $m < n$ one of the following occurs:

        \begin{enumerate}
            \item[(a)]  $\nu_{m} \in M$,

            \item[(b)] for some $k_{m} < \lg(\nu_{n}), \nu_{m} \rest k_{m} \in M, \nu_{m} \rest (k_{m}+1) \notin N$,

            \item[(c)] $ \ell g(\nu_m) = \omega,\nu_m \notin N, (\forall \ell < \omega )[\nu_m \rest \ell \in M]$.
        \end{enumerate}

        \item[$(vii)$] for each $\alpha$, if $\alpha = \nu_m(k_m)$ (where $m < n, $ and $\nu_m, k_{m}$ satisfy sub-clause \noindent (b) of (vi)) or $\alpha \in w$ then\footnote{e.g. both can be undefined} $\Min[(\chi^{\ast} \setminus \alpha) \cap M] = \Min[(\chi^{\ast} \setminus \alpha) \cap N].$
    \end{enumerate}
    \end{enumerate}
    
    \sn
    \item[$(C)$]   $\ell = 3$ and for every regular large enough $\chi^*$,
    and $x \in {\cH}(\chi^*)$ such that $\{I,J,\mu,\kappa\} \in {\cH}(\chi^*)$,
    there is $\eta \in P^I_{\omega}$ satisfying:

    \begin{enumerate}
        \item[$(\ast)$]  for any $n < \omega, \nu_{0},\ldots,\nu_{n-1} \in J$, there are sequences
        $\langle M_{i},N_{i}:i < \omega \rangle, \langle \bar{k}^{i} \colon i < \omega \rangle, $ where $ \bar{k}^{i} = \langle k^i,k^i_{0},\ldots, k^i_{n-1}\rangle$. 
    such that:  
    
    \begin{enumerate}
        \item[$(i)$] $M_{i} \prec N_{i} \prec M_{i + 1} \prec (\cH(\chi^{\ast}), \in, <_{\chi^{\ast}}^{\ast})$, 

        \item[$(ii)$] $M_{i} \cap \mu = N_{i} \cap \mu$ and $\kappa \subseteq M_{0}$, 

        \item[$(iii)$] $\{ I, J, \mu, \kappa, x \} \in M$, 

        \item[$(iv)$] $\eta \rest k^{i} \in M_{i}$, 

        \item[$(v)$] $\eta \rest (k^{i} + 1) \in N_{i} \setminus M_{i}$,

        \item[$(vi)$] for each $i < \omega$ and $m < n$ one of the following occurs: 

        \begin{enumerate}
            \item[(a)] $\nu_{m} \in M_{i}$,

            \item[(b)] $k_{m}^{i} < \omega$, $\nu_{m} \rest k_{m}^{i} \in M_{i}$, $\nu_{m} \rest (k_{m}^{i} + 1) \notin N_{i}$, 

            \item[(c)] $\lg(\nu_{m}) = \omega$; $\nu_{m} \notin N_{i}$, $(\forall \ell < \omega)[ \nu_{m} \rest \ell \in M ]$, 
        \end{enumerate}
            \item[(viii)] for each $\alpha$, if $\alpha = \nu_{m}(k_{m})$ (where $m < n,$ $\nu_{m}$ satisfies (b) of (vi))  then: 
            $$ \Min[(\chi^{\ast} \setminus \alpha) \cap M] = \Min[(\chi^{\ast} \setminus \alpha) \cap N]. $$
    \end{enumerate}
    
    % $M_i,N_i,k^i,k^i_0,...,k^i_{n(i)-1}$ satisfy (i) --- (vii) of
    % \ref{7.3}(B) omitting ``or $\alpha \in w$" in clause (vii) 
    % $k^i \ge i$, with $k^i,k^i_{0},...,k^i_{n(i)-1}$ here standing for 
    % $k,k_0,...,k_{n-1}$ there and
    % \[
    % M_i \prec N_i \prec M_{i+1},M_i \cap \mu = N_i \cap \mu,\kappa
    % \subseteq M_0
    % \]
    % (we can assume $k^i_\ell \le k^{i+1}_{\ell}$)

    \end{enumerate}
    
    \sn
    \item[$(D)$] $\ell = 4$ and for every regular large enough $\chi^*$
    (for which $\{I,J,\mu,\kappa\} \in {\cH}(\chi^*))$ and $x \in 
    {\cH}(\chi^*)$ we have:
    \sn
    \item[${{}}$]  $(*) \quad$ there are $\eta,\filterD$ 
    and $M_n$ for $n < \omega$  such that:
    \sn
    \begin{enumerate}
    \item[${{}}$]   $(i) \quad M_{n} \prec M_{n+1} \prec
      ({\cH}(\chi^*),\in,<^*_{\chi^*})$,
    \sn
    \item[${{}}$]   $(ii) \quad M_{n} \cap \mu = M_{0} \cap \mu$ and $\kappa
    \subseteq M_0$,
    \sn
    \item[${{}}$]   $(iii) \quad \{ I,J,\mu,\kappa \}$ and $x$ belongs to $M_0$,
    \sn
    \item[${{}}$]   $(iv) \quad \eta \in I$, in fact $\eta \in
      P^I_{\omega}$,
    \sn
    \item[${{}}$]   $(v) \quad \filterD$ is a filter on $\omega$
     containing the filter of all co-finite sets (usually it is equal to
     it),
    \sn
    \item[${{}}$]   $(vi) \quad \{n < \omega$: for some $k,\eta \rest k
     \in M_n,\eta \rest (k+1) \in (M_{n+1} \setminus M_{n})\}$ belongs to 
    $\filterD$,
    \sn
    \item[${{}}$]   $(vii) \quad$ for every $\nu \in P^J_{\omega}$, we have\\
    $\{n$: for some $k < \omega,\nu \rest k \in M_n,\nu \rest (k+1) \in
    M_{n+1}\setminus M_{n}\}$\\
    is $\emptyset \mod \filterD$.
    \end{enumerate}
    \sn
    \item[$(D^-)$] $\ell =4^-$ and the condition of (D) holds just
    weakening (ii) to:
    \sn
    \begin{enumerate}
    \item[$(ii)'$]  $\{n:M_{n} \cap \mu = M_{n+1} \cap \mu\} \in \filterD$
    and $\kappa \subseteq M_0$,
    \end{enumerate}
    \sn
    
    \item[$(D^+)$]  $\ell = 4^+$ and the condition of Definition \ref{7.1} holds 
    (so $(\mu,\kappa)$-super$^{4^+}$-unembeddable will mean 
    $(\mu,\kappa)$-super-unembeddable).
    \sn
    
    \item[$(E)$]  $\ell = 5$  and for every regular large enough $\chi^*$
    (for which $\{I,J,\mu,\kappa\} \in {\cH}(\chi^*))$ and 
    $x \in {\cH}(\chi^*)$ we have:
    \sn
    \item[${{}}$]  $(*) \quad$ there are $\eta,\filterD,M_{n}$ (for $n < \omega$)
    such that:
    \sn
    \begin{enumerate}
    \item[${{}}$]   $(i) \quad M_{n} \prec M_{n+1} \prec ({\cH}(\chi^*),\in
    ,<^*_{\chi^*})$,
    \sn
    \item[${{}}$]  $(ii) \quad \mu \subseteq M_{n} \in M_{n+1}$ (so 
    $\kappa \subseteq M_0)$,
    \sn
    \item[${{}}$]  $(iii) \quad \{I,J,\mu,\kappa\}$ and $x$ belong to $M_0$,
    \sn
    \item[${{}}$]  $(iv) \quad \eta \in I$, in fact, $\eta \in
      P^I_{\omega}$,
    \sn
    \item[${{}}$]   $(v) \quad \filterD$ is a filter on $\omega$ 
    extending the filter of all co-finite sets (usually it is equal to
    it),
    \sn
    \item[${{}}$]  $(vi) \quad \{n < \omega$: for some $k,\eta \rest k 
    \in M_{n},\eta \rest (k+1) \in (M_{n+1} \setminus M_{n})\}$ belongs to  
    $\filterD$,
    \sn
    \item[${{}}$]  $(vii) \quad$ for every $\nu \in P^J_{\omega}$ we have
    $\{n$: for some $k < \omega,\nu \rest k \in M_{n},\nu \rest (k+1) \in
    M_{n+1} \setminus M_{n}\}$ is $\equiv \emptyset \mod \filterD$.
    \end{enumerate}
    \sn
    
    \item[$(F)$]   $\ell = 6$ and for every regular large enough $\chi^*$
    for which $\{I,J,\mu,\kappa\} \in {\cH}(\chi^*)$, and $x \in {\cH}(\chi^*)$
    there are $\langle M_{n}:n < \omega \rangle,\eta$ such that:
    \sn
    \begin{enumerate}
    \item[$(*)$]  $(i) \quad M_{n} \prec M_{n+1} \prec ({\cH}(\chi^*),\in
    ,<^*_{\chi^*})$,
    \sn
    \item[${{}}$]  $(ii) \quad M_{n} \cap \mu = M_{0} \cap \mu$ and $\kappa
    \subseteq M_0$,
    \sn
    \item[${{}}$]  $(iii) \quad \{I,J,\mu,\kappa,x\} \in M_0$,
    
    \sn
    \item[${{}}$]  $(iv) \quad \eta \in P^I_{\omega}$,
    \sn
    \item[${{}}$]  $(v) \quad \eta \rest n \in M_n$,
    \sn
    \item[${{}}$]  $(vi) \quad \eta \rest (n+1) \notin M_n$
    \sn
    \item[${{}}$]  $(vii) \quad$ for every $\nu \in P^J_{\omega}$, for some $n,
    \{\nu \rest \ell:\ell < \omega\} \cap (\bigcup\limits_{m < \omega} M_m)
    \subseteq M_n$.
    \end{enumerate}
    \sn
    
    \item[$(F^+)$] $\ell = 6^+$ and (i) - (vi) of (F) and: 
    \sn
    \begin{enumerate}
    \item[$(vii)^+$]  for every $\nu \in P^J_{\omega}$ we have:
    \[
    \left[\bigwedge\limits_{\ell} \nu \rest \ell \in 
    \bigcup\limits_{n < \omega} M_n \Rightarrow \nu \in
    \bigcup\limits_{n < \omega} M_n\right]\!.
    \]
    \end{enumerate}
    \sn
    \item[$(F^-)$]  $\ell=6^-$ and the conditions in (F) hold but replace clause
    (v) by
    \sn
    \begin{enumerate}
    
    \item[$(v)^-$]   $(\forall n)(\exists m)[\eta\restriction n \in M_m]$
    but $(\forall m)(\exists n) [\eta\restriction n\notin M_m]$.
    \end{enumerate}
    
    \sn
    \item[$(F^\pm)$] $\ell=6^\pm$, and the condition in (F) when we  make both changes (use $(vii)^{+}$ and $(v)^{-}$). 
    \sn
    \item[$(G^-)$]   $\ell = 7^-$ and for every regular large enough
      $\chi^*$ for which $\{I,J,\mu,\kappa\} \in {\cH}(\chi^*)$ and $x \in 
    {\cH}(\chi^*)$  we have:
    \sn
    \item[${{}}$]   $(*) \quad$ there are $M_{n}(n < \omega),\eta$ such that:
    \sn
    \begin{enumerate}
    \item[${{}}$]   $(i) \quad M_{n} \prec M_{n+1} \prec ({\cH}(\chi^*),\in
    ,<^*_{\chi^*})$, 
    \sn
    \item[${{}}$]  $(ii) \quad M_{n} \in M_{n+1}$ and $\mu \subseteq M_0$,
    \sn
    \item[${{}}$]  $(iii) \quad \{I,J,\mu,\kappa,x\} \in M_0$,
    \sn
    \item[${{}}$]  $(iv) \quad \eta \in P^I_{\omega}$,
    \sn
    \item[${{}}$]  $(v) \quad$ for every  $k < \omega,\eta \rest k \in
    \bigcup\limits_{n < \omega} M_n$,
    \sn
    \item[${{}}$]  $(vi) \quad$ for every $n$ for some $k,\eta \rest k 
    \notin M_n$,
    \sn
    \item[${{}}$]  $(vii) \quad$ for every $\nu \in P^J_{\omega}$, for some $n$
    \[
    \{\nu \rest \ell:\ell < \omega\} \cap (\bigcup\limits_{m < \omega} M_m)
    \subseteq M_n.
    \]
    \end{enumerate}
    \sn
    \item[$(G)$]  $\ell = 7$ and (i) - (iv),(vii) of $(G^-)$ and:
    \sn
    \begin{enumerate}
    \item[$(v)^+$]   $\eta\restriction n \in M_n$,
    \sn
    \item[$(vi)^+$]  $\eta \rest (n+1) \notin M_n$.
    \end{enumerate}
    
    \sn
    \item[$(G^+)$]  $\ell = 7^+$ and (i) -- (iv) of $(G^-)$ and, (v)$^+$, (vi)$^+$ of (G), and:
    \sn
    \begin{enumerate}
    \item[$(vii)^+$]  for every $\nu \in P^J_{\omega}$, 
    \[
    \{\nu \rest \ell:\ell < \omega\} \subseteq
    \bigcup\limits_{m < \omega} M_m \Rightarrow \nu
    \in \bigcup\limits_{m < \omega} M_m.
    \]
    \end{enumerate}
    \sn
    \item[$(G^\pm)$] $\ell=7^\pm$ and (i)-(iv) of (G)$^-$ of \ref{7.1} and (vii)$^+$ of 
    $(G^+)$ and 
    \sn
    \begin{enumerate}
    \item[$(vi)^{++}$]   for every $n$, for some $k$ we have $\eta
    \rest k \in M_n,\eta \rest (k+1)\in M_{n+1} \setminus M_n$ .
    \end{enumerate}
    \end{enumerate}
\end{definition}

\begin{definition}\label{7.3A}
    The $(\chi, \lambda, \mu, \kappa)$-super$^{l}$-bigness property and the full $(\chi, \lambda, \mu, \kappa)$-super$^{l}$-bigness property are defined in a similar way to Definition 1.4. 
\end{definition}

\begin{fact}\label{7.4}
    \mn
    \begin{enumerate}
    \item  If $I \in K^\omega_{\tr}$ is
    $(\mu,\kappa)$-super$^m$-unembeddable into $J \in K^\omega_{\tr}$
      \then \,  $I$ is $(\mu,\kappa)$-super$^\ell$-unembeddable into $J$ 
    when $1 \le \ell \le m \le 7,(\ell,m) \ne (5,6),\ell,m \in \{1,2,3,4,5,6,7\}$
    and when $(\ell,m) \in \{(3,4^-),(4^-,4),(4,4^+),(4^+,6),(6,6^+),
    (4^+,7^-)$,\\
    $(7^-,7),(7,7^+),(7^-,7^\pm),(7^\pm,7^+),(6^+,7^+),(6,7),
    (6^-,6^\pm),(6^\pm,6^+)\}$.
    \sn
     
    \item  If $K^\omega_{\tr}$ has the 
    $(\chi,\lambda,\mu,\kappa)$-super$^m$-bigness property \then \, 
    $K^\omega_{\tr}$ has the $(\chi,\lambda,\mu,\kappa)$-super$^\ell$-bigness
    property for  $(\ell,m)$  as above. 
    \sn
    
    \item    If $K^\omega_{\tr}$ has the full
    $(\chi,\lambda,\mu,\kappa)$-super$^m$-bigness property \then \,
    $K^\omega_{\tr}$ has the full
    $(\chi,\lambda,\mu,\kappa)$-super$^\ell $-bigness property for $(\ell,m)$
    as above. 
    \sn
    \item   All those properties has obvious monotonicity properties: we
    can decrease $\mu,\kappa$ and $\chi$ and increase $\lambda$ (if we add
    to $I$ a well-ordered set in level 1, nothing happens).
    \sn
    
    \item   The notions \\
    ``$(\mu,\kappa)$-super$^{4^+}$-unembeddable" and \\
    ``$(\mu,\kappa)$-super-unembeddable" are the same; also \\
    ``[full]$(\chi,\lambda,\mu,\kappa)$-super$^{4^+}$-bigness" and \\
    ``[full]$(\chi,\lambda,\mu,\kappa)$-super bigness" are the same.
    \end{enumerate}
\end{fact}

\begin{PROOF}{\ref{7.4}}
    The proof follows by the definitions. 
\end{PROOF}

% \begin{picture}(160,380)

% \put(40,345){Implication Diagram}
% \put(70,0){
% \begin{picture}(30,250)
% \put(5,0){1}
% \put(8,30){\vector(0,-1){20}}
% \put(5,35){2}
% \put(8,65){\vector(0,-1){20}}
% \put(5,75){3}
% \put(8,105){\vector(0,-1){20}}
% \put(5,110){$4^-$}
% \put(8,145){\vector(0,-1){20}}
% \put(5,155){4}
% \put(8,185){\vector(0,-1){20}}
% \put(5,195){$4^+$}
% \put(8,225){\vector(0,-1){20}}
% \put(5,235){6}
% \put(0,235){\vector(-1,-1){23}}
% \put(-80,235){6$^{\pm}$}
% \put(-68,230){\vector(1,-1){20}}

% \put(60,230){\vector(-1,-1){30}}

% \put(-40,205){6$^{-}$}

% \put(-35,205){\vector(1,-1){20}}

% \end{picture}}
% \put(20,245){
% \begin{picture}(105,100)
% \put(10,30){$6^+$}
% \put(25,25){\vector(1,-1){20}}
% \put(4,24){\vector(-1,-1){20}}
% \put(95,30){7}
% \put(85,20){\vector(-1,-1){20}}
% \put(55,70){$7^+$}
% \put(65,60){\vector(1,-1){20}}
% \put(45,60){\vector(-1,-1){20}}
% \end{picture}}
% \put(80,145){
% \begin{picture}(80,125)

% \put(50,50){\vector(-1,-1){31}}

% \put(60, 90){$7^-$}

% \put(62, 85){\vector(0, -1){15}}
% \put(60,55){5}
% \put(40,120){\vector(1,-1){15}}
% \end{picture}}
% \end{picture}
% \medskip

    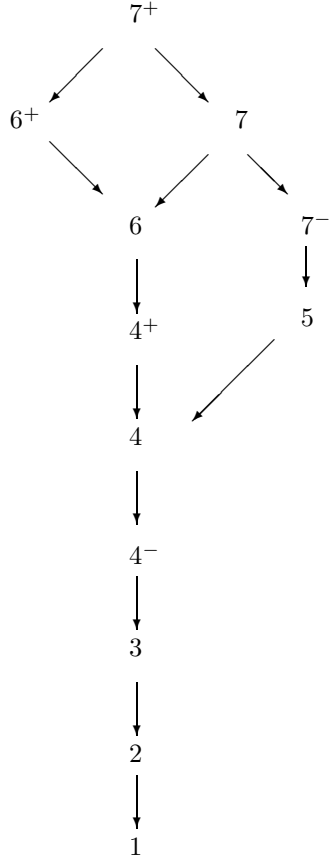
\begin{figure}
    \centering
    
    \begin{tikzpicture}
        \small{
        
        \node (1) at (3,1) {$1$};
        
        \node (2) at (3,2){$2$};
        
        \node (3) at (3,3){$3$};
        
        \node (4-) at (3,4) {$4^{-}$} ;

        \node (4) at (3, 5) {$4$};

        \node (4+) at (3, 6) {$4^{+}$};

        \node (6-) at (2,7) {$6-$};

        \node (6pm) at (1, 8) {$6^{\pm}$};
       
        \node (6) at (3, 8) {$6$};

        \node (7+) at (3, 10) {$7^{+}$};

        \node (6+) at (2, 9) {$6^{+}$};

        \node (7) at (4, 9) {$7$};

        \node (7pm) at (5, 10) {$7^{\pm}$};

        \node (7-) at (5, 8) {$7-$};

        \node (5) at (5, 7) {$5$};
        %%%%  diagram 1a
        
        \draw  (7+) edge [->] (6+)

        (6pm) edge [->] (6+)

        (7) edge [->] (6)

        (6+) edge [->] (6)

        (7+) edge [->] (7)
        
            (2) edge [->] (1)
        
              (3) edge[->] (2)
        
              (4-) edge [->] (3)
        
              (4)edge [->] (4-)
        
              (4+) edge [->] (4)

              (5) edge [->] (4)

              (7) edge [->] (7-)

              (7-) edge [->] (4+)

              (7-) edge [->] (5)

              (6) edge [->] (6-)

              (6) edge [->] (4+)

              (6pm) edge [->] (6-)

              (6-) edge [->] (4+)

              (7+) edge [->] (7pm)

              (7pm) edge [->] (7-);
        }
    \end{tikzpicture}
    
        \caption{Implication diagram.}   
    \label{f25}
\end{figure}

We shall now observe two things:

\begin{observation}
    1) First (\ref{7.5}(1)) the ``full" version (see Def. \ref{7.2},\ref{7.3A}) is much stronger (increasing the $\chi$) hence we shall later concentrate on it.

    2) Second (\ref{7.5}(2)), the super version (from here) implies the ``strong'' version from \cite[\S2]{Sh:E59}, hence, for example, implies the results on unsuperstable theories.
\end{observation}

\begin{claim}\label{7.5}
    1) If $K^\omega_{\tr}$ has full 
    $(\chi,\lambda,\mu,\kappa)$-super$^\ell$-bigness property, \then \,
    $K^\omega_{\tr}$ has the
    $(2^{\Min\{\lambda,\chi\}},\lambda,\mu,\kappa)$-super$^\ell$-bigness
    property.
    
    \noindent
    2) If $K^\omega_{\tr}$ has the [full] $(\chi,\lambda,\mu^{<\kappa},
    2^{<\kappa})$-super$^\ell $-bigness property,
    $(\chi \ge \lambda)$ \then \,  $K^\omega_{\tr}$ has the [full] strong
    $(\chi,\lambda,\mu,\kappa)$-bigness property for $\varphi_{\tr}$ for
    functions $f$ which are strongly finitary on $P_{\omega}$.
\end{claim}

\begin{remark}\label{7.5z}
    (1) On ``strongly finitary on $P_{\omega}$'' (see Definition \cite[2.2=Lf5]{Sh:E59} and \ref{z2}(3) here). 

    (2) On ``$(\chi, \lambda, \mu^{< \kappa}, 2^{< \kappa})$-super$^{\ell}$-bigness property, see Definition \ref{7.2}(1) with ``full'', see \ref{7.2}(2).

    (3) On the ``strong $(\chi, \lambda, \mu, \kappa)$-bigness property'', see \cite[2.2 = Lf5(1), (2), pg. 21]{Sh:E59} and for the full version, see \cite[2.2=Lf5(3)]{Sh:E59}.
\end{remark}

\begin{PROOF}{\ref{7.5}}
1)  Easy (and similar in essence to \cite[2.20=Lj11(1)]{Sh:E59}).
Suppose $\langle I_{\alpha}:\alpha < \chi \rangle$ witnesses the full
$(\chi,\lambda,\mu,\kappa)$-super$^\ell$-bigness property, and 
(by \ref{7.4}(4)) \wilog \, $\chi \le \lambda$.  \Wilog \, 
the $I_{\alpha}$'s have a common root  $< >$,  and except this are pairwise
disjoint.  We can find $\langle A_{\alpha}:\alpha < 2^\chi \rangle$ such
that:
\[
A_{\alpha} \subseteq \chi, |A_{\alpha}| \le \lambda \text{ and }
[\alpha \ne \beta \Rightarrow A_{\alpha} \nsubseteq A_{\beta}]
\]

\mn
(use just $A \subseteq \lambda$ such that $[2 \alpha \in A
 \Leftrightarrow 2 \alpha + 1 \notin  A])$.

Now, let $$I^*_{\alpha} \coloneqq  \sum\limits_{i \in A_\alpha} I_i,$$ defined naturally: the universe is the union of the universes, where each $I_i$ is a submodel of $I^*_\alpha$ when $i \in A_{\alpha}$, and  the lexicographic order is such that:
$$i < j \ \& \ \eta \in I_i \setminus \{< >\} \ \& \ \nu \in 
I_j \setminus \{< >\} \Rightarrow \eta <_{lx}\nu.$$

\mn
[Note:  $\chi > 2^\lambda$ never occurs].

\noindent
2) It suffices by \ref{7.4}(1), of course, to prove for the case  $\ell  = 1$,
and clearly, it suffices to show the following Subclaim.
\end{PROOF}

\begin{subc}\label{7.5A}
    If $\mu_{1} = \mu^{< \kappa}$ and $\kappa_{1} = \sum \{ (\vert \alpha \vert^{\aleph_{0}})^{+} \colon \alpha < \kappa \}$ and, if $I \in K^\omega_{\tr}$ is $(\mu_{1}, \kappa_{1})$-super$^1$-unembeddable into $J \in
    K^\omega_{\tr}$ \then \, $I$ is strongly $\varphi_{\tr}$-unembeddable for $(\mu,\kappa)$ into $J$, for functions $f$ which are strongly finitary on $P_\omega^I$ (see \cite[2.2=Lf5(1),~(4),~(5)]{Sh:E59}).
\end{subc}

%%%%skopiowac tez do VII claim 1.6%%%%%%%
\begin{PROOF}{\ref{7.5A}}
Recalling Definition \ref{7.3}(A), \wilog \, $I,J$ are subsets of ${}^{\omega \ge}\theta$ for  some cardinal $\theta$ (see \ref{z2} and \cite[1.9(2)=Lb11(2)]{Sh:E59}), and let $<^*$ be a well-ordering of ${\cM}_{\mu,\kappa}(J)$ (respecting being a subterm, i.e. if $a$ is a subterm of $b$ then $a \le^* b$). Suppose $f$ is a function from $I$ into ${\cM}_{\mu,\kappa}(J)$, so

\mn
\begin{enumerate}
    \item[$(*)_1$]  for $\varrho \in I$, we have $f(\varrho) = \sigma_{\varrho}(\nu_{\varrho,0},...,\nu_{\varrho,i},...)_{i<\alpha_\varrho}$ for some $\alpha_{\varrho} < \kappa,\nu_{\varrho,i} \in J$.
\end{enumerate}

\mn
Recalling that we are assuming $f$ is strongly finitary on $P_\omega^{I}$ we have,
\mn
\begin{enumerate}
    \item[$(*)_2$]  $\varrho \in P^I_{\omega} \Rightarrow \alpha_\varrho <
      \omega \, \& \,$ [$\sigma_\varrho$ has finitely many subterms].
\end{enumerate}
\mn
We should now find $\bar{s}$, $\bar{t}$ as in \cite[2.2=Lf5(1)]{Sh:E59}. Let $\chi$ be regular large enough, $x = \langle
\mu,\kappa,I,J,f\rangle$ and define for $\varrho \in P^I_{\omega}$,
\mn
\begin{enumerate}
\item[$(*)_3$]  $g(\varrho) = \{\alpha$: the $\alpha$-th element by $<^*$
is a subterm of $f(\varrho)\}$,
\end{enumerate}
\mn
which is finite (so we use ``the strongly finitary" so that $g(\eta)$ is finite, this is the only use). Now we shall  use Definition \ref{7.3}(A).

So let $\eta,k,M,N$ be as in (A) of Definition \ref{7.3} (for the given $\chi, f, g$ we use  just one $k$), hence noting  $(\ast)_{1}$, $(\ast)_{2}$ apply in particular for $\varrho= \eta$, and clearly  $\sigma_\eta(\nu_{\eta,0},\ldots,\nu_{\eta,i}, \ldots)_{i<\alpha_\eta}$ is well defined and equal to $f(n)$.  So by reordering $\nu_{\eta,\ell} (\ell < \alpha_{\eta})$ we can have: 

\begin{subf}\label{7.5D}
    There are $n_{0} \leq
    n_{1} \leq n_{2} = \alpha_{\eta} $ such that:

    \begin{enumerate}
\item[$(*)_4$]   $(a) \quad$ for  $\ell < n_0,\nu_{\eta,\ell} \in M$, 
and let $j_{\ell} = % % % 
\ell g
(\nu_{\eta, \ell})$, 
\sn

\item[${{}}$]  $(b) \quad$ for $\ell \in 
[n_{0},n_{1})$ for a (unique) $j_{\ell},
\nu_{\eta,\ell} \rest j_{\ell} \in M,\nu_{\eta,\ell}(j_{\ell})
\notin N$ and
\sn
\begin{equation*}
\begin{array}{clcr}
\gamma_{\ell} &:= \min\{\gamma:\gamma \text{ an ordinal from } M,
\nu_{\eta,\ell}(j_{\ell}) \le \gamma\} \\
  &= \min\{\gamma:\gamma \text{ an ordinal from } N,\nu_{\eta,\ell}
(j_{\ell}) \le \gamma\},
\end{array}
\end{equation*}
\item[${{}}$]  $(c) \quad$  for $\ell \in [n_{1},n_{2}), \nu_{\eta,\ell} \notin N$, $\lg(\nu_{\eta, \ell}) = \omega$ but $\{\nu_{\eta, \ell} \rest m:m < \omega\} \subseteq M$.
\end{enumerate}
\end{subf}

\begin{PROOF}{\ref{7.5D}}
    The sets $u_{3} \coloneqq \{ \ell < \alpha_{\eta} \colon \nu_{\eta, \ell} \notin N, \lg(\nu_{\eta, \ell}) = \omega \text{ and } \{ \nu_{\eta, \ell} \rest j \colon j < \omega \} \subseteq M \}$ and  $u_{1} \coloneqq \{ \ell < \alpha_{\eta} \colon \nu_{\eta, \ell} \in M\}$ are disjoint. So letting $u_{2} =  \alpha_{\eta} \setminus (u_{1} \cup u_{3})$, clearly $(u_{1}, u_{2}, u_{3})$ is a partition of $\alpha_{\eta}$, so remaining $u_{1} = [0, n_{0}],$ $u_{2} = [n_{0}, n_{1}]$, $u_{3} = [n_{1}, n_{2}]$, where $n_{2} = \alpha_{\eta}$ . So checking $(\ast)_{4}$, the least trivial  point is why the two definitions of $\gamma_{\ell}$ (for $\ell \in u_{2}$) are equivalent, which holds by clause (vii) of Definition \ref{7.3}(A). 
\end{PROOF}

Clearly $k$ was chosen together with $\eta,M,N$ and the 
sequence $\langle \nu_{\eta \rest (k+1),i}:i <
\alpha_{\eta \rest (k+1)}\rangle$ evidently belong to $N$ (as $f \in N$
and $\eta \rest (k+1)$ belongs to $N$).  

Now,

\begin{fact}\label{7.5E}
    We have:

    \begin{enumerate}
    \item[$(*)_5$]   for each $\ell \in [n_{1},n_{2})$ for some $k_{\ell} 
    < \omega$ we have:
    $\nu_{\eta,\ell} \rest k_{\ell} \notin A_1 := B_0 \cup B_1 \cup B_2$
    where:
    \begin{enumerate}
    \item[$\bullet$]  $B_0 = \{\nu_{\eta \rest (k+1),i} \rest m \colon i < \alpha_{\eta \rest (k + 1)}  , m \le \ell g(\nu_{\eta \rest (k+1),i}), i < \alpha_{\eta \rest
      (k+1)}\}$
    \sn
    \item[$\bullet$]  $B_1 = \{\nu_{\eta,j} \rest m:j < n_0,m \le 
    \ell g(\nu_{\eta,j})\}$
    \sn
    \item[$\bullet$]  $B_2 = \{(\nu_{\eta,\ell} \rest j_{\ell}) \char94
    \langle \gamma_j \rangle: \ell \in [n_0,n_1)\}$.
    \end{enumerate}
    \end{enumerate}
\end{fact}

\mn
[Recall $\gamma_{\ell}$ is from clause (b) of \ref{7.5D}.]

\begin{PROOF}{\ref{7.5E}}
    
\noindent
\underline{Case 1}:  $\kappa > \aleph_{0}$

So for each $\ell \in [n_{1},n_{2})$ if no such $k_{\ell}$ exists, so by clause (e) of \ref{7.5D}, we have $\{\nu_{\eta,\ell} \rest m \colon m < \omega\}$ is a subset of the set $A_1$ appearing
in the right side above, which belongs to $N$.

\noindent
[Why?  The first set $B_0$ in the union belongs to $N$ as $\eta \rest (k+1) \in N$ by the choice of $\eta, k, M, N$. The second set $B_1$ as $j < n_{0} \nRightarrow \nu_{\eta,j} \in M$ by clause (a) of \ref{7.5D}, i.e. the choice of $n_{0}$ and the third set $B_2$ by the choice of $\gamma_j$ in clause (b) of \ref{7.5D}.] 

Now $A_1$ has cardinality $<\kappa$ (as $\alpha_{\eta \rest (k+1)} <
\kappa \wedge \aleph_0 < \kappa)$; hence $A_1 \subseteq N$ and (as $ \kappa_{1} +1 \subseteq M$ because $\kappa_{1}$ in \ref{7.5A} plays the  role of $\kappa$ in Definition \ref{7.3A} recalling we are assuming $\kappa >\aleph_0$)  not only is included in it, but every $\omega$-sequence from it belongs to $N$, hence $\nu_{\eta,\ell} \in N$,  contradicting $\ell \in [n_{1},n_{2})$.  
\smallskip

\noindent
\underline{Case 2}:  $\kappa = \aleph_{0}$

So $\alpha_{\eta \rest (k+1)} < \omega$ and let $\ell \in
[n_{1},n_{2})$; toward contradiction assume the conclusion in \ref{7.5E} fails for $\ell$.  So one of the following
possibly occurs.  First, if $(\exists^\infty m) \nu_{\eta,\ell} \rest m
\in B_0$,  then for some  $i < \alpha_{\eta \rest (k+1)}$ for 
infinitely many $m < \omega,\nu_{\eta,\ell} \rest m = 
\nu_{\eta \rest (k+1),i} \rest m$.  As $\ell g(\nu_{\eta,\ell}) 
= \omega$ (remembering $\ell \in [n_{1},n_{2}))$ this implies 
$\nu_{\eta,\ell} = \nu_{\eta \rest (k+1),i}$, but 
$\nu_{\eta \rest (k+1),i}$ belongs to $N$ whereas
$\nu_{\eta,\ell}$ does not belong to $N$ (remembering $\ell \in
  [n_1,n_2)$), contradiction.  Second, if
$(\exists^\infty m)(\nu_{\eta,\ell} \rest m \in B_1)$, similarly some  
$j < n_0,\nu_{\eta,\ell} = \nu_{\eta,j}$ but $j < n_0 < \ell$,
 contradiction to clause (a) of \ref{7.5D}.  Third, if $(\exists^\infty n) (\nu_{\eta,\ell} \rest m \in B_2)$ but $B_2$ is finite,
 contradiction.  So \ref{7.5E} holds indeed.
\end{PROOF}

Note:

\begin{fact}\label{7.5F}
    We have:
    \begin{enumerate}
        \item[$(*)_6$]  $\sigma_{\eta \rest (k+1)}$ belongs to $M$.
    \end{enumerate}
\end{fact}

\begin{PROOF}{\ref{7.5F}}
    Why?  It belongs to $N$ (as $f,\eta \rest (k+1) \in N)$  and it belongs
    to a set of cardinality  $\mu_{1} := \mu^{<\kappa}$ from $M$ (the set of
    $\tau_{\mu,\kappa}$-terms) and $M \cap \mu^{<\kappa} = N \cap
    \mu^{<\kappa}$ by clause (ii) of Definition \ref{7.3}(A) as in
    subclaim \ref{7.5A} the cardinal $\mu^{<\kappa}$ play the role 
    of $\mu$ in \ref{7.3}(A)].
\end{PROOF}

\underline{Continuation of the proof of \ref{7.5}(2)}. Now recalling the definition of $\varphi_{\tr}$ (in \cite[2.6=Lg5, pg. 23]{Sh:E59}) and of unembeddable (in \cite[2.2(1)=Lf5(1)]{Sh:E59}) clearly, it is enough to show:
\mn
\begin{enumerate}
\item[$(*)_7$]   there is $\rho$ such that:
\sn
\begin{enumerate}
\item[$(A)$]   $\rho \in P^I_{k+1},\rho(k) \ne \eta(k),
\rho \rest k = \eta \rest k,\rho \in M$ and $\sigma_{\rho} = 
\sigma_{\eta \rest (k+1)}$ (so $\alpha_{\rho} = \alpha_{\eta \rest
  (k+1)})$,
\sn
\item[$(B)$]  the sequence $\langle \nu_{\rho,i}:i < \alpha_\rho \rangle$
is similar (i.e. realizes the same quantifier free type in $J$) to $\langle \nu_{\eta \rest (k+1),i}:i < \alpha_{\eta \rest (k+1)}\rangle$
over the set
\begin{equation*}
\begin{array}{clcr}
A_2 = \{\nu_{\eta,\ell}:\ell < n_{0}\} &\cup \{(\nu_{\eta,\ell}
\rest j_{\ell}) \char94 \langle \gamma_\ell \rangle: \ell \in 
[n_{0},n_{1})\} \\
  &\cup\{\nu_{\eta,\ell} \rest k_{\ell}:\ell \in [n_{1},n_{2})\}
\end{array}
\end{equation*}
\sn
\item[$(C)$] Let $A_{3,\rho} = A_{3,\eta \rest (k+1)}$ and 
  $A_{4,\rho,\eta} = A_{4,\eta \rest (k+1),\eta}$ 
\newline
where  for $\rho \in I$
\sn
\item[${{}}$]   $\bullet \quad A_{3,\rho} = 
\{(\sigma^1,\sigma^2):\sigma^1,\sigma^2$ subterms of 
$\sigma_\rho$

\hskip35pt  and $\sigma^1(\ldots,\nu_{\rho,i},\ldots) <^*
\sigma^2(\ldots,\nu_{\eta,\rho,i},\ldots)\}$
\newline
and
\sn
\item[${{}}$]  $\bullet \quad A_{4,\rho,\eta} = 
\{(\iota,\sigma^1,\sigma^3):\iota \in \{0,1\},\sigma^1 
\text{ subterm of } \sigma_{\rho},\sigma^3 \text{ a subterm of}$

\hskip35pt $\sigma_{\eta}, $ we have that 
\[\iota = 0 \Rightarrow 
\sigma^1(\ldots,\nu_{\eta \rest
  (k+1),i},\ldots) <^* \sigma^3(\ldots,\nu_{\eta,i},\ldots), \text{ and}\]
\hskip35pt  $\iota=1 \Rightarrow$ they are equal$\}.$
\sn
\item[$(*)_8$]  the sets $A_{3,\rho},A_{4,\rho,\eta}$ belongs to $M$.
\end{enumerate}
\end{enumerate}
\mn
\mn
[Why?  Like the proof of \ref{7.5F}, remembering \ref{7.5F}.

Now, the set $A_2$ is a finite subset of $M$ by clause (a) of \ref{7.5D}, the choice
of $k_\ell$ in clause (b) of \ref{7.5D} and the choice of $k_\ell$ in \ref{7.5E}.  Also
the ``similarly type in $J$" of $\langle \nu_{\eta \rest (k+1),i}:
i < \alpha_{\eta \rest (k+1)}\rangle$ over $A_2$ belongs to 
$M$ (in whatever reasonable way we represent it), as the set of 
such similarly types over $A_{2}$ is of cardinality $\le 2^\kappa$ and it belongs to $M$, hence there is a first-order formula $\psi(x)$ in the vocabulary 
of $({\cH}(\chi),\in,<^*_{\chi})$, with parameters from $M$
saying $x \in I$ is an immediate successor of $\eta \rest k,\sigma_{x}
= \sigma_{\eta \rest (k+1)}$, and $\langle \nu_{x,i}:i <
\alpha_x \rangle$ is similar to $\langle \nu_{\eta \rest (k+1),i}:i <
\alpha_{\eta \rest (k+1)}\rangle$ over $A_2$ in $J$ 
and $\psi$ express what clauses of $(C)$ from $(*)_7$ says (using the choice of $g$ and (vii) of \ref{7.3} clause (A)).  So there is a
solution to $\psi$ in $M$ (as $M \prec N \prec ({\cH}(\chi), \in,<^*_\chi))$,  now $\eta \rest (k+1)$ cannot be the $<_{\chi}^{\ast}$-first in  
$\{x:\psi(x)\}$,  but the first is in $M$, hence there is an  $x \in M$ such that $\psi(x) \, \& \, x <_{lx} \eta \rest (k+1)$.]

So we have finished.   
\end{PROOF}

\begin{remark}\label{7.5B}
    In \ref{7.5A}, if we weaken the conclusion ``$I$ is strongly $\varphi_{\tr}$-unembeddable..." to ``$I$ is $\varphi_{\tr}$-unembeddable"
    (see middle of Definition \cite[2.1(1)=Lf5(1)]{Sh:E59}) then we can weaken the demand 
    on $f$ replacing ``$f$ is strongly finitary for $\eta \in P_{\omega}^{I}$'' to ``$f(\eta)$ is finitary for $\eta\in P_\omega^I$".
\end{remark}

\begin{claim}\label{7.6}
1) If $\lambda$ is regular $> \mu$ \then \, 
$K^\omega_{\tr}$ has the full 
$(\lambda,\lambda,\mu,\mu)$-super$^{7^\pm}$-bigness property.

\noindent
2)  If $\lambda$ is singular $> \chi = \chi^{\kappa}$ and $2^\chi \ge
\lambda$ \then \,  $K^\omega_{\tr}$ has the full
$(\lambda,\lambda,\chi,
\aleph_0)$-super$^6$-bigness property (even the full 
$(2^\chi,\lambda,\chi,\kappa)$-super$^6$ bigness property) getting
$M_n$'s such that $(\forall \theta)[ \kappa^\theta= \kappa
\Rightarrow {}^\theta(M_n) \subseteq M_n]$; so if
$\kappa^{\aleph_0} = \kappa$ we actually have the full
$(\lambda,\lambda,\chi,\kappa)$-super$^{6^+}$-bigness property 
(and even the full $(2^\chi,\lambda,\chi,\kappa)$-super$^{6^+}$ 
bigness property).

\noindent 2A) If $\lambda$ is singular $> \chi = \chi^{\kappa}$ and $2^{\chi} \geq \lambda$, \underline{then} $K_{\tr}^{\omega}$ has the full $(\lambda, \lambda, \kappa, \kappa)$-supper$^{7}$-bigness-property, and if $\kappa > \kappa^{\aleph_{0}}$ then even the full $(\lambda, \lambda, \kappa, \kappa)$-supper$^{7^{+}}$-bigness-property.

\noindent
3) If $\lambda$ is strong limit singular of cofinality $>\kappa \ge
\aleph_0,\kappa \le \mu < \lambda$ \then \, $K^\omega_{\tr}$ has the
full $(\lambda,\lambda,\mu,\kappa)$-super$^6$-bigness property and even the
full $(\lambda, \lambda,\mu,\kappa)$-super$^6$-bigness property.

\noindent
4) We can, in part (3), weaken ``$\lambda$ strong limit" to
$(\forall \theta < \lambda )[\theta^\kappa < \lambda]$.
\end{claim}

\begin{remark}
    On part (1) see also \ref{7.8I}.
\end{remark}

\begin{PROOF}{\ref{7.6}}
1) Earlier relative is the proof of \cite[Ch.VIII 2.2]{Sh:a}), 
latter relatives is \ref{7.14}(1) case 1, (and see in \cite{Sh:511}) but we give it fully.

Let $S = \{\delta < \lambda:\cf (\delta) = \aleph_{0} \}$,  let $\langle S_{\zeta}:\zeta < \lambda \rangle$ be a sequence of pairwise disjoint stationary subsets of $S$.  For each $\zeta$ we can find $\bar C=\langle C_{\delta}:\delta \in S_{\zeta} \rangle$ such that:
\mn
\begin{enumerate}
\item[$(a)$]  $C_{\delta}$ is an unbounded subset of $\delta$
\sn
\item[$(b)$]  \otp$(C_{\delta}) = \omega$.
\end{enumerate}
\mn
For $\delta \in S_{\zeta}$ let $\eta_{\delta} \in {}^\omega \lambda$ be
defined by:
\[
\eta_{\delta} (n) \text{ is the (2n)-th member of } C_{\delta}.
\]

\mn
For $\zeta < \lambda$, let $I_{\zeta}  =  \{ \langle \zeta \rangle^\char94 \nu \colon \nu \in {}^{\omega >} \lambda \} \cup
\{ \langle \zeta \rangle^{\char94} \eta_{\delta}:\delta \in S_{\zeta}\} \cup \{ \langle \, \rangle \}$,  and we shall show that
$\langle I_{\zeta}:\zeta < \lambda \rangle$ exemplify the conclusion
(for super$^{7^\pm}$). For this, we use $\langle \zeta \rangle^{\char94} \nu$ to make the $I_{\zeta}$-s essentially pairwise disjoint. 

So let $\zeta(*) < \lambda, I := I_{\zeta(*)}$ and $J =:
\sum\limits_{\zeta \ne \zeta(*)} I_{\zeta}$ and we should prove that $I \in K_{\tr}^{\omega}$ is $(\mu, \kappa)$-super-$7^{+}$-embeddable into $J \in J_{\tr}^{\omega}$.   

Let $\chi^* = \chi(\ast)$ be regular large enough, $<^*_{\chi(\ast)}$ a well-ordering of
${\cH}(\chi^*)$ and $x \in {\cH}(\chi^*)$.  
We choose by induction on $\alpha < \lambda,M^*_{\alpha}$ such that:
\mn
\begin{enumerate}
\item[$\boxplus \ (\alpha)$]   $M^*_{\alpha} \prec ({\cH}(\chi),\in ,<^*_{\chi})$
 increasing and continuous with $\alpha$,
\sn
\item[$(\beta)$]  $\|M^*_{\alpha}\| < \lambda$,
\sn
\item[$(\gamma)$]  $M^*_{\alpha} \cap \lambda$ an ordinal,
\sn
\item[$(\delta)$]  $\langle M^*_{\beta}:\beta \le \alpha \rangle$
belongs to $M^*_{\alpha +1}$,
\sn
\item[$(\varepsilon)$]  $\mu \subseteq M^*_{0}$,
\sn
\item[$(\zeta)$]   $\{ \lambda,\mu,I,J,x,\langle \langle \eta_\delta: 
\delta\in S_\zeta\rangle:\zeta < \lambda\rangle,
\langle I_{\zeta}:\zeta < \lambda \rangle$, $\zeta(*) \}$ belongs to $M^*_0$.
\end{enumerate}
\mn

Let $E = \{\delta < \lambda:M^*_{\delta} \cap \lambda = \delta\}$,  it is a club of $\lambda$ (by clauses $(\alpha)$, $(\gamma)$, and $(\delta)$), so for some $\delta(*) \in S_{\zeta(*)}$ we have $\delta(*)\in \acc(E)$. Let $\langle m_\ell:\ell<\omega\rangle$ be a strictly increasing  sequence of natural numbers such that  letting $k_\ell$ be minimal satisfying $\eta_{\delta(*)} (k_\ell) \ge \lambda \cap M^*_{\eta_{\delta(*)}(m_\ell)}$, we have $k_\ell<m_{\ell+1}$  (or just $\eta_{\delta(*)} (k_\ell) \in M^\ast_{\eta_{\delta(*)}(m_{\ell+1})}$
which follows as $\alpha\subseteq M^*_\alpha$).

Let $M_n=M^*_{\eta_{\delta(*)}(m_n)}$ and $\eta=\eta_{\delta(*)}$.

Let us check the conditions in $(*)$ of Def.\ref{7.3}$(G^\pm)$ 
(see $(i)$-$(iv)$ of $(G)^-$, and $(v)^{+}$, $(vi)^{++}$ of (G)$^{\pm}$, and $(v)^{+}$, $(vii)^+$ from \ref{7.3}(G)$^{+}$) hold
for those $M_n$, $\eta$.

\noindent
\underline{Clause (i)}:  is obvious, as $\eta_{\delta(*)}(n)$ is strictly
increasing, and $M^*_\alpha$ is $\prec$-increasing with $\alpha$.
\medskip

\noindent
\underline{Clause (ii)}: Now $\mu+1\subseteq M_n$ as $M_n\cap \lambda$ is an ordinal (by clause $(\gamma)$ of $\boxplus$) and $\mu \in M_n$ (by clause $(\zeta)$ of $\boxplus$) and $\mu<\lambda$ by assumption. Also $M_n\in M_{n+1}$ by
clauses $(\gamma)+(\delta)$ of $\boxplus$.

% We are using clauses $(\varepsilon)$, $(\zeta)$ of $\boxplus$. 

\noindent
\underline{Clause (iii)}:  by clause $(\zeta)$ of $\boxplus$ above.
\medskip

\noindent
\underline{Clause (iv)}:  $\eta\in P^I_\omega$ as
$I=I_{\zeta(*)}$, $\delta(*) \in S_{\zeta(*)},\eta=\eta_{\delta(*)}$ and
our definitions.

\medskip

\noindent
\underline{Clause (vi)$^{++}$}:  By our choice of $k_\ell$ and of $M_{0}$.  

\noindent
\underline{Clauses (vii)$^+$}: Note: if $\nu \in P^J_{\omega},\alpha<\lambda$ and 
$\{\nu \rest \ell :\ell < \omega\} \subseteq M^*_{\alpha}$ then for some 
$\xi<\lambda$ and $\delta \in S_\xi$ we have $\xi \ne
\zeta(*),\nu = \langle\xi\rangle \char94 \eta_{\delta}$,  so
$\cf(\delta) = \aleph_{0}$ and $\delta = \sup(\delta \cap
M^*_{\alpha})$ but $\langle S_\xi: \xi<\lambda\rangle$ are
pairwise disjoint so for every $\alpha,\delta < \lambda$ we have 
at most one such $\nu$, so $\{\nu \in P^J_{\omega}:
\bigwedge\limits_{\ell < \omega} \nu \rest \ell \in M^*_{\alpha}\}$  
has cardinality $\le \|M^*_{\alpha}\|$ hence is a subset of 
$M^*_{\alpha +1}$ (as $M^*_{\alpha},J \in M^*_{\alpha +1})$.
Moreover $\delta\in  S_\xi$.
\medskip

To prove Clause (vii)$^+$ we assume $\nu\in P^J_\omega$; we should prove that $\{\nu\rest\ell:\ell<\omega\} \subseteq \bigcup\limits_{m<\omega}  M_m \Rightarrow \nu \in \bigcup\limits_{n<\omega} M_n$,
but this union is equal to $M^*_{\delta(*)}$. 
So using $\alpha=\delta(*)$ above we have 
$(\xi,\delta)$ as there and one of the following cases occurs, and it 
suffice to check the implication in each of them.
\medskip

\noindent
\underline{Case 1}:  $\xi< \delta(*) \, \& \, \delta< \delta(*)$

Clearly $\Rang(\nu) \subseteq ((\xi+1)\cup \delta)$ and so $\nu\in
M^*_{((\xi+1)\cup\delta)+1}$ hence $\nu\in
M^*_{\delta(*)}=\bigcup\limits_{n<\omega} M_n$.
\medskip

\noindent
\underline{Case 2}:  $\xi \ge \delta(*)$

So even $\nu\restriction 1 \notin M_{\delta(*)}$ hence
\[
(\exists k<\omega)(\nu(k)\notin \bigcup\limits_{n<\omega} M_n \, \& \, 
\nu \restriction k\in \bigcup_{n<\omega} M_n)
\]

\noindent
\underline{Case 3}:  $\xi < \delta(*) \le \delta$

So as $\delta(*) \in S_{\zeta(*)} $ necessarily $\delta \notin S_{\zeta(*)}$, hence $\delta(*)< \delta$. Recall that $\nu = \eta_{\delta}$, so for some $k$ we have $(\forall \ell< k) \eta_\delta(\ell)< \delta(*)$ and $\eta_\delta(k)\geq \delta(*)$. So $\eta_\delta\restriction k \in
M^*_{\delta(*)}$, and $\eta_\delta(k) \ge \delta(*)$,  hence as $M^*_{\delta(*)} \cap \lambda = \delta(*)$ because $\delta(*) \in E$  we have $\eta_\delta\rest (k+1)\notin M^*_{\delta(*)}$, but $M^*_{\delta(*)}=\cup\{M_n:n<\omega\}$, so we are done.

% As this holds for any $\nu\in P^J_\omega$ we are done.

%This take care of a weak version of (vii). The full version holds
%(as by (d),  $\alpha \in  \rang (\eta) \Rightarrow  \cf (\alpha ) >
%\aleph_{0})$.
%
%[Saharon]
%%%%%%%%%%%%

\noindent
2) Compare with the earlier version \cite[2.7,pg.116,\S3]{Sh:136},  it is easier than the proof of part (3).

We are assuming $\chi = \chi^{\kappa}$, now 
there are subsets $A_{i}$ of $\chi$ for $i < 2^\chi$ \, ($i < \lambda$ 
is enough) such that (see \cite[3.13=L4.EK]{Sh:E62}, i.e. by 
Engelking-Karlowicz \cite{EK}):
\mn
\begin{enumerate}
\item[$(*)$]   if $w \subseteq \chi$ has cardinality $\le \kappa$ and  
$i \in \chi \setminus w$ then $A_{i} \nsubseteq 
\bigcup\limits_{j \in w} A_{j}$.
\end{enumerate}
\mn
Let $S^\zeta \subseteq \{\delta < \chi^+:\cf(\delta) = \aleph_{0}\}$  be
stationary pairwise disjoint for  $\zeta < \chi$.  For $i < 2^\chi$ let
$S_{i} = \bigcup\limits_{\zeta \in A_i} S^\zeta$ and
\begin{multline*}
    I_{i} = \{ \langle \, \rangle \} \cup \{ \langle i \rangle^{\char94} \nu \colon \nu \in {}^{\omega >}\lambda \} \cup \{ \langle i \rangle^{\char94} \eta \colon \eta \in {}^\omega(\chi^+) \text{ and }\\ \eta \text{ strictly increasing with limit } \sup\limits_{n < \omega} \eta(n) \in S_{i} \}.
\end{multline*}

We shall show that $\langle I_{i}:i < 2^\chi \rangle$ is as required, so for $j < 2^\chi$ let $J_j = \sum\limits_{i < 2^\chi,i \ne j} I_i$ and $\chi^*$ large enough, $x \in {\cH}(\chi^*)$, so for the rest of the proof of \ref{7.6}(2), 

\begin{enumerate}
    \item[$(\ast)$] we fix $i_{\ast} = i(\ast) < 2^{\chi}$ and will prove $I_{i(\ast)}$ is $(\kappa, \aleph_{0})$-super$^{6^{+}}$-unembeddable into $J_{i(\ast)}$.   
\end{enumerate}

By \cite[1.17(2)=La48(2), pg. 10]{Sh:E62} with $\chi^{+}$ here standing for $\lambda$ there, we can find a sequence $\langle N_{\eta}:\eta \in {\cT}
\rangle$, such that:
\begin{enumerate}
    \item[$\boxplus_{1} (a)$]  $\langle \rangle \in {\cT} \subseteq {}^{\omega >} (\chi^+)$,
    \sn
    \item[$(b)$]  $\nu \triangleleft \eta \in {\cT} \Rightarrow \nu \in {\cT}$,
    \sn
    \item[$(c)$]  $(\forall \eta \in {\cT})(\exists^{\chi^+} \alpha < 
    \chi^+)[\eta \char94 \langle \alpha \rangle \in {\cT}]$,
    \sn
    \item[$(d)$]  $N_{\eta} \prec ({\cH}(\chi^*),\in,<^*_{\chi}),
    \|N_\eta\|=\kappa,\kappa \subseteq N_\eta$,
    \sn
    \item[$(e)$]  $N_{\eta} \cap N_{\nu} = N_{\eta \cap \nu}$, where  $\eta \cap \nu$ is the maximal
      $\rho$ such that $\rho \trianglelefteq \eta \wedge \rho
      \triangleleft \nu$,
    \sn
    \item[$(f)$]  $\eta \in N_{\eta}$, 
    \sn
    \item[$(g)$]  $(\forall \theta)[\kappa = \kappa^\theta 
    \Rightarrow {}^\theta(N_\eta) \subseteq N_\eta]$,
    
    \sn 
    \item[$(h)$]  $\{i_{\ast}, \kappa, \chi, x, \langle S_{i} \colon i < 2^{\chi} \rangle , \langle S^{\varepsilon} \colon \varepsilon < \chi^{+} \rangle, \langle (A_{i}, J_{i}) \colon i < 2^\chi \rangle \} $ belongs to $N_{<>}$,
    
    \sn 
    \item[$(i)$] $N_{\eta} \cap \chi = N_{\langle \rangle} \cap \chi$. 
\end{enumerate}

Clearly $(2^{\chi} \setminus \{ i_{\ast} \}) \cap N_{\langle \rangle}$ has cardinality  $\leq \kappa$.

For each $\eta \in \lim ({\cT}) := \{\eta \in {}^\omega(\chi^+)$: if $n < \omega$ then $\eta \rest n \in \cT\}$, clearly $N_{\eta} :=  \bigcup\limits_{{\ell} < {\omega}} N_{\eta \rest \ell}$  has cardinality $\kappa$ so there is $\epsilon_\eta \in A_{i(\ast)} \setminus \bigcup \{ A_{j} \colon j \in 2^{\chi} \cap N_{\langle \rangle}$ but $j \neq i_{\ast} \}$ $(\subseteq \chi)$.  For each $\epsilon  < \chi$, let $Y_{\epsilon}$ be the set of $\eta \in \lim(\cT)$ such that: 
\begin{itemize}
    \item $\eta$ is strictly increasing with limit $\varepsilon$, and

    \item $\epsilon \in A_{i(\ast)}
\setminus \bigcup\{A_{j}: j \in N_{\eta} \text{ and } j \ne i(\ast)\}$.
\end{itemize}

Clearly $Y_\varepsilon$ is a closed subset of $\lim({\cT})$, (for the topology where $\cU \subseteq \lim(\cT)$ is open \underline{iff} for every $\eta \in \cT$, there is $n < \omega$ such that $(\forall \nu)(\eta \rest \nu \lhd \nu \in  \lim(\cT) \Rightarrow \nu \in \cU)$). Now those $\chi$ closed sets $\langle Y_\epsilon \colon \epsilon< \chi \rangle$ cover $\lim({\cT})$ as $\eta \in \lim(\cT) \Rightarrow \varepsilon_\eta$ well defined by a previous paragraph,  
so possibly shrinking $\cT$, by \cite[1.17(1)=La48(1), pg.~10]{Sh:E62}, without loss of generality $\epsilon_{\eta}  = \epsilon(\ast)$ for every $\eta \in \lim({\cT})$.  

Now the set
\[
C = \{\delta < \chi^+:\nu \in {}^{\omega >}\delta \Rightarrow
\sup(N_{\nu} \cap \chi^+) < \delta\}
\]
is a club of $\chi^+$ and we can find $\rho \in \lim({\cT})$, strictly increasing with limit $\delta \in C \cap S^{\epsilon(\ast)}$. 
Now let $M_n = N_{\rho \restriction n}$ and choose by induction on $n$ an ordinal $\alpha_n \in (M_{n+1} \setminus M_n) \cap \chi^+$ and let $\eta = \langle \alpha_n: n<\omega\rangle$, and we can prove as in the
proof of part (1) of \ref{7.6} that it is as required.

\noindent
2A) Until $\boxplus_{1}$, we repeat the proof of part (2), but replace $\boxplus_{1}$ by:

\begin{enumerate}
    \item[$\boxplus_{2}$] there is $\bar{N}^{\ast} = \langle N_{\eta}^{\ast} \colon \eta \in \cT_{\ast} \rangle$ satisfying clauses (a)-(d), (f)-(h) of $\boxplus_{1}$ and: 

    \begin{enumerate}
        \item[(e$'$)] if $\nu \lhd \eta \in \cT_{\ast}$, \underline{then} $N_{\nu} \in N_{\eta}$ (so $N_{\nu} \prec N_{\eta}$).  
    \end{enumerate}
\end{enumerate}

As we waive $\boxplus_{1}$(e), there is no problem doing it. As above, we can find

\begin{enumerate}
    \item[$\boxplus_{3}$] there is $\bar{N} = \langle N_{\eta} \colon \eta \in \cT \rangle$ satisfying $\boxplus_{1}$ such that $\cT \subseteq \cT_{\ast}$ and $\bar{N}^{\ast} \in N_{\eta}$ for $\eta \in \cT$.   
\end{enumerate}

Clearly, 

\begin{enumerate}
    \item[$\boxplus_{a}$] if $\eta \in \cT$, then $\eta \in N_{\eta}^{\ast} \prec N_{\eta}$.  
\end{enumerate}

Continuing the proof of part (2), getting $\rho, \delta$, we let $M_{\eta} = N_{\rho \rest n}^{\ast}$, choose $\alpha_{n} \in M_{n + 1} \cap \chi^{+}$, $\sup(M_{\eta} \cap \chi^{+})$, $\eta = \langle \alpha_{n} \colon n < \omega \rangle$, we get the desired conclusion for the $(\lambda, \lambda, \kappa, \kappa)$-supper$^{7}$-stability-property. The ``super$^{7^{+}}$'' variant is similar.  

\noindent
3)-4) Choose an increasing continuous sequence $\langle \lambda_{i}:i <
\cf(\lambda)\rangle$ of cardinals, such that:
\begin{enumerate}
\item[$(a)$]  $\lambda = \sum\limits_{i < \cf(\lambda)} \lambda_{i}$
\sn
\item[$(b)$] $i$ non-limit $\Rightarrow \lambda_{i} = \mu^+_{i} \ \&
  \, \mu^\kappa_{i} = \mu_{i}$ for some $\mu_{i} > \mu$,
\sn
\item[$(c)$]   $\lambda_{0} > \mu^\kappa + \cf(\lambda)$.
\end{enumerate}
\mn
Choose further for any
\[
\delta \in S =: \{i:i < \cf(\lambda) \text{ and } \cf(i) = \aleph_{0}\}
\]

a sequence $\langle \lambda_{\delta,n}:n < \omega \rangle$ such that:
\[
\lambda_{\delta,n} \in \{\lambda_{j+1}:j < \delta\} \text{  and }
\lambda_{\delta,n} < \lambda_{\delta,n+1} \text{ and }
\lambda_{\delta} = \sum\limits_{n<\omega} \lambda_{\delta,n}.
\]

\mn
For $\delta \in S,$ let $\gs^0_{\delta}$ be the family of sequences $\bar N =
\langle N_{\eta}:\eta \in \cT \rangle$ satisfying:
\mn
\begin{enumerate}
\item[$(A)$]  ${\cT}$ is a subset of $\bigcup\limits_{n < \omega} \, \prod\limits_{\ell < n} \lambda_{\delta,\ell}$, closed under
initial segments, $\langle \rangle \in {\cT},[\eta \in {\cT} \, \& \, 
\ell g(\eta) = n \Rightarrow (\exists^{\lambda_{\delta,n}} 
\alpha)(\eta \char94 \langle \alpha \rangle \in {\cT})]$.
\sn
\item[$(B)$]  for some countable vocabulary $\tau=\tau_{\bar N}$ where   $<_*$ belongs to $\tau$, each $N_{\eta}$ is a $\tau$-model with universe a bounded subset of 
$\lambda_{\delta}$ of 
cardinality $\kappa, \kappa +1 \subseteq N_{\eta},
\{\lambda_{\delta,n}:n<\omega\}\subseteq N_\eta, <_*^{N_{\eta}} =  < \rest N_{\eta},$  $N_{\eta \rest k} \prec N_{\eta}$ and $N_{\eta} \cap
N_{\nu} \prec N_{\eta}$ and\footnote{if $\eta\in {\cT} \Rightarrow N_\eta \prec {\gC}$ and ${\gC}$ has Skolem 
functions then $N_\eta \cap N_\nu \prec N_\eta$ follows, we can add 
$\kappa^\theta=\kappa\Rightarrow [N_\eta]^{\le \theta}\subseteq N_\eta$}
$(\forall \ell < \ell g(\ell))[ \eta(\ell) \in N_{\eta}$ and $\eta \in \lim(\cT) \Rightarrow \sup \{ N_{\eta \rest n} \colon n < \omega \} = \delta]$.
\end{enumerate}

For ${\cT}$ as in clause (A), recall
\begin{equation*}
    \begin{array}{clcr}
    \lim({\cT}) = \{\eta:&\eta \text{ an } \omega\text{-sequence such that 
    every proper initial} \\
     &\text{ segment of } \eta \text{ is in } \cT \}.
    \end{array}
\end{equation*}

For a given $\bar N \in {\gs}^0_\delta$, and $\eta \in \lim({\cT})$ 
 we use freely $N_{\eta}$ as $\bigcup\limits_{\ell < \omega} 
N_{\eta \rest \ell}$, (clearly still $N_{\eta}$ is a $\tau$-model of
cardinality $\kappa$ with universe $\subseteq \lambda_{\delta},
\kappa + 1 \subseteq N_{\eta}$ and $N_{\eta \rest \ell} \prec N_{\eta})$.

Recall that $\eta \cap \nu$ is the largest common initial segment of $\eta$ and $\nu$.

Let $\gs^1_{\delta,\mu}$ be the family of $\bar N = \langle N_{\eta}:\eta
\in {\cT}\rangle$ satisfying (in addition to being in
$\gs^0_{\delta}$):
\mn
\begin{enumerate} 
\item[$(C)$]  $(i) \quad$ if $\eta,\nu \in \lim({\cT})$ 
then\footnote{this simplifies the clause before last in $(B)$ above}
$N_{\eta} \cap N_{\nu} = N_{\eta \cap \nu}$
\sn
\item[${{}}$]  $(ii) \quad$ if $\eta,\nu \in {\cT}$ and 
$\Rang(\eta) \subseteq N_{\nu}$ then $\eta\trianglelefteq \nu$
\sn
\item[${{}}$]  $(iii) \quad N_{\eta} \cap \mu = N_{<>} \cap \mu$.
\end{enumerate}

\noindent
Before finishing the proof of \ref{7.6}, we prove the following subfacts:
\begin{subf}
\label{7.6A}
(Recall $\lambda$ be strong limit singular, $\cf(\lambda)> \kappa$.)
Suppose $M^*$ is a model with countable vocabulary and universe 
$\lambda$ and $<_*^{M^*} = < \rest \lambda$.  \Then \, for some club $C$
of $\cf(\lambda)$, for every $\delta \in S \cap C$ we have:
\mn
\begin{enumerate}
\item[$(*)^\delta_{1}$]   for some $\langle N_{\eta}:\eta \in
{\cT} \rangle \in \gs^0_{\delta}$ we have:
\[
\text{ for every } \eta \in \cT, \quad N_{\eta} \prec M^*.
\]
\end{enumerate}
\end{subf}

\begin{PROOF}{\ref{7.6A}}  
Define a function $f$ from $^{\omega >}\lambda$ to 
$\{A:A \subseteq \lambda,|A| = \kappa < \cf(\lambda)\}$ by:  $f(\eta)$ is the
(universe of the) Skolem Hull of ($\Rang(\eta)) \cup \{i:i \le \kappa\}\cup
\{\langle \lambda_i:i < \cf(\lambda)\rangle,\langle \lambda_{\delta,n}:\delta
\in S,n < \omega\rangle\}$ in $({\cH}(\chi^*),\in,<_{\chi^*})$.
Now apply \cite[1.23=L1.17]{Sh:E62}.

\begin{subf}
\label{7.6B}
In \ref{7.6A} we can strengthen $(*)^\delta_{1}$ to:
\mn
\begin{enumerate}
\item[$(*)^\delta_{2}$]  for some $\langle N_{\eta}:\eta \in \cT \rangle \in
\gs^1_{\delta,\mu}$ we have:
\[
\text{ for every } \eta \in {\cT},N_{\eta} \prec M^*.
\]
\end{enumerate}
\end{subf}

\begin{PROOF}{\ref{7.6B}}
Let $\langle N_{\eta}:\eta \in {\cT}\rangle$ be a member of
$\gs^0_{\delta}$ satisfying $\eta \in \cT$ implies $N_{\eta} \prec M^*$.

We now will apply \cite[1.19=La54]{Sh:E62} with $|N_\eta|$ here 
standing for $A_n$ there.

So there is ${\cT}' \subseteq {\cT}$ such that $\langle N_{\eta}:\eta \in 
\cT'\rangle$  is a $\Delta$-system; i.e. 
\mn
\begin{enumerate}
\item[$(i)$]  ${\cT}' \subseteq {\cT}$ satisfies (A)
\sn
\item[$(ii)$]  there is a function ${\bf h}$ with domain ${\cT}' \times \omega
\times \omega$ such that for all incomparable  $\eta,\nu \in {\cT}'$ we
have:
\[
N_{\eta} \cap N_{\nu} = {\bf h}(\eta \cap \nu,\ell g(\eta),\ell g(\nu)).
\]
\end{enumerate}
\mn
Let
\[
{\bf h}^+(\eta) := \bigcup\limits_{n,m > \ell g(\eta)}{\bf h}(\eta,n,m) 
\]
so $\mathbf h^+(\eta)$ is a subset of $\lambda_\delta$ of cardinality
$\kappa$; as $[\eta \triangleleft \nu \Rightarrow N_{\eta} \prec
  N_{\nu}]$ clearly:
\mn
\begin{enumerate}
\item[$(*)$]  if $\eta \ne \nu \in \lim (\cT')$ then
$\mathbf h^+(\eta \cap \nu) = N_{\eta} \cap N_\nu$.
\end{enumerate}
\mn
As $M^*$ has definable Skolem functions, if $\eta,\nu \in
\lim({\cT}')$ then
\[
M_{\eta \cap \nu} := N_{\nu} \rest \mathbf h^+(\eta \cap \nu) =
N_{\eta} \rest \mathbf h^+(\eta \cap \nu)
\]

\mn
is an elementary submodel of $N_{\eta},N_{\nu}$ 
(remember:  $<^{N_{\eta}} = < \rest N_{\eta}$ is a well ordering).  So
it is easy to check $\langle M_{\eta}:\eta \in {\cT}'\rangle$
is almost as required.  The missing point is $M_{\eta} \cap \mu = M_{<>} \cap
\mu$ for every $\eta \in \cT'$.  As $\langle N_{\eta \char 94 \langle
  \alpha \rangle}:\eta \char 94 \langle \alpha \rangle \in
\cT'\rangle$ are pairwise disjoint and $\lambda_{\delta,\ell g(\eta)}
> \mu$ for some $\alpha_\eta < \lambda_{\delta,\ell g(\eta)}$ we have
$\eta \char 94 \langle \alpha_{\eta} \rangle \in \cT' \Rightarrow N_{\eta \char
  94 \langle \alpha_{\eta} \rangle} \cap \mu = N_\eta \cap M$.  So
by throwing away enough members of ${\cT}'$  (i.e. we choose $\{\nu\in
{\cT}':\ell g(\nu) = n\}$ by induction on $n$) we can manage. 
\end{PROOF}

\begin{subf}
\label{7.6C}
We can find $\langle \eta^{\delta,\alpha},\langle M^{\delta,\alpha}_{n}:n <
\omega \rangle:\delta \in S,\alpha < 2^{\lambda_{\delta}} \rangle$
such that:
\mn
\begin{enumerate}
\item[$(i)$]   $M^{\delta,\alpha}_{n}$ is a model of power $\kappa$,
countable vocabulary $\subseteq {\cH}(\aleph_0)$ including the 
predicate $<_*$, universe including $\kappa + 1$ and being included in
$\lambda_{\delta}$
\sn
\item[$(ii)$]   $M^{\delta,\alpha}_{n} \prec M^{\delta,\alpha}_{n+1}$
and $M^{\delta,\alpha}_n$ is a proper initial segment of 
$M^{\delta,\alpha}_{n+1}$
\sn
\item[$(iii)$]  $M^{\delta,\alpha}_{n} \cap \mu =
  M^{\delta,\alpha}_{0} \cap \mu$
\sn
\item[$(iv)$]  $\eta^{\delta, \alpha} \in \prod\limits_{n}
  \lambda_{\delta,n}$ and $\eta^{\delta,\alpha} \rest (n+1)$  belongs to
$M^{\delta,\alpha}_{n+1}$ but not to $M^{\delta,\alpha}_{n}$
\sn
\item[$(v)$] for some increasing sequence $\bar{k}_{\delta, \alpha} = \langle k_{\delta, \alpha}(\ell) \colon \ell < \omega \rangle$ from ${}^{\omega} \omega$ we have $\bigcup\limits_{\ell < n} \lambda_{\delta,k_{\delta,\alpha}(\ell)} 
< \eta^{\delta,\alpha}(n) < \lambda_{\delta,k_{\delta,\alpha}(n)}$
[hence $\lambda_{\delta} = \bigcup\limits_{n} 
\eta^{\delta,\alpha}(n)$],
\sn
\item[$(vi)$]  if $\alpha < \beta< 2^{\lambda_\delta}$ and 
$\delta \in S$ then for some $m< \omega$ we have
\[
(\bigcup\limits_{n < \omega} M^{\delta,\beta}_n) \cap
(\bigcup\limits_{n < \omega} M^{\delta,\alpha}_n) \subseteq M^{\delta,\beta}_m
\]
\end{enumerate}
\mn
hence
\mn
\begin{enumerate}
\item[$(vii)$]  for $\alpha < 2^{\lambda_\delta},\delta\in S$ we have
$\sup(M^{\delta,\alpha}_n \cap \lambda_{\delta,n}) = \sup(M^{\delta,
\alpha}_{n+1} \cap \lambda_{\delta,n})$
\sn 
\item[$(viii)$]  if $M^*$ is a model with countable vocabulary 
$\subseteq {\cH}(\aleph_0)$ and universe $\lambda$ with 
$<_*^{M^*} = < \rest \lambda$ then for some
\footnote{really for a club of  $\delta \in S$ for ``many"  $\alpha <
2^{\lambda_{\delta}}$ this holds}
$\delta \in S$ and $\alpha < 2^{\lambda_{\delta}}$ we have 
$\bigwedge_{n} M^{\delta,\alpha}_{n} \prec M^*$
\sn
\item[$(ix)$]  if $\delta\in S,\alpha \ne \beta$ are $< 2^{\lambda_\delta}$
then $\{\eta^{\delta,\alpha}\rest n:n<\omega\} \nsubseteq 
\cup\{M^{\delta,\beta}_n:n <\omega\}$.
\end{enumerate}
\end{subf}

\begin{PROOF}{\ref{7.6C}}
Straightforward from \ref{7.6B} (and diagonalizing).
\end{PROOF}

\noindent
\underline{Proof of \ref{7.6}(3)}:  Should be clear now using \cite[1.16=La45]{Sh:E62} and similar to \cite[Th.~2.6, pg.~113]{Sh:1116}  
(and see the proof of \ref{7.8}(1) below).

\begin{PROOF}[Proof of \ref{7.6}(4)]{\ref{7.6}}
     Similar (and not used for \ref{7.11}). 
\end{PROOF}
\end{PROOF}
\end{PROOF}

\noindent
For our main conclusion \ref{7.11} we shall not actually use \ref{7.7}
(as other cases  cover it).

\begin{claim}
\label{7.7}
Suppose $\lambda$ is singular,  $\mu < \lambda$ and for
arbitrarily large $\theta < \lambda $ at least one of the conditions
$(*)^1_\theta$, $(*)^2_\theta$ below holds.
\Then \,  $K^\omega_{\tr}$ has the full 
$(\lambda,\lambda,\mu,\mu)$-super$^7$-bigness property
\mn
\begin{enumerate}
\item[$(*)^1_{\theta}$]  $\theta$ singular, 
$\pp(\theta) > \theta^+$ (see Definition \cite[3.16=Lprf.2]{Sh:E62})
\sn
\item[$(*)^2_{\theta}$]  there is a set $\ga$ of regular cardinals 
$< \theta$ unbounded below $\theta,|\ga| < \theta$,  such that $\partial <
\theta \Rightarrow \max \pcf(\ga \setminus \partial) > \theta^+.$
\end{enumerate}
\end{claim}

\begin{PROOF}{\ref{7.7}}
First, by \cite[3.22=Lpcf.8]{Sh:E62} we have $(*)^1_{\theta} \Rightarrow 
(*)^2_{\theta}$, second, by \cite[3.20=Lpcf.6a]{Sh:E62} \wilog \, 
$\cf(\theta)=\aleph_0$; third, by \cite[3.22=Lpcf.8]{Sh:E62} 
\wilog \, ${\ga}$ has order type $\omega$ and 
$J$ is the ideal of bounded subsets of ${\ga}$, lastly by 
\cite[3.10=Lpcf.1]{Sh:E62} (and easy manipulation) $(*)^2_{\theta} 
\Rightarrow (*)^3_{\theta^+}$, where
\begin{enumerate}
\item[$(*)^3_{\partial}$]  $\partial$ regular, there is a stationary 
$S \subseteq \{\delta < \partial:\cf(\delta) = \aleph_{0}\}$ 
and $\eta_{\delta}$,  an increasing $\omega$-sequence converging to
$\delta$, for $\delta \in S$, such that for every 
$\alpha < \partial$, for some $h:S \cap \alpha \rightarrow \omega$ we have:
$\{\eta_{\delta} \rest \ell:h(\delta) < \ell < \omega\}:
\delta \in S \cap \alpha\}$ are pairwise disjoint.
\end{enumerate}
\mn
So assume $\langle \theta_{i}:i < \cf(\lambda)\rangle$ is strictly increasing,

\[
\mu < \theta_{i} < \sum\limits_{j< \cf\lambda}\theta_{j} = \lambda,
\]

\mn
each $\theta_{i}$ regular and for each $i,\langle \eta^i_{\delta}:
\delta \in S_{i}\rangle$ is as required in $(*)^3_{\theta_{i}^+}$.
Let $\langle S_{i,\alpha }:\alpha < \theta_{i}\rangle$ be a partition 
of $S_{i}$ to (pairwise disjoint) stationary sets.
For $\bigcup\limits_{j < i} \theta_j \le \alpha < \theta_{i}$, let 
$I_{\alpha} = {}^{\omega >}\lambda \cup \{\eta^i_{\delta}:\delta \in 
S_{i,\alpha}\}$. The rest is as in  \ref{7.6}(1) above (or the proof of \ref{7.14}, Case 1 below).
\end{PROOF}

\begin{remark}
\label{7.7d}
In \ref{7.7} we can use $(*)^3_\partial,\partial$ regular arbitrarily
large $< \lambda$.
\end{remark} 
\newpage

\section{Further cases of super unembeddability}

\begin{lemma}
\label{7.8}
1) Suppose $\lambda \ge \mu^+ + \chi^{+2}$  \then \, $K^\omega_{\tr}$ has the full
$(\chi^{\aleph_{0}},\lambda,\mu,\mu)$-super-bigness property.

\noindent
2) In addition $K^\omega_{\tr}$ has the full 
$(\chi^{\aleph_{0}},\lambda,\mu,\mu)$-super$^5$-bigness property (with $\dot{D} =
D^{\cbe}_{\omega} = \{A \subseteq \omega$: every large enough even
number belongs to $A\}$).

\noindent
3) In part (2), we can add in Definition \ref{7.3}, Case $E$ the
 requirement:
\mn
\begin{enumerate}
\item[$\circledast$]  if $\nu \in P^I_{\omega} \cup P^J_{\omega}$ 
and $\{\nu \rest k:k < \omega\} \subseteq M_{n}$ then $\nu \in M_{n}$.
\end{enumerate}
\end{lemma}

\noindent
This claim will be proved later (after \ref{7.8I1}). Towards this we develop ``guessing of clubs'' in $ZFC$, in fact for
this it  was introduced.
\begin{claim}
\label{7.8A}
Suppose $\kappa,\lambda$ are regular cardinals, $\kappa^+ < \lambda$, and

\[
S \subseteq \{\delta < \lambda:\cf(\delta) = \kappa\}
\]

\mn
is a stationary subset of $\lambda$.

\Then \, we can find $\langle C_{\delta}:\delta \in S\rangle$ such
that:
\mn
\begin{enumerate}
\item[$(a)$]  $C_{\delta}$ is a club of $\delta$ of order type $\kappa$ (if
$\kappa = \aleph_{0},C_{\delta}$ is just an unbounded subsetd of $\delta$ and
$\otp(C_{\delta}) = \omega)$
\sn
\item[$(b)$]  for every club $C$ of $\lambda$, the set $\{\delta \in
S:C_{\delta} \subseteq C\}$ is stationary.
\end{enumerate}
\end{claim}

\begin{PROOF}{\ref{7.8A}}
    Toward contradiction,  suppose that such $\langle C_{\delta}:\delta \in S\rangle$ does not
    exist.  Let  $\langle C^*_{\delta}:\delta \in S\rangle$ satisfy (a).  We
    choose $E_{\zeta}$ by induction on $\zeta < \kappa^+$ such that:
    \mn
    \begin{enumerate}
    \item[$(i)$]  $E_{\zeta}$ is a club of $\lambda,$ and $0 \notin E_{\zeta}$,
    \sn
    \item[$(ii)$]  $\xi < \zeta \Rightarrow  E_{\zeta} \subseteq E_{\xi}$,
    \sn
    \item[$(iii)$]  for no $\delta \in S$ does $C^\zeta_{\delta} \subseteq
    E_{\zeta +1}$ hence $\delta = \sup(E_{\zeta +1} \cap \delta)$ where
    \[
    C^\zeta_{\delta} := \{\sup(\alpha \cap E_{\zeta}):\alpha \in
    C^*_{\delta},\alpha > \min (E_{\zeta})\}.
    \]
    \end{enumerate}
    \mn
    For $\zeta = 0,\zeta$ limit: we have no problem.  For $\zeta = \xi + 1$, first define $C^\xi_{\delta}$ for $\delta \in S$ as in clause (iii). Then, let $E'_{\zeta}$ be the set of accumulation points of $E_{\xi}$, so clearly $\delta \in E'_\zeta \Rightarrow \delta = \sup(C_{\delta}^{\xi}) \wedge \kappa = \otp(C_{\delta}^{\xi})$. By our assumption toward contradictions, $\langle C_{\delta}^{\xi} \colon \delta \in E_{\zeta}' \cap S \rangle$ does not satisfy both properties (a) and (b), but obviously it satisfies clause (a), so it fails to satisfy clause (b) (it does not matter what we do for $\delta \in S \setminus E_{\zeta}'$). So for some club $E_\zeta''$ of $\lambda$, the set $A_{\zeta} = \{\delta \in E'_{\zeta} \cap S:C^\xi_{\delta} \subseteq E_\zeta''\}$ is not stationary, so  it is disjoint to some club $E_{\zeta}$ and without
    loss of generality $E_{\zeta}$ is a subset of $E_\zeta''\cap E'_{\zeta}$. So we have carried the induction. 
    
    In the end as $\kappa^{+} < \lambda = \cf(\lambda)$, clearly $E^+ = \bigcap\limits_{{\zeta} < {\kappa^+}}E_{\zeta}$ is a club of  $\lambda$, choose $\delta(*) \in S$ which is an accumulation point of $E^+$; so $\delta(*)\in E_\zeta\cap S$
    for every $\zeta < \kappa^+$. Now for each  $\alpha \in C_{\delta(*)}^{\ast}$, which is $ > \min(E^+)$, the sequence
    \[
    \langle \sup (\alpha \cap E_{\zeta}):\zeta < \kappa^+\rangle
    \]
    is a non-increasing sequence of ordinals
    $\le \alpha$,  hence is eventually constant.  As $\kappa^+$ is regular $>
    \kappa = |C_{\delta(*)}^{\ast}|$, for some  $\zeta(*) < \kappa^+$,  for
    every $\zeta \in [\zeta(*),\kappa^+)$ and $\alpha \in C^*_{\delta(\ast)}$, which is above $\min(E^{+})$, we have
    \[
    \sup(\alpha \cap E_{\zeta}) = \sup(\alpha \cap E_{\zeta(*)}).
    \]
    
    Hence $C^{\zeta(*)}_{\delta(*)} = C^{\zeta(*)+1}_{\delta(*)}$, and 
    we get a contradiction to the choice of $E_{\zeta(*)+1}$.
\end{PROOF}

\begin{remark}\label{7.8A1}
    If $\kappa > \aleph_0$, the proof is simpler, just $C^\zeta_\delta 
    = C^*_\delta \cap E_\zeta$ is O.K.
\end{remark}

\begin{claim}\label{7.8B}
Suppose that in \ref{7.8A} we have also $\kappa < \theta = \cf(\theta)
< \lambda$; \then \, we can add
\mn
\begin{enumerate}
\item[$(c)$]  for some club $C$ of $\lambda$, if $\delta \in S \cap C,
\alpha \in C_{\delta}$ and $\alpha > \sup(\alpha \cap C_{\delta})$
then $\cf(\alpha) \ge \theta$.
\end{enumerate}
\end{claim}

\begin{PROOF}{\ref{7.8B}}
    Let $\bar{C} = \langle C_{\delta} \colon \delta \in S \rangle$ be as in \ref{7.8A}. Let $S^+ =: \{\delta < \lambda:\cf(\delta) < \theta\}$  (so $S
    \subseteq S^+$).  For each $\delta \in S^+$ choose a club
    $C^*_{\delta}$ of $\delta$ of order type $\cf(\delta)$.
    Assume that the conclusion fails,   that is, we cannot find $\bar{C}'$ satisfying clauses (a), (b), (c). Above me may weaken the demand ``$\otp(C_{\delta}') = \kappa$'' as we can shrink the $C_{\delta}'$-s appropriately.  
    
    We define by induction on $\zeta < \theta,E_{\zeta}$ such that:
    \mn
    \begin{enumerate}
        \item[$(i)$]  $E_{\zeta}$ is a club of $\lambda,0 \notin E_{\zeta}$
        \sn
        \item[$(ii)$]  $\xi < \zeta$ implies $E_{\zeta} \subseteq E_{\xi}$
        \sn
        \item[$(iii)$] for no $\delta \in S$ do we have $C^\zeta_{\delta} \cap E_{\zeta} \subseteq E_{\zeta +1},\delta = \sup(C^\zeta_{\delta} \cap E_{\zeta})$ where $C_{\delta}^{\zeta} = C_{\delta}^{\zeta, \omega}$ and by recursion on $n  \leq \omega$, we define $C_{\delta}^{\zeta, n}$ as follows:
    
        \begin{enumerate}
            \item[(a)] $C^{\zeta,0}_{\delta} \coloneqq \{\sup (\alpha \cap E_{\zeta}):\alpha \in C_{\delta},\alpha > \min(E_{\zeta})\}$,
    
            \item[(b)] $ C^{\zeta,n+1}_{\delta} \coloneqq C^{\zeta, n}_{\delta} \cup \{\sup(\alpha \cap  E_{\zeta}) \colon  \text{ for some } \beta \in C^{\zeta,n}_{\delta} \text{ we have }
            \cf(\beta) < \theta \text{ and }
            \alpha \in C^*_{\beta}, \alpha > \min(E_{\zeta}) \text{ and } \alpha > \sup[C^{\zeta,n}_{\delta} \cap \beta]\}$, and 
    
            \item[(c)] $C^\zeta_{\delta} \coloneqq \bigcup\limits_{n < \omega} C^{\zeta,n}_{\delta}$.
        \end{enumerate}
    \end{enumerate}

    For $\zeta = 0,\zeta$ limit: we have no problems. For $\zeta = \xi +1$,  
    we first define $C^{\xi,0}_{\delta}$ and then $C^{\xi,n}_{\delta}$ (by induction on $n$) and lastly $C^\xi_{\delta}$. We can show by induction on $n$ that $C_{\delta}^{\xi, n} \subseteq E_{\zeta}$ and 
    \begin{enumerate}
        \item[$(\ast)_{0}$] $C_{\delta}^{\zeta, n}$ is a closed subset of $\delta$ of cardinality $< \theta$.
    \end{enumerate}
    
    Now, 
     
    \begin{enumerate}
        \item[$(\ast)_{1}$] 

        \begin{enumerate}
            \item[(a)] $C_{\delta}^{\xi}$ is of cardinality $< \theta$ (as $\aleph_{0} < \theta = \cf(\theta)$), 

            \item[(b)] if $\gamma = \sup(C_{\delta}^{\xi} \cap \gamma) < \delta$, then 

    \begin{itemize}
        \item $\gamma$ limit of cofinality $\leq \vert C_{\delta}^{\xi} \vert < \theta$, 

        \item $\langle \min(C_{\delta}^{\xi, n} \setminus \gamma) \colon n < \omega \rangle$ is $\leq$-decreasing, hence, 

        \item for some $n(\ast) < \omega$, we have: 
        $$ n \geq n(\ast) \Rightarrow \min (C_{\delta}^{\xi, n} \setminus \gamma) = \min(C_{\delta}^{\xi, n(\ast)}),$$ 

        \item so by the choice of $C_{\delta}^{\xi, n(\ast) + 1}$ necessarily $\gamma \in C_{\delta}^{\xi, n(\ast)}$. 
    \end{itemize}
        \end{enumerate}
    \end{enumerate}

    Together, 

    \begin{enumerate}
        \item[$(\ast)_{2}$] $C_{\delta}^{\xi}$ is a closed subset of $\delta$.  
    \end{enumerate}

    Hence, 

    \begin{enumerate}
        \item[$(\ast)_{3}$] if $\delta$ is an accumulation point of $E_{\delta}^{\xi}$, then $C_{\delta}^{\xi}$ is a club of $\delta$.  
    \end{enumerate}

    % Let us show that we can construct them such that for each $\alpha \in C_{\delta}^{\xi}$: 

    % \begin{enumerate}
    %     \item[$(\ast)_{4}$]  $\alpha > \sup(C^\xi_{\delta} \cap \alpha) \, \& \, \alpha \in  C^\xi_{\delta} \Rightarrow  \cf(\alpha) \ge \theta \vee \alpha >  \sup(E_{\xi} \cap \alpha).$
    % \end{enumerate}

    If ``for every club $E$ of $\lambda$ for some $\delta \in S \cap
    \acc(E_{\xi}),C^\xi_{\delta} \subseteq E"$ then we can shrink the 
    club $E$; i.e. deduce  $C^\xi_{\delta}$ is included
    in the set of accumulation points of $E \cap E_{\xi}$ hence
    $\langle C^\xi_\delta:\delta\in S \cap \acc(E_\xi)\rangle$
    satisfies ``for every club $E$ of $\lambda$ for some $\delta \in S
    \cap E_{\xi}$, we have $C^\xi_\delta \subseteq E$ and $(\forall
    \alpha)[\alpha \in C^\delta_{\xi} \, \& \, \alpha > \sup
     ( C_{\delta}^{\xi} \cap \alpha) \Rightarrow \cf(\alpha) \ge \theta]"$ 
    so the desired conclusion holds.
    
    Hence we can assume that for some club $E^1_{\zeta}$ of $\lambda$, for no
    $\delta \in  S \cap E^1_{\zeta} \cap \acc(E_{\xi})$ does $C^\xi_{\delta}
    \subseteq E^1_{\zeta}$; let $E_{\zeta}$ be the set of accumulation points of
    $E^1_{\zeta} \cap E_{\xi}$. So we have finished defining $C_{\delta}^{\xi, n}$ for $n < \omega$ and $C_{\delta}^{\xi}$. 
    
    So we have finished defining $\langle E_{\zeta} \colon \zeta < \theta \rangle$ and we choose $\delta(*) \in S$  a accumulation point of $\bigcap\limits_{\zeta < \theta} E_{\zeta}$. Again as in the proof of \ref{7.8A}, for some $\zeta(0) < \theta$, we have

    \begin{enumerate}
        \item[$(\oplus)$] $ \zeta(0) \le \zeta < \theta \Rightarrow C^{\zeta,0}_{\delta(*)} = C^{\zeta(*),0}_{\delta(*)}.$
    \end{enumerate}

    Similarly we can prove by induction on $n$ that for some $\zeta(n)< \theta$:

    \begin{enumerate}
        \item[$\oplus_{2}$] $\zeta(n) \le \zeta < \theta \Rightarrow C^{\zeta,n}_{\delta(*)} = C^{\zeta(*),n}_{\delta(*)}.$      
    \end{enumerate}
    
    Let $\zeta(*) = \bigcup\limits_{n < \omega} \zeta(n)$, and we 
    get contradiction as in the proof of \ref{7.8A}.
\end{PROOF}

\noindent
Recall (see \cite[3.8]{Sh:309} or \cite[0.4=L0.2A]{Sh:E12}).

\begin{definition}\label{7.8B1}
    1) We say $\bar C$ is a pre-partial square sequence of $\lambda$: (omitting $\lambda$ means $(\exists \lambda)[\bar{C}$ is a partial square of $\lambda]$) \when \,: $\bar C$ has the form $\langle C_\alpha:\alpha \in S\rangle$ and satisfies:
    \begin{enumerate}
    \item[$(a)$]  $S \subseteq \lambda$
    \sn
    \item[$(b)$]  $C_\alpha$ is a closed subset of $\alpha$
    \sn
    \item[$(c)$]  $C_\alpha \subseteq S$
    \sn
    \item[$(d)$]  if $\beta \in C_\alpha,\alpha \in S$ \then \, $C_\beta =
      C_\alpha \cap \beta$.
    \end{enumerate}

    2) We say $\bar{C}$ is partial square when in addition:
    
    \begin{enumerate}
        \item[$(e)$]  if $\alpha \in S$ is a limit ordinal \then \, $\alpha = \sup(C_\alpha)$.

        \item[(f)] $\otp(C_{\alpha}) < \alpha$ if $0 < \alpha \in S$.
    \end{enumerate}
\end{definition}

We next state a consequence of the failure of partial square (the reader may suspect the assumption (1) is vacuous, e.g., by \ref{7.8E}, but e.g. there we speak on successor cardinals).

\begin{conclusion}\label{7.8C}
    If $\lambda > \kappa$ are regular, \underline{then} (A) $\Rightarrow$ (B), where: 

    \begin{enumerate}
        \item[(A)] There are no pairs $(S^{+}, \bar{C})$ such that: 

        \begin{itemize}
            \item $S^{+} \subseteq \lambda$, 

            \item $\bar{C} = \langle C_{\delta} \colon \delta \in S^{+} \rangle$ is a partial square, and

            \item $\{ \delta \in S^{+} \colon \cf(\delta) = \kappa \}$ is stationary. 
        \end{itemize}

        \item[(B)] For every regular $\lambda_{1} \ge \lambda$ there is a stationary $S \subseteq \{\delta < \lambda^+_{1}:\cf(\delta) = \kappa\}$ which does not reflect in any $\delta < \lambda^+_{1}$ of cofinality $\lambda$, (really one $S\subseteq \lambda^+_{1}$ works for all such $\lambda,\kappa < \lambda < \lambda_{1}$).
    \end{enumerate}
\end{conclusion}

\begin{PROOF}{\ref{7.8C}}
    The case $\kappa=\aleph_0$ is trivial because clause (A) never holds. In detail, we can choose $\langle C_{\alpha} \colon \alpha < \lambda \text{ is a successor ordinal} \rangle$ such that: 

    \begin{enumerate}
        \item[$(\ast)$] 

        \begin{enumerate}
            \item[(a)] $C_{\alpha}$ is a finite subset of $\alpha$, 

            \item[(b)] $\beta \in C_{\alpha} \Rightarrow C_{\beta} = C_{\alpha} \cap \beta$, 

            \item[(c)] if $C$ is a finite subset of $\alpha$, then for some $\beta < \alpha + \vert \alpha \vert + \vert \alpha \vert$ we have $C_{\beta + 1} = C$. 
        \end{enumerate}
    \end{enumerate}

    [Why? By induction on $\alpha$ belonging to the club $E \coloneqq \{ \beta < \lambda \colon \beta$ is infinite divisible by $\vert \beta \vert\}$, we can choose $\langle C_{\beta + 1} \colon \beta < \alpha \rangle$ satisfying the above (actually is ``$\beta < \alpha + \vert \alpha \vert$'' suffice).]

    Let $S \coloneqq \{ \alpha \in \Ord \colon \cf(\alpha) = \aleph_{0} \wedge \alpha = \sup(E \cap \alpha) \}$ and 
    \[
    S^{+} \coloneqq \{ \alpha \in \Ord \colon \alpha \text{ is a successor or } \alpha \in S  \}.
    \]

    For $\alpha \in S$, let $\langle \beta_{\alpha, n} \colon n < \omega \rangle$ be increasing with limit $\alpha$, and choose $\gamma_{\alpha, n}$ by induction on $n < \omega$ increasing with $n$ such that $\gamma_{\alpha, n} \in ( \beta_{\alpha, n}, \alpha)$ is a successor ordinal and $C_{\gamma_{\alpha, n}} = \{ \gamma_{\alpha, m} \colon m < n  \}$ and lastly, let $C_{\alpha} = \{ \gamma_{\alpha, n} \colon n < \omega \}$. 
    
    So assume $\kappa>\aleph_0$. By clauses (c) and (d) of \cite[3.8(2)=L6.4,3.9=L6.4B]{Sh:309} in the arXiv version, \cite[4.8(2)]{Sh:309} in the published one, we can find $S^+\subseteq \lambda^+_1$ and a partial  square $\bar C=\langle C_\delta: \delta\in S^+\rangle$ such that $\delta\in S^+ \Rightarrow \otp(C_\delta)\le \kappa$, and
    \[
    S = \{\delta \in S^+:\otp(C_\delta)=\kappa\} = \{\delta\in S^+:
    \cf(\delta)=\kappa\} \text{ is stationary.}
    \]
    
    Towards a contradiction, assume that $S$ reflect in some $\delta,\cf(\delta)=\lambda$, let $\langle
    \alpha_\zeta \colon \zeta<\lambda\rangle$ be a strictly increasing continuous sequence
    with limit $\delta$.
    
    Let $S^+_\lambda=\{\zeta<\lambda:\alpha_\zeta \in S^+\}$ and for $\zeta\in S^+_\lambda$, $C^\lambda_\zeta=\{\epsilon<\zeta:
    \alpha_\epsilon\in C_{\alpha_\zeta}\}$; now $\langle C^\lambda_\zeta:
    \zeta\in S^+_\lambda\rangle$ essentially show that for $\lambda$ satisfying the assumption, the conclusion holds; note possibly for some $\zeta \in S^+_\lambda$ we have $\sup(C^\lambda_\zeta) < \zeta$ but then
    $\cf(\zeta) = \aleph_0$ and for no $\varepsilon \in S_{\lambda}^{+}$ do we have $\zeta \in C_{\varepsilon}^{\lambda}$. To correct this, let 
    \[S'_\lambda = \{\zeta\in
    S_\lambda \colon \zeta = \sup(C^\lambda_\zeta)\}\cup \{\zeta+1: \zeta \in S_\lambda \text{ and } \zeta > \sup(C^\lambda_\zeta)\}
    \] and redefine
    $C^\lambda_\zeta$ accordingly.
\end{PROOF}

\begin{claim}\label{7.8D}
    Suppose $\lambda = \theta^+,\theta$ regular uncountable, $S^0 = \{\delta < \lambda:\cf(\delta) = \theta\}$.  \Then \, we can find $\langle C_{\delta}:\delta \in  S^0\rangle$ such that:
    
    \begin{enumerate}
    \item[$(a)$]  $C_{\delta}$ is a club of $\delta$ of order type $\theta$,
    \sn
    \item[$(b)$]  for any club $E$ of $\lambda$, the set
    \[
    \{\delta \in S^0:\delta = \sup\{\alpha :\alpha \in C_{\delta},
    \alpha > \sup(\alpha \cap  C_{\delta}) \text{ and } \alpha \in E\}\}
    \]
    is stationary.
    \end{enumerate}
    \mn
    Equivalently, 
    \mn
    \begin{enumerate}
    \item[$(c)$]   for any club $E$ of $\lambda$
    \[
    \{\delta \in S^0:\{\zeta < \theta: \text{ the }(\zeta +1)
    \text{-th member of } C_{\delta} \text{ is in } E\} \text{ is an unbounded subset of } \theta \}
    \]
    
    is a stationary subset of $\lambda$.
    \end{enumerate}
\end{claim}

\begin{remark}\label{7.8Da} 
    In clause (c) above, obviously for a club of $\zeta<\theta$, the
    $\zeta$-th member of $C_\delta$ is in $E$.
\end{remark}

\begin{PROOF}{\ref{7.8D}}
    Like the proof of \ref{7.8A}, again we continue $\omega$ times,
    and assume the failure of the statement here, noting: if $C_{\delta}$ is a club of $\delta$ of order type $\kappa$, $E$ is a club of $\delta$ such that $\otp(E)$ is divisible by $\kappa^{2}$, then for unboundedly many $\alpha \in C_{\delta}$ we have $\alpha > \sup (C_{\delta} \cap \alpha)$. 
\end{PROOF}

\noindent
Now we give a small improvement of \ref{7.8A}.

\begin{claim}\label{7.8E}
    Suppose $\lambda > \kappa + \aleph_{1}$ where $\lambda$ and $\kappa$ are regular cardinals, and $\epsilon(*)$ is a limit ordinal $< \lambda$ of cofinality $\kappa$.  \Then \, if $S_{\ast}$ satisfies (A) then we can find $\bar{C} = \langle C_{\delta} \colon \delta \in S_{\ast} \rangle$ satisfying\footnote{As in~\ref{7.8F}(3) we can add to clause (B)(d), there is a partial square $\langle C_{\delta}' \colon \delta \in S' \rangle$ where $S \subseteq S' \subseteq \lambda$ and $\bar{C}' \rest S = \bar{C}$.} (B), where: 

    \begin{enumerate}
        \item[(A)] 

        \begin{enumerate}
            \item[(a)] $S_{\ast} \subseteq \{ \delta < \lambda^{+} \colon \cf(\delta) = \kappa$ and $\delta$ is divisible\footnote{``\emph{$\delta$ divisible} by $\lambda \times \kappa$'' means $(\forall \alpha < \delta)(\alpha + \lambda \times \kappa \leq \delta)$} by $\lambda \times \kappa\}$, 

            \item[(b)] $S_{\ast}$ is a stationary subset of $\lambda^{+}$.
        \end{enumerate}

        \item[(B)] 

        \begin{enumerate}
            \item[(a)] $C_{\delta}$ is a closed unbounded subset of $\delta$, 

            \item[(b)] $\otp(C_{\delta}) = \epsilon(\ast)$, 

            \item[(c)] for every club $E$ of $\lambda^{+}$, $\{ \delta \in S_{\ast} \colon C_{\delta} \subseteq E \}$ is stationary.
        \end{enumerate}
    \end{enumerate}
    % for any stationary 
    % \[
    % S_* \subseteq \{\delta < \lambda^+:\cf(\delta) =
    % \kappa,\delta \text{ divisible by } \lambda \times \kappa\}
    % \]
    
    % \mn
    % we can find $\langle C_{\delta}:\delta \in S_* \rangle$ such that:
    % \mn
    % \begin{enumerate}
    % \item[$(a)$]  $C_{\delta}$ is a closed unbounded subset of $\delta$
    % \sn
    % \item[$(b)$]  $\otp(C_{\delta}) = \epsilon(*)$ 
    % \sn
    % \item[$(c)$]  for every club $E$ of $\lambda^+,\{\delta \in S_*:C_{\delta}
    % \subseteq E\}$ is stationary.
    % \end{enumerate}
\end{claim}

\noindent
Before we prove \ref{7.8E}, we phrase a further claim:

\begin{claim}\label{7.8F}
    In \ref{7.8E}:
    
    \noindent
    1) E.g.
    \[
    S_* := \{\delta < \lambda^+:\cf(\delta) = \kappa \text{ and } \delta 
    \text{ is divisible by } \lambda \times \kappa\}
    \]
    
    \mn
    is as required.
    
    \noindent
    2) We can add:
    \mn
    \begin{enumerate}
    \item[$(d)$]  $\alpha \in C_{\delta} \, \& \, \alpha > \sup(C_{\delta} \cap
    \alpha) \Rightarrow \cf(\alpha) = \lambda$.
    \end{enumerate}
    \mn
    3) For any sequence $\bar \epsilon =\langle \epsilon_\zeta(*):
    \zeta < \lambda \rangle$, where $\epsilon_\zeta(*) < \lambda$ is a
    limit ordinal let $\kappa_\zeta = \cf(\varepsilon_\zeta(*))$, we can find
    $\langle(S^*_\zeta,S^+_{\zeta},\bar C^\zeta):\zeta < \lambda \rangle$ 
    such that:
    \mn
    \begin{enumerate}
    \item[$(i)$]  $\{\alpha:\alpha < \lambda^+,\cf(\alpha) < \lambda\}
    = \bigcup\limits_{\zeta < \lambda} S^+_\zeta$,
    \sn
    \item[$(ii)$]  $S^*_{\zeta} \subseteq S^+_{\zeta} \subseteq 
    \{\alpha:\alpha < \lambda^+,\cf(\alpha) < \lambda\}$,
    \sn
    \item[$(iii)$]  $S^*_{\zeta} \subseteq \{\delta < \lambda^+:\cf(\delta) =
    \kappa_{\zeta}\}$,
    \sn
    \item[$(iv)$]  $\delta \in S^*_{\zeta} \Rightarrow \otp(C^\zeta_{\delta}) =
    \epsilon_{\zeta}(*)$,
    \sn
    \item[$(v)$]  $\bar C^\zeta= \langle C^\zeta_{\alpha}:\alpha \in 
    S^+_{\zeta} \rangle$ satisfies (a), (c) of \ref{7.8E}(B) (and (b) - for
    $\epsilon_{\zeta}(*))$,
    \sn
    \item[$(vi)$]  $\bar C^\zeta$ is a partial square, 
    
    \item[$(vii)$] $\langle S_{\zeta}^{\ast} \colon \zeta < \lambda \rangle$ is a sequence of pairwise disjoint sets. 
    \end{enumerate}
\end{claim}

\begin{PROOF}{\ref{7.8F}}
\noindent
\underline{Proof of \ref{7.8E}}:

We know here essentially by \cite[3.8(2)=L6.4(2)]{Sh:309} (in the arXiv version, \cite[4.8(2)]{Sh:309}) in the published one and by
\cite[\S4]{Sh:351} that there are $\langle S_{\zeta}:\zeta < \lambda
\rangle$ and $\bar C^\zeta = 
\langle C^\zeta_{\alpha} \colon \alpha \in S_{\zeta}\rangle$ for
$\zeta < \lambda$  such that:
\mn
\begin{enumerate}
\item[$(*)_1$]  $(\alpha) \quad \bar C^\zeta $ is a pre-partial square sequence of $\lambda^{+}$,
\sn
\item[${{}}$]  $(\beta) \quad$ if $\alpha > \sup(C^\zeta_\alpha)$ then
  $\cf(\alpha) \in \{1,\lambda\}$,
\sn
\item[${{}}$]  $(\gamma) \quad \lambda^{+} =  \bigcup\limits_{\zeta < \lambda} S_{\zeta}$, 
\sn
\item[${{}}$]  $(\delta) \quad |C^\zeta_{\alpha}| < \lambda$,
\sn
\item[${{}}$]  $(\varepsilon) \quad$ if $\alpha < \lambda^+$, 
then for some $\xi < \lambda$ we have:
\sn
\begin{enumerate}
\item[${{}}$]  $\bullet \quad$ if $\zeta < \lambda$ then $ \alpha \in S_\zeta \Leftrightarrow \zeta 
\ge \xi$,
\sn
\item[${{}}$]  $\bullet \quad \langle C^\zeta_\alpha:
\zeta \in [\xi,\lambda]\rangle$ is $\subseteq$-increasing,
\sn
\item[${{}}$]  $\bullet \quad$ if $\cf(\zeta) = \aleph_1 \wedge \zeta
> \xi$, then $C^\zeta_\alpha$ is the closure of $\cup\{C^\varepsilon_\alpha:\varepsilon \in [\xi,\zeta)\}$ in $\alpha$,
  
\item[${{}}$]  $\bullet \quad \cup \{C^\zeta_\alpha:\zeta \in [\xi,\lambda)\}$ is equal to $\alpha$.
\end{enumerate}
\end{enumerate}
\mn
Obviously,
\mn
\begin{enumerate}
\item[$(*)_2$]  for every $\epsilon,\xi < \lambda$ and club $E^{0}$ of $\lambda^+$ for some club $E^1$ of $\lambda^+$, for each $\delta \in E^1$ 
of cofinality $<\lambda$, for some $\zeta < \lambda$ above $\xi$ we have: $\delta = \sup(E^{0} \cap C^\zeta_{\delta})$ and $\otp(E^{0} \cap
C^\zeta_\delta)$ is divisible by $\epsilon$.
\end{enumerate}

Next, (relying on $(\ast)_{1} + (\ast)_{2}$ above)

\begin{fact}\label{7.8E3}
    We have:
    \begin{enumerate}
        \item[$(\ast)_{3}$] 

        \begin{enumerate}
            \item[$(a)$] for every stationary $S \subseteq S^{\lambda^+}_{< \lambda}$ and $\xi,\varepsilon < \lambda$ and for some $\zeta(*) < \lambda,\zeta(*) > \xi$ and  $S \cap S_{\zeta(*)}$ is stationary, moreover:

            \item[$(b)$] above, if $E$ is a club of $\lambda^+$, then $\set^1_{\zeta(*),\varepsilon} (E,S)$ is stationary where $\set^1_{\zeta(*),\varepsilon}(E,S)$ is the set of $\delta$ such that:

            \begin{itemize}
                \item $\delta \in S \cap S_{\zeta(*)} \cap E$, 

                \item $\delta = \sup(C^{\zeta(*)}_{\delta} \cap E)$,

                \item $\otp(C^{\zeta(*)}_{\delta} \cap E)$ is divisible by $\epsilon \cdot \omega$ (ordinal multiplication), 

                \item if $\alpha \in C^{\zeta(*)}_\delta \wedge \alpha > \sup(C^{\zeta(*)}_\alpha)$ then\footnote{This is \underline{not} the same as $(\ast)_{2}(\beta)$} $\cf(\alpha) \in \{1,\lambda\}$,

                \item  if $\alpha \in C^{\zeta(*)}_\delta$ and $ \alpha \in E \wedge (\alpha > \sup(\alpha \cap E) > \sup(C^{\zeta(*)}_\delta \cap E)$  then\footnote{used only for (d) from \ref{7.8F}(2)} $\cf(\alpha) = \lambda$.
            \end{itemize}
        \end{enumerate}

        % \item[$(c)$] moreover, above $\set^1_{\zeta(*),\varepsilon}(E,\delta)$ is a stationary subset of $\lambda^+$. 
    \end{enumerate}
\end{fact}

\begin{PROOF}{\ref{7.8E3}}
    If not, then for every $\zeta \in [\xi,\lambda)$ there is a club $E^1_\zeta$ of $\lambda^+$ such that  $\set^1_{\zeta,\varepsilon}(E^{1}_\zeta,S)$ is not stationary, hence there is a club $E^{2}_{\zeta}$ of $\lambda^+$ disjoint to it. Let $E \coloneqq \cap\{E^1_\zeta \cap E^2_\zeta \colon \zeta \in [\xi,\lambda)\}$, it is a club of $\lambda^+$. As $E$ is a club of the regular uncountable cardinal $\lambda$, necessarily there is $\delta_* \in S$ such that $\delta_* = \otp(E \cap \delta_*)$.  Now for some club $E'$ of $\lambda^{+}$, for every $\zeta \in E'$ of cofinality $\aleph_1$ we get a contradiction to $(*)_2$,  therefore $(\ast)_{3}$ holds. 
\end{PROOF}
    
\begin{fact}\label{7.8E5}
    For regular $\kappa < \lambda$ and stationary $S \subseteq \{\alpha <
    \lambda^+:\cf(\alpha) = \kappa\}$ by induction on $\zeta < \lambda$ we
    can choose $\xi(\zeta,S) \in S$ such that: 
    
    \begin{enumerate}
        \item[$(*)_{4,S,\zeta}$]  $\xi(\zeta, S)$ is an ordinal $> \xi(\zeta_{1}, S)$ for every $\zeta_1 < \zeta$ but $< \lambda$ hence $\xi(\zeta, S) \in [\zeta,\lambda)$ such that: 

        \begin{itemize}
            \item if $E$ is a club of $\lambda^{+}$, then $\set_{\xi(\zeta)}^{1}(E)$ from $(\ast)_{3}$ is a stationary subset of $\lambda^{+}$. 
        \end{itemize}
    \end{enumerate}
\end{fact}

\begin{claim}\label{7.8E6}
    We have: 
     \begin{enumerate}
        \item[$(*)_5$]  for every regular $\kappa < \lambda$, stationary $S \subseteq \{\delta < \lambda^+:\cf(\delta) = \kappa\}$ and $\zeta < \lambda$ there are a club $E_*$ of $\lambda^+$ and ordinals $\delta_*,\delta_{1}(*),\delta_2(*)$ of cofinality $\kappa$ divisible by $\zeta$ such that: if $E$ is a club of $\lambda^+$ then the following is a stationary subset of $\lambda^+$:
        \begin{equation*}
            \begin{array}{clcr}
            \set^2_{\zeta,\delta_*}(E,E_*,S) := \{\delta:&\delta \in S \cap S_{\xi(\zeta, \delta)}, \delta = \sup(C^{\xi(\zeta, \delta)}_\delta \cap E_* \cap E),\\
            &\otp(C^{\xi(\zeta, \delta)}_\delta) = \delta_1(*),
            \otp(C^{\xi(\zeta, \delta)}_\delta \cap E_*) = \delta_2(*) \text{ and} \\
            &C^{\xi(\zeta, \delta)}_\delta \cap (E_* \cap
            E) = C^{\xi(\zeta, \delta)}_\delta \cap E_*\}.
            \end{array}
        \end{equation*}
    \end{enumerate}    
\end{claim}

\begin{PROOF}{\ref{7.8E6}}       
    Let $\langle \delta_i:i < \lambda\rangle$ list the ordinals $< \lambda$ of cofinality $\kappa$ divisible by $\zeta$, each appearing
    stationarily often. We choose by induction on $i < \lambda$, a club $E_i$ of $\lambda$, decreasing with $i$ such that $E_{i+1}$ exemplifying $(E_{i}, \delta_i)$ are not as required on $(E_*,\delta_*)$, moreover $E_{i+1}$ is disjoint to $\set^2_{\zeta,\delta_*}(E_{i+1},E_i,S)$, therefore $(\ast)_{5}$ holds.
\end{PROOF}

    Now we \underline{can prove Fact \ref{7.8E}}: apply $(*)_{4,S,\zeta}$
    for $\kappa,S$ being the $\kappa,S_*$ from \ref{7.8E} and $\zeta$
    being $\varepsilon(*) \times \omega$ and get $\xi := \xi(\zeta,S)$ as
    there, and let $(E_*,\delta_*)$ be as in $(*)_5$ of \ref{7.8E6}. 
    
    Clearly $\delta_*$ has a closed unbounded subset $C_*$ of order type
    $\varepsilon(*)$, as $\cf(\delta_*) = \kappa = \cf(\varepsilon(*))$
    and $\varepsilon(*) \cdot \omega$ divides $\delta_*$).

\underline{Continuing the proof of \ref{7.8F}}: 

Now for each $\delta \in S_*$ we choose $C_\delta$ as follows:
\begin{enumerate}
    \item[$\bullet$]  if $\delta = \sup(C^\xi_\delta \cap E_*)$ and $\delta_* = \otp(C^\xi_\delta)$ \then \, $C_\delta = \{\beta \in C^\xi_\delta \cap E_*,\otp(\beta \cap C^\xi_\delta \cap E_*) \in C_*\}$ and if \underline{otherwise} let $C_\delta$ be any closed unbounded subset of $\delta$, possible by the assumption on $S_*$ in \ref{7.8E}.
\end{enumerate}

% PROOF of 7.8F (doubs)
%\begin{PROOF}{\ref{7.8F}}

Considering Claim \ref{7.8F}(1) is obvious.

\noindent
Considering Claim \ref{7.8E}(2) stated below 
use the last clause in $(*)_3(b)$. 

\noindent
We are left with \ref{7.8F}(3).

\noindent
3) We can find a sequence $\langle S[\zeta] \colon \zeta < \lambda \rangle$ such that: 

\begin{enumerate}
    \item[$(\ast)$] 

    \begin{enumerate}
        \item[(a)] $S[\zeta] \subseteq \{ \alpha < \lambda^{+} \colon \cf(\alpha) =  \kappa_{\varepsilon} \text{ and } \kappa \times \varepsilon \text{ divide } \delta \}$,

        \item[(b)] $S[\zeta]$ is a stationary subset of $\lambda$, 

        \item[(c)] the $S[\zeta]$-s are pairwise disjoint. 
    \end{enumerate}
\end{enumerate}

Toward this we choose $\xi(\zeta),E^*_\zeta,\delta_1(\zeta),\delta_2(\zeta)$ by induction on $\zeta$ such that:

\begin{enumerate}
    \item[$\boxplus$]  $(a) \quad E^*_\zeta$ is a club of $\lambda^+$
    \sn
    \item[${{}}$]  $(b) \quad$ if $\alpha \in E^*_\zeta$ and $\alpha >
      \sup(\alpha \cap E^*_\zeta)$ then $\cf(\alpha) \in \{1,\lambda\}$
    \sn
    \item[${{}}$]  $(c) \quad$ if $\zeta(1) < \zeta$ then $E^*_\zeta
      \subseteq E^*_{\zeta(1)}$
    \sn
    \item[${{}}$]  $(d) \quad (\zeta,E^*_\zeta,\xi(\zeta), \varepsilon_\zeta(*),\delta_1(\zeta),\delta_2(\zeta))$ are as $(\zeta,E_*,\xi,\varepsilon,\delta_1(*),\delta_2(*))$ in $(*)_5$ for $S[\zeta]$. 
\end{enumerate}

\Then \, we can find a club $E_*$ of $\lambda^+$ which is $\subseteq
\cap\{E^*_\zeta:\zeta < \lambda\}$ and satisfies $\boxplus(b)$.  We
shall define $\langle (S^*_\zeta,S^+_\zeta,\bar C_{\bar \zeta}):\zeta
< \lambda\rangle$ as required in \ref{7.8F}(3) except that:
\begin{enumerate}
\item[$\bullet$]  $\bar C_\zeta$ is now $\langle
  C_{\zeta,\alpha}:\alpha \in S^+_\zeta\rangle$
\sn
\item[$\bullet$]  we replace $\lambda = \{\alpha:\alpha < \lambda^+\}$
  by $E_*$ so renaming we get the promised result filter.
\end{enumerate}
\mn
For each $\zeta$ let $C^*_\zeta$ be a club of $\delta_2(*)$ of order
type $\varepsilon_\zeta(*)$ such that $\alpha \in C^*_\zeta \wedge
\alpha > \sup(\alpha \cap C^*_\zeta),1 = \cf(\alpha)$.

We let

\begin{enumerate}
\item[$\bullet$]  $S^+_\zeta = S_{\xi(\zeta)} \cap E_*$
\sn
\item[$\bullet$]  if $\alpha \in S^+_\zeta$ and
  $\otp(C^{\xi(\zeta)}_\alpha \cap E_* > \delta_2(\zeta)$ then
  $C^*_{\zeta,\alpha} = \{\beta \in C^{\xi(\zeta)}_\alpha \cap
  E_*:\otp(\beta \cap C^{\xi(\zeta)}_\alpha \cap E_*) > \delta$
\sn
\item[$\bullet$]  if $\alpha \in S^+_\zeta$ and
  $\otp(C^{\xi(\zeta)}_\alpha \cap E_*)$ is $< \delta_2(\zeta)$ and
  $\notin C^*_\zeta$ then $C_{\zeta,\alpha} = \{\beta \in
  C^{\xi(\zeta)}_\alpha \cap E_*:\otp(\beta \cap C^{\xi(\zeta)}_\alpha
  \cap E_*)$ is $> \sup(C^*_\zeta \cap \otp(C^{\xi(\zeta)}_\alpha \cap
  E_*)\}$
\sn
\item[$\bullet$]  if $\alpha \in S^+_\zeta$ and
$\otp(C^{\xi(\zeta)}_\alpha \cap E_*) \in C^*_\zeta
  \cup\{\delta_2(\zeta)\}$ then $C_{\zeta,\alpha} = \{\beta \in 
C^{\xi(\zeta)}_\alpha \cap E_*:\otp(C^{\xi(\zeta)}_\beta \cap E_*) \in
C^*_\zeta\}$
\sn
\item[$\bullet$]  $S^*_\zeta = \{\alpha \in
  S^+_\zeta:\otp(C_{\zeta,\alpha}) = \varepsilon_\zeta(*)\}$.
\end{enumerate}
\mn
Now check.
\end{PROOF}

\begin{remark}\label{7.8G}
    We may start with a partial square $\langle C^1_{\delta}:\delta \in S^1\rangle,S^1
    \subseteq \mu$ such that: $C_{\delta}$ is a club of $\delta$, and
    \begin{equation*}
        \begin{array}{clcr}
        S^2 =: \{\delta \in S^1:&\{\alpha:\alpha \in  C^1_{\delta} \cap S^1 
        \text{ and } \cf(\alpha) = \kappa\} \\
          &\text{ is a stationary subset of } \delta \}
        \end{array}
    \end{equation*}
    
    is a stationary subset of $\mu,\mu = \cf(\mu) > \epsilon(*),
    \cf(\epsilon(*)) = \kappa$ and find $S \subseteq \{\delta \in S^1:
    \cf(\delta) = \kappa\}$ stationary in $\mu,\bar C =
    \langle C_\delta:\delta \in S \rangle,C_{\delta}$ a club of $\delta$
    of order type $\epsilon(*)$, such that for every club $E$ of 
    $\mu$ for some  $\delta \in S,C_{\delta} \subseteq E$.
    See also \cite[\S2]{Sh:E59}, \cite[\S3]{Sh:413}.
\end{remark}

\begin{claim}\label{7.8H}
    Suppose $\lambda$ is regular, $\langle C_{\alpha}:\alpha \in
    S\rangle$ is a partial square ($S \subseteq \lambda$ stationary), $\kappa =
    \cf(\kappa) < \lambda,\varepsilon(*) < \lambda$ and
    \[
    S_1 \subseteq \{\delta \in S:\cf(\delta) = \kappa, \text{ and }
    \otp(C_{\delta}) < \delta\} \text{  is stationary.}
    \]
    \Then \, we can find $S_{2}$ and $E$ such that:
    \mn
    \begin{enumerate}
    \item[$(i)$]  $S_{2} \subseteq S_{1}, S_{2}$ stationary in $\lambda$, $E$ a club of $\lambda$, 
    \sn
    \item[$(ii)$]  for some $\epsilon(*)$ for all $\delta \in S_{2},
    \otp(C_{\delta}) = \epsilon(*)$
    \sn
    \item[$(iii)$]  $\langle C_{\delta} \cap E:\delta \in S \cap E,\delta = \sup
    C_{\delta} \cap E\rangle$ satisfies $(a) + (c)$ of \ref{7.8E}
    \sn
    \item[$(iv)$]  letting
    
    \[
    C'_{\delta}  = \begin{cases} C_\delta \cap E  &\text{ if } C_\delta\cap
    S_2 = \emptyset \\
       C_\delta \cap E \setminus [\min(S_2\cap C_\delta)+1]  &\text{ if } 
    C_\delta \cap S_2 \ne \emptyset
      \end{cases}
    \]
    
    \mn
    we have $\langle C'_{\delta}:\delta \in S \cap E\rangle$ is a partial square.
    \end{enumerate}
\end{claim}

\begin{PROOF}{\ref{7.8H}}
Straightforward if you read \ref{7.8A}--\ref{7.8G}.
\end{PROOF}

\noindent
We now go back to bigness properties, first an easy improvement of \ref{7.6},
and then to the promises from the beginning of this section.
\begin{claim}
\label{7.8I}
If $\lambda=\cf(\lambda)>\mu+\aleph_1$ \then \, $K^\omega_{\tr}$ has the full 
$(\lambda,\lambda,\mu,\mu)$-super$^{7^+}$- bigness property.
\end{claim}

\begin{PROOF}{\ref{7.8I}}
For each stationary $S\subseteq \{\delta<\lambda:\cf(\delta)=
\aleph_0\}$ let $\langle C^S_\delta:\delta\in S\rangle$ be as in \ref{7.8A}
with $\kappa=\aleph_0$. Now repeat the proof of \ref{7.6}(1) only now, for 
$\delta\in S_\zeta$ the sequence $\eta_\delta$ list the set 
$C^{S_\zeta}_\delta$ in increasing order. See also the proof of case 1
in \ref{7.14}.
\end{PROOF}

\begin{remark}\label{7.8I1}
    Note: to define partial square on a club of $\lambda$ or on the set of all
    limit ordinals, usually makes minor difference (only for non-Mahlo $\lambda$,
    limit of inaccessible, we can get $\otp(C_{\delta}) < \delta$ more
    easily in the first case).
\end{remark}

\begin{PROOF}{\ref{7.8I}}
\underline{Proof of \ref{7.8}}:  
We can find $\lambda_{1}$ a successor of regular cardinal satisfying
$\mu + \chi^{+} < \lambda_{1} \le \lambda$ and $\chi^+ < \lambda_{1}$ (just
let $\lambda_1 = \mu^+ + \chi^{++}$ if $\mu$ is regular and let
$\lambda_1 = \mu^{++} + \chi^{++}$ if $\mu$ is singular, hence $\lambda_{1} \geq \mu^{+} \Rightarrow \lambda \geq \mu^{++}$). 

Also \wilog \, $\cf(\chi)=\aleph_{0}$.

\noindent
[Why?  As letting $\chi_1 = \min\{\chi_0:\chi_0 \ge \aleph_{0}$ 
and $\chi^{\aleph_{0}}_{0} = \chi^{\aleph_{0}}\}$,  we have 
$\chi_{1} \le \chi,\chi^{\aleph_{0}}_{1} = \chi^{\aleph_{0}}, \cf(\chi_{1}) = \aleph_{0}$ and: $(\forall \alpha <
\chi_{1})[|\alpha |^{\aleph_{0}} < \chi_{1}]$ or $\chi_{1} = \aleph_{0}$ (Notice that, instead changing $\chi$ we can use below in clause (a) the ordinal $\chi \times\omega)$.]

By Fact \ref{7.8E} there are a stationary  set $S \subseteq
\{\delta < \lambda_{1}:\cf(\delta) = \aleph_{0}\}$  and a sequence 
$\langle C_{\delta}:\delta \in S\rangle$ such that:

\begin{enumerate}
\item[$(*)_1$]  $(a) \quad C_{\delta}$ is a club of $\delta$ of order type $\chi$,

\item[${{}}$]  $(b) \quad$ for every club $E$ of $\lambda_{1}$ for stationary many $\delta \in S,C_{\delta} \subseteq E$.
\end{enumerate}

For any $\rho \in {}^\omega \chi$ we define

\begin{enumerate}
\item[$(*)_2$]  $I_{\rho} = {}^{\omega >}\lambda \cup 
\{\rho^{[\delta]}:\delta \in S\}$
\end{enumerate}
\mn
where $\rho^{[\delta]} \in {}^\omega(\lambda_{1})$ is defined by
$\rho^{[\delta]}(n) = $  the $\rho(n)$-th member of $C_{\delta}$.

Easily there are $\Upsilon,\bar\chi$ such that
\mn
\begin{enumerate}
\item[$(*)_3$]  $(a) \quad \Upsilon \subseteq {}^\omega \chi$ 
have cardinality $\chi^{\aleph_{0}}$,
\sn
\item[${{}}$]  $(b) \quad$ each $\rho \in \Upsilon$ is increasing with 
limit $\chi$
\sn
\item[${{}}$]  $(c) \quad$ for $\rho_{1} \ne \rho_{2}$ from
$\Upsilon,\Rang(\rho_{1}) \cap \Rang(\rho_{2})$ is finite
\sn
\item[${{}}$]  $(d) \quad \bar \chi = \langle \chi_n:n < \omega\rangle$ is a strictly increasing sequence of $\chi$-s,
\sn
\item[${{}}$]  $(e) \quad \chi = \bigcup\limits_{n < \omega} \chi_n$
\sn
\item[${{}}$]  $(f) \quad \rho \in \Upsilon \Rightarrow \rho(n) \in (\chi_{n}, \chi_{n+1})$.
\end{enumerate}

We shall show that $\{I_{\rho}:\rho \in \Upsilon\}$ exemplifies the desired conclusion: the full $(\chi^{\aleph_{0}},\lambda,\mu,\mu)$-super-bigness property.

Suppose $\rho \in \Upsilon, J = J_{\rho}  \coloneqq \sum \{I_{\nu}:\nu \in \Upsilon 
\setminus \{\rho\}\}$, for example let $\Upsilon = \{\rho_{i}:i < 
|\Upsilon|\}$  and $J = \{\langle \rangle\} \cup \{\langle \zeta
\rangle \mathop{\otimes}\limits_{\lambda} \nu:\zeta < |\Upsilon|,\rho_{\zeta}
\ne \rho$ and $\nu \in I_{\rho_{i}}\}$, where: 

\begin{enumerate}
    \item[$\boxplus$] for $\rho$ a sequence of ordinals and $\zeta <
    \lambda$ let $\langle \zeta \rangle \underset{\lambda}{\otimes}
    \rho$ or $\zeta \mathop{\otimes}\limits_\lambda \rho$ be the sequence 
    $\rho'$ of length $\ell g(\rho),\rho'(0)=\lambda \times \zeta+\rho(0),
    \rho'(1+\gamma)=\rho(1+\gamma)$.
\end{enumerate}

Clearly

\begin{enumerate}
    \item[$\oplus$] It suffice to show that $I_{\rho} \in K_{\tr}^{\omega}$ is $(\mu, \kappa)$-super-unembedabble into $J_{\rho} \in K_{\tr}^{\omega}$.  Let $\chi^*$ be regular large enough and $<^*$ a well ordering of  
    ${\cH}(\chi^*)$ and $x \in \cH(\chi^{\ast})$. 
\end{enumerate}

We choose by induction on $\alpha < \lambda_{1},M^*_{\alpha}$ such that:

\begin{enumerate}
    \item[$(*)_4$]  $(a) \quad M^*_\alpha \prec ({\cH}(\chi^*),
    \in,<^*_{\chi^*})$, 
    \sn
    \item[${{}}$]  $(b) \quad M^*_{\alpha}$ is increasing continuous,
    \sn
    \item[${{}}$]  $(c) \quad \|M^*_{\alpha}\| < \lambda_{1}$,
    \sn
    \item[${{}}$]  $(d) \quad M^*_{\alpha} \cap \lambda_{1}$ is an ordinal,
    \sn
    \item[${{}}$]  $(e) \quad \langle M^*_{\beta}:\beta \le \alpha \rangle \in M^*_{\alpha +1}$,
    \sn
    \item[${{}}$]  $(f) \quad \mu+1$ is a subset of $M^*_{0}$,
    \sn
    \item[${{}}$]  $(g) \quad \mu,I_{\rho},x,J =
    \sum\limits_{\nu \in \Upsilon \setminus \{\rho\}} I_{\nu}$ belong to $M^*_{0}$.
\end{enumerate}

Let

\begin{enumerate}
\item[$(*)_5$]  $E \coloneqq \{\delta < \lambda_{1}:M^*_{\delta} \cap \lambda_{1} =
\delta\}$.
\end{enumerate}

Clearly $E$ is a club of $\lambda_{1}$. So, by the choice of $\langle C_{\delta} :\delta \in S\rangle$, for some $\delta(*) \in  S$ we have $C_{\delta(*)} \subseteq E$.

We shall show that $\eta := \rho^{[\delta(*)]},M_{n} := M^*_{\eta(n)},
N_{n} := M^*_{\eta (n)+1}$ are as required in Definition \ref{7.1}.

Note:

\begin{enumerate}
\item[$(*)_6$]  $\eta(n)+1 \le \chi_{n+1} < \eta(n+1)$.
\end{enumerate}

Hence,

\begin{enumerate}
\item[$(*)_7$]  $(a) \quad M_{n} \in N_{n}$, 
\sn
\item[${{}}$]  $(b) \quad N_{n} \in M_{n+1}$, 
\sn
\item[${{}}$]  $(c) \quad \eta \rest n \in M_n$, 
\sn
\item[${{}}$]  $(d) \quad \eta \rest (n+1) \notin N_n$.
\end{enumerate}

[E.g. Why clause (c)?  It suffices to prove $\ell < n \Rightarrow \eta(\ell)
\in M_n$ because $\eta(\ell) < \eta(n) = \lambda_1 \cap M_{n} \subseteq M^*_{\eta(n)} = M_n$ recalling $C_{\delta(*)} \subseteq E$ and the definition of $E$. Why clause (d) holds? By $(\ast)_{6}$.]

Of course,

\begin{enumerate}
    \item[$(*)_8$]  $\mu \subseteq M_{n} \prec N_{n} \prec M_{n+1} \prec
     ({\cH}(\chi^*),\in,<^*_{\chi^*})$.
\end{enumerate}

So clearly clauses (i)-(v) of $(*)$ of \ref{7.1} holds and we 
are left with proving clause (vi). 

Let $\nu \in P^J_{\omega}$, and choose an ordinal $\alpha =
\max\{\alpha_1,\alpha_2,\alpha_3\} < \delta(*)$ where:

\begin{enumerate}
\item[$(*)_9$]  $(a) \quad$ if $\nu \in M^*_{\delta(*)}$ then $\alpha_1 <
\delta(*)$ is such that $\nu \in M_{\alpha_{1}}^*$,

\item[${{}}$]  $(b) \quad$ if for some $m < \ell g(\nu),\nu \rest m \in
M_{\delta(*)},\nu \rest (m+1) \notin  M_{\delta(*)}$, then $\alpha_2 <
\delta(*)$

\hskip25pt  is large enough such that $\nu \rest m \in M_{\alpha_2}^*$,
\sn
\item[${{}}$] $(c) \quad$ if $\nu = \langle i \rangle 
\mathop{\otimes}\limits_\lambda \rho^{[\delta(*)]}_{i}$ (so $\rho_{i} \ne
\rho)$, let $\alpha_3 \in C_{\delta(*)}$ be such that 
$(\Rang(\rho)) \cap$

\hskip25pt $\Rang(\rho_{i}) \subseteq \alpha_3 
\cap C_{\delta(*)}$  (exists by the choice of $\Upsilon$).
\end{enumerate}
\mn
It is easy to check that every $n < \omega$ such that
$\rho^{[\delta(*)]}(n) > \alpha$ is as required in Definition 
\ref{7.1}(vi), but every large enough $n < \omega$ is like that 
by the choice of $\alpha$.

So indeed $I_{\rho} \in K_{\tr}^{\omega}$ is $(\mu, \kappa)$-super-unembeddable into $J_{\rho} \in K_{\tr}^{\omega}$.

So we have proved \ref{7.8}(1). 

For proving \ref{7.8}(2) let $M'_{2n} = M_{n},M'_{2n+1} = N_{n}$.

As for proving \ref{7.8}(3), (using again $M'_n$) make the following changes.  \underline{First} in proving \ref{7.8E} we can guarantee
\[
[\sup(C_{\delta} \cap \alpha) < \alpha \in
C_{\delta} \Rightarrow \cf (\alpha) > \aleph_{0}],
\]

(apply e.g. \ref{7.8E} to $\omega_1 \times \epsilon(*)$ getting
$C'_\delta$, and let
\[
C_\delta = \{\zeta\in C_\delta:\otp(C_\delta \cap \zeta) 
\text{ divisible by } \omega_1\}).
\]

Second choosing $\Upsilon$ in $(\ast)_{3}$ above (in the proof of \ref{7.8}(11)), we can demand:
\[
\eta \in \Upsilon \Rightarrow  \Rang(\eta) 
\text{ consists of successor ordinals only).}
\]

Then the requirement holds --- check. 
\end{PROOF}

\begin{lemma}\label{7.9}
Suppose $\lambda$ is singular, $\lambda > \mu,\lambda >
\theta > \cf(\theta) = \aleph_{0},\theta \ge \mu + \cf(\lambda),
\ga_{\epsilon}$ for $\epsilon  < \cf(\lambda)$  is a set of regular uncountable
cardinals, $\omega = \otp(\ga_{\epsilon}),\theta = \sup(\ga_{\epsilon})$,
they are pairwise almost disjoint (i.e. for $\epsilon < \zeta < \cf(\lambda)$,
$\ga_{\epsilon} \cap \ga_{\zeta}$ is finite) and 
$\max \pcf_{J^{bd}_{\ga_\epsilon}}(\ga_{\epsilon}) = \theta^+ < \lambda$, 
see Definition \cite[3.16=Lprf.2]{Sh:E62}.

\Then \, $K^\omega_{\tr}$ has the full 
$(\lambda,\lambda,\mu,\mu)$-super-bigness property.
\end{lemma}

\begin{remark}\label{7.9f}
    We shall repeat this proof with some changes in \ref{7.14} case 3.
\end{remark}

\begin{PROOF}{\ref{7.9}}  
Let $\langle \mu_{\epsilon}:\epsilon < \cf(\lambda)\rangle$ be a strictly increasing sequence of regular cardinals such that
$\sum\limits_{\epsilon < \cf(\lambda)}\mu_{\epsilon} = \lambda,$ and $\mu + \theta^+ < \mu_{\epsilon}$. For $\epsilon < \cf(\lambda)$, let $\lambda_{\epsilon} =\mu^{+3}_{\epsilon}$.

We now apply~\ref{7.8F}(3) with $\lambda_{\varepsilon}$, $\langle \theta \colon \zeta < \lambda \rangle$, $\langle \aleph_{0} \colon \zeta < \lambda_{\varepsilon} \rangle$ here standing for $\lambda$, $\langle \varepsilon_{\zeta}(\ast) \colon \zeta < \lambda \rangle$, $\langle \kappa_{\zeta} \colon \zeta < \lambda \rangle$ there. Let us get $\bar{S}_{\varepsilon} = \langle S_{\epsilon, \zeta} \colon \zeta < \lambda_{\epsilon} \rangle$, $\bar{C}_{\varepsilon} = \langle C_{\delta}^{\epsilon} \colon \delta \in S_{\epsilon} \rangle$ here standing for $\langle S_{\zeta} \colon \zeta < \lambda \rangle$, $\langle \rho_{\delta}^{\zeta} \colon \delta \in S_{1}, \zeta < \lambda \rangle$ there, hence: 

\begin{enumerate}
\item[$(*)_1$]  $(i) \quad \langle S_{\epsilon,\zeta}:\zeta 
< \lambda_{\epsilon}\rangle$ is a partition of $S_{\varepsilon}$,
\sn
\item[${{}}$]  $(ii) \quad S_{\epsilon,\zeta} \subseteq
\{\delta < \lambda_{\epsilon}:\cf(\delta) = \aleph_{0}\}$ are stationary
subsets of $\lambda_{\epsilon}$,
\sn
\item[${{}}$]  $(iii) \quad$ if $\delta \in S_{\epsilon}$ then
$C^\epsilon_{\delta}$ is a club of $\delta $ and
$C^\epsilon_{\delta}$ has order type $\theta$

\hskip30pt (recall that $\cf(\theta) = \aleph_0$ by the claim's assumptions and

\hskip30pt $\delta \in S_{\varepsilon,\zeta}$ implies $  \cf(\delta) = \aleph_0$, so $\delta$ is divisible by $\theta$),
\sn
\item[${{}}$]  $(iv) \quad$ for every club $E$ of $\lambda_{\epsilon}$
and $\zeta < \lambda_{\epsilon}$,  the set $\{\delta \in 
S_{\epsilon,\zeta}:C^\epsilon_{\delta} \subseteq E\}$

\hskip30pt  is stationary
\sn
\item[${{}}$]  $(v) \quad \langle C^\epsilon_{\delta}:\delta \in S_{\epsilon}
\rangle$ is a partial square.
\end{enumerate}

For transparency, $S_{\epsilon}$ is disjoint to 
$\bigcup\limits_{\xi < \epsilon} \lambda_{\xi}$ and let $\langle \kappa_{\varepsilon, n} \colon n < \omega \rangle$ list $\mathfrak{a}_{\varepsilon}$ in increasing order. 

For each $\epsilon < \cf(\lambda)$ we can find $\langle
\rho_{\epsilon ,i}:i < \theta ^+\rangle$ such that:
\mn
\begin{enumerate}
\item[$(*)_2$]  $(i) \quad \rho_{\epsilon,i} \in \Pi \ga_{\epsilon}$ is
(strictly) increasing
\sn
\item[${{}}$]  $(ii) \quad i < j < \theta^+ \Rightarrow \rho_{\epsilon,i}
<_{J^{\bd}_{\ga_{\epsilon}}} \rho_{\epsilon,j}$;
(i.e. for every large enough 

\hskip30pt $\kappa \in \ga_{\epsilon},
\rho_{\epsilon,i}(\kappa) < \rho_{\epsilon,j}(\kappa)<\kappa)$,
\sn
\item[${{}}$]  $(iii) \quad$ for every $\rho \in \Pi \ga_\epsilon$ for
some $i < \theta^+$ we have $\rho <_{J^{\bd}_{\ga_{\epsilon}}}
\rho_{\epsilon,i}$
\sn
\item[${{}}$]  $(iv) \quad \rho_{\epsilon,i}(k_{\varepsilon, n})$ is a limit 
ordinal of uncountable cofinality
\sn
\item[${{}}$]  $(v) \quad \rho_{\epsilon,i}(\kappa) >
\sup(\ga_\varepsilon \cap \kappa)$ hence $\theta =
\cup\{\rho_{\epsilon,i}(\kappa):\kappa \in \ga_\varepsilon\}$.
\end{enumerate}

Let $\Upsilon_{\epsilon} := \{\rho_{\epsilon,i}:i < \theta^+\}$.

For $\epsilon < \cf(\lambda)$, and $\zeta$ such that
$\bigcup\limits_{{\xi} < {\epsilon}} \lambda_{\xi} \le \zeta
< \lambda_{\epsilon}$ let $I_{\zeta} = {}^{\omega >}\lambda \cup 
\{\rho^{[\delta]}:\rho \in 
\boldsymbol{\Upsilon}_{\epsilon},\delta \in S_{\epsilon,\zeta}\}$
recalling $\rho^{[\delta]}$ is an $\omega$-sequence of ordinals:
$\rho^{[\delta]}(n)=$ the $\rho(\kappa_{\varepsilon, n})$-th element of $C^\epsilon_{\delta}$
(now $\rho^{[\delta]}$ depend on $C^\epsilon_\delta$ and $\rho$ so on
$\delta,\epsilon,\rho$, but 
$\epsilon$ can be reconstructed from $\delta$, as $S_\epsilon
\subseteq [\bigcup\limits_{\xi< \epsilon} \lambda_\xi, \lambda_\epsilon)$).

We shall show that $\langle I_{\zeta}:\zeta < \lambda \rangle$
exemplify the desired conclusion, this suffices.

\noindent
So let $\epsilon(*) < \cf(\lambda),
\bigcup\limits_{{\xi <} \epsilon{(*)}} \lambda_{\xi} \le \zeta
(*) < \lambda_{\epsilon(*)},\chi^*$ regular large enough and 
$x \in {\cH}(\chi^*)$, and let $J =
\sum\limits_{\xi <\lambda,\xi \ne \zeta(*)} I_{\xi}$.  We can choose by
induction on $\alpha<\lambda_{\epsilon(*)}$ 
a model $M^*_{\alpha}$ such that:
\begin{enumerate}
\item[$(*)_3$]  $(a) \quad M^*_{\alpha} \prec ({\cH}(\chi^*),\in,
<^*_{\chi^*})$
\sn
\item[${{}}$]  $(b) \quad M^*_{\alpha}$ increasing continuous in
$\alpha$
\sn
\item[${{}}$]  $(c) \quad \langle M^*_{\beta}:\beta \le \alpha \rangle \in
M^*_{\alpha +1}$
\sn
\item[${{}}$]  $(d) \quad \|M^*_{\alpha}\| < \lambda_{\epsilon(*)}$
\sn
\item[${{}}$]  $(e) \quad M^*_{\alpha} \cap \lambda_{\epsilon(*)}$ is an
ordinal $> \mu^{+2}_\epsilon > \mu + \cf(\lambda) + 
\sum\limits_{\zeta < \epsilon(*)} \lambda_\zeta$
\sn
\item[${{}}$]  $(f) \quad$ the objects $I_{\varepsilon(*)},J$ and
$\langle < \rho_{\epsilon,j}:j < \theta^+>:\epsilon < \cf(\lambda)
\rangle,\epsilon(*),\langle \lambda_{\epsilon}:\epsilon <$

\hskip30pt $\cf(\lambda) \rangle,\langle < S_{\epsilon,\zeta}:
\zeta < \lambda_{\epsilon} >:\epsilon < \cf(\lambda)\rangle$ and 

\hskip30pt $\langle < C^\epsilon_{\delta}:\delta \in S_{\epsilon} >:
\epsilon < \cf(\lambda)\rangle$ belong to $M^*_{\alpha}$.
\end{enumerate}
\mn
Let $E = \{\delta < \lambda_{\epsilon(*)}:M^*_{\delta} \cap
\lambda_{\epsilon(*)} = \delta\}$;  clearly $E$ is a club of
$\lambda_{\epsilon(*)} $.  So for some $\delta(*) \in E \cap
S_{\epsilon(*),\zeta(*)}$,  we have $C^{\epsilon(*)}_{\delta(*)} 
\subseteq E$.

We can find $N \prec M^*_{\delta(*)}$ such that:
\begin{enumerate}
\item[$(*)_4$]  $(\alpha) \quad \|N\| = \theta$, $\theta +1 \subseteq
N$ hence $\mu+1 \subseteq N,\{\mu,\kappa\} \subseteq N$, 
and $C^{\epsilon(*)}_{\delta(*)} \subseteq N$;  
\sn
\item[${{}}$]  $(\beta) \quad$ if $\delta \in M^*_{\delta(*)},
\cf(\delta) = \aleph_{0},\delta = \sup(N \cap \delta)$ then $\delta
\in N$;
\sn
\item[${{}}$]  $(\gamma) \quad$ the following objects belong to $N$ 
\sn
\begin{enumerate}
\item[${{}}$]  $\bullet \quad \langle <\rho_{\epsilon,j}:j <
\theta^+>:\epsilon < \cf(\lambda)\rangle$,
\sn
\item[${{}}$]  $\bullet \quad I_{\zeta(*)},J,x$,
\sn
\item[${{}}$]  $\bullet \quad \epsilon(*),\langle
\lambda_{\epsilon}:\epsilon < \cf(\lambda)\rangle$,
\sn
\item[${{}}$]  $\bullet \quad \langle <S_{\epsilon,\zeta}:\zeta <
\lambda_{\epsilon} >:\epsilon < \cf(\lambda) \rangle$,
\sn
\item[${{}}$]  $\bullet \quad \langle <C^\epsilon_{\delta}:
\delta \in S_{\epsilon} >:\epsilon < \cf(\lambda)\rangle$
\end{enumerate}
\sn
\item[${{}}$]  $(\delta) \quad \langle M^*_{\alpha}:\alpha < \gamma 
\rangle \in N$ for $\gamma \in C^{\epsilon(*)}_{\delta(*)}$.
\end{enumerate}

Let

\begin{enumerate}
\item[$\boxplus_1$]  $(a) \quad W \coloneqq \{(\epsilon,\zeta,\delta):
\epsilon < \cf(\lambda),\zeta \ne \zeta(*),
\bigcup\limits_{j < \epsilon} \lambda_j \le \zeta < 
\lambda_{\epsilon},\delta \in S_{\epsilon,\zeta}$

\hskip30pt  and $\zeta \in N,
\delta = \sup(N \cap \delta) \notin N$ but $C^\epsilon_\delta
\subseteq N\}$,
\sn
\item[${{}}$]  $(b) \quad W_1 \coloneqq \{(\epsilon,\zeta,\delta) \in
W:\epsilon > \epsilon(*)\}$.
\end{enumerate}
\mn
It is enough to show that for some $\rho \in 
\boldsymbol{\Upsilon}_{\epsilon(*)}$ we have:
\mn
\begin{enumerate}
\item[$\boxplus_{2,\rho}$]  $\eta_\rho := \rho^{[\delta(*)]},
M^\rho_{n} := N \cap M^*_{\rho^{[\delta(*)]}(n)},N^\rho_{n} := N \cap
M^*_{\rho^{[\delta(*)]}(n)+1}$
\end{enumerate}
\mn
(for $n < \omega$) satisfy the requirement $(*)$ of Definition
\ref{7.1}.  

Now, for every $\rho \in \boldsymbol{\Upsilon}_{\epsilon(*)}$, consider the conditions 
(from $(*)$ of \ref{7.1}):
\mn
\begin{enumerate}
\item[$\bullet$]  $(i),(iii),(iv)$ are trivial
\sn
\item[$\bullet$]  $(v)$ holds by the definition, in fact for every
$n, \eta_\delta \rest \Upsilon \in M^\ell_n,\eta_\delta \rest (n+1) \in
N^\rho_n \backslash M^\rho_n$
\sn
\item[$\bullet$]  $(ii)$ holds as $\mu + 1 \subseteq M_{0}$ 
because $\mu \le \theta \subseteq N$.
\end{enumerate}

The main point is condition (vi) and we shall show that for some $\rho
\in \boldsymbol{\Upsilon}_{\epsilon(*)}$ it holds

\begin{enumerate}
\item[$\boxplus_3$]  let $\Lambda = \{\rho \in
\Upsilon_{\epsilon(*)} \colon$  clause (vi) of Definition \ref{7.1} fails
for $\eta_\rho = \rho^{[\delta(*)]},M^\rho_n,N^\rho_n(n < \omega)\}$
\sn
\item[$\boxplus_4$]  if $\rho \in \Lambda$, then let $\Lambda_\rho$ be the set of $\nu
\in P^J_\omega$ such that:
$\{\nu \rest \ell:\ell < \omega\} \subseteq N$ but for no  
$\alpha < \delta(*)$ do we have $\{\nu \rest \ell:\ell < \omega\}
\subseteq N \cap M^*_{\alpha}$ and for infinitely many $n$ for some $k$ 
we have $\nu \rest k \in M^\rho_n,\nu(k) \in N^\rho_n \setminus
M^\rho_n$. 

(Recall that $N$ is from $(\ast)_{4}$).

\sn
\item[$\boxplus_5$]  it suffice to find $\rho \in \Lambda$ such that $\Lambda_{\rho} = \emptyset$, so towards a contradiction, assume that: 

\begin{itemize}
    \item if $\rho \in \Lambda$, 
    then we choose $(\nu,\varrho,\epsilon,\zeta,\delta) =
    (\nu_\rho,\varrho_\rho,\epsilon_\rho,\zeta_\rho,\delta_\rho)$ such
    that (but if $\rho$ is clear from the context, we may omit the
    subscript $\rho$): 

    \begin{enumerate}
        \item[(a)] $\nu_{\rho} \in \Lambda_{\rho}$, 

        \item[(b)] $\zeta_{\rho} \in \lambda \setminus \{ \zeta(\ast) \}$, 

        \item[(c)] $\varepsilon_{\rho}$ is the unique $\varepsilon < \cf(\lambda)$ such that $\bigcup_{\xi < \epsilon} \lambda_{\xi} \leq \zeta_{\rho} < \lambda_{\epsilon}$,

        \item[(d)] $\delta_{\rho} = \delta$, 

        \item[(e)] $\varrho_{\rho} \in \Lambda_{\varepsilon_{\rho}}$ is such that $\nu_{\rho} = \varrho^{[\delta(\ast)]}$, see after $(\ast)_{2}$.
    \end{enumerate}
\end{itemize}

    [Why such a quintuple exists? If no such $\nu_{\rho} \in \Lambda_{\rho}$ exists, then we get the desired contradiction. By the definition of $\Lambda_{\rho}$ there is $\varrho$ such that $\nu = \langle \zeta \rangle \mathop{\otimes}\limits_\lambda
    \varrho^{[\delta]}$ 
    (if we use the first version in the proof of \ref{7.8I}, or $\langle
    \zeta\rangle \char94 \varrho$ if we use another one there) and 
    $\varrho \in \boldsymbol{\Upsilon}_{\epsilon},\delta \in S_{\epsilon,\zeta},
    \bigcup\limits_{\xi < \epsilon} \lambda_{\xi} <
    \zeta < \lambda_{\epsilon},\zeta \ne \zeta(*)$;  hence \wilog \,  
    $(\nu,\varrho,\epsilon,\zeta,\delta) =
    (\nu_{\rho},\varrho_{\rho},\epsilon_{\rho},\zeta_{\rho},
    \delta_{\rho})$. ]

\end{enumerate}
\mn
 
\mn
\begin{enumerate}
\item[$\boxplus_6$]  if $\rho \in \Lambda$ then
$\epsilon < \epsilon(*)$ is impossible.
\end{enumerate}
\mn
Why? In this case $\lambda_{\epsilon} \subseteq M^*_0$ (see condition
$(*)_3(e)$ on the $M^*_\alpha$'s, hence $N \cap \{\nu \rest \ell:\ell
< \omega\} \subseteq M^*_{0}$, contradiction.

Next, 

\begin{fact}\label{7.9R7}
    We have: 
    \begin{enumerate}
        \item[$\boxplus_7$]  if $\rho \in \Lambda$ then $\epsilon = \epsilon(*)$ is impossible.
    \end{enumerate}
\end{fact}

\begin{PROOF}{\ref{7.9R7}}
    As $\nu \in J$ necessarily $\zeta \ne \zeta(*)$. As $\delta \in S_{\epsilon,\zeta}$, clearly $S_{\epsilon,\zeta} \cap
    S_{\epsilon(*),\zeta(*)} = \emptyset$ so necessarily  $\delta \ne \delta(*)$.  If  $\delta > \delta(*)$,  as $\langle \nu(n):1
    \le n < \omega \rangle$ is strictly increasing with limit $\delta$, for
    some $n,\lambda_{\epsilon(*)} > \nu(n) > \delta(*)$ hence
    $\nu \rest (n+1) \notin  M^*_{\delta(*)}$ hence  $\nu \rest (n+1) \notin N$,  contradiction.  If $\delta < \delta(*)$ then for some
    $\alpha < \delta(*),\{\nu \rest \ell:\ell < \omega\} \subseteq 
    N \cap M^*_\alpha$, (remember that $\theta \subseteq N$ by $(*)_4(\alpha)$
    and  $\{\nu \rest \ell:\ell < \omega\} \subseteq N$ by the assumption on $\nu$); again impossible by $\boxplus_{4}$, so $\boxplus_7$ holds. 
\end{PROOF}

\begin{subf}\label{7.9R8}
    We have: 
    \begin{enumerate}
        \item[$\boxplus_8$]  $(\varepsilon, \zeta, \delta) \in W_1$.
    \end{enumerate}
\end{subf}

\begin{PROOF}{\ref{7.9R8}}
    By $\boxplus_6,\boxplus_7$ we have $\epsilon > \epsilon(*)$. Now (remembering $\bar C^\epsilon$ is a partial square), for $1 \le 
    n < m < \omega,C^\epsilon_{\nu(n)} = C^\epsilon_{\nu(m)} \cap \nu(n)$, 
    and as $\nu(n) \in N$ by $(*)_4(\gamma)$ necessarily $C^\epsilon_{\nu(n)} \in N$, so as $\theta \subseteq N \wedge
    |C^\epsilon_{\nu(n)}| \le |C^{\epsilon(*)}_\delta| = \theta$ clearly
    $C^\epsilon_{\nu(n)} \subseteq N$.  

    It follows that $C_{\delta}^{\epsilon} = \bigcup_{1 < n < \omega} C_{\nu(n)}^{\epsilon}$, hence $C_{\delta}^{\epsilon} \subseteq N$, so $\delta = \sup(\delta \cap N)$. Next, 
    
    \begin{subf}\label{7.9R.5}
        $\delta \notin N$.
    \end{subf}
    
    \begin{PROOF}{\ref{7.9R.5}}
        As otherwise for some $\alpha < \delta(*),\delta \in M^*_\alpha$, hence $C^\epsilon_{\delta} \subseteq M^*_\alpha$;   now from $\nu \rest 1 \in N$ it follows that $\zeta \in N$ but
        $\epsilon < \cf(\lambda) \subseteq \theta \subseteq N$ so also $\epsilon \in N$ and $\Upsilon_\epsilon \in N$. Hence $\Lambda = \{\langle \gamma,\eta,\eta^{[\gamma]}\rangle:\gamma \in  S_{\epsilon,\zeta}$ and $\eta \in \boldsymbol{\Upsilon}_{\epsilon}\} \in N$ but we are assuming $\delta \in N$, and  hence $\{\nu(n) \colon n < \omega\} \subseteq C_{\delta}^{\varepsilon} \subseteq M^*_\alpha$ but also $\{ \nu(n) \colon n < \omega \} \subseteq N$ by $\oplus_{4}$, hence $\{\nu(n) \colon n < \omega\} \subseteq N \cap M_\alpha$, therefore by $\oplus_{2, \rho}$ is $\subseteq N_n$
        for $n$ large enough, which contradicts $\oplus_{4}$.  So really $\delta \notin N$]. 
    \end{PROOF}
    
    \underline{Continuing the proof of \ref{7.9R8}}: By clause $(*)_4(\beta)$ in the choice of $N$ necessarily 
    $\delta \notin M_{\delta(*)}$ and recalling $W$ is defined in
    $\boxplus_1$ above clearly $(\epsilon,\zeta,\delta)\in W$.
    
    Clearly $(\epsilon,\zeta,\delta) \in  W_{1}$ as we have 
    shown $\epsilon>\epsilon(*)$ by $\boxplus_6 + \boxplus_7$, so
    $\boxplus_8$ holds indeed. 
\end{PROOF}

Note that
\mn
\begin{enumerate}
\item[$\boxplus_9$]   $|W_{1}| \le |W| \le \theta$.
\end{enumerate}

This is because 

\begin{enumerate}
    \item[$\bullet_{1}$] the set $W_{1, 1}$ has cardinality $\leq \theta$, where $W_{1, 1} \coloneqq \{\epsilon \colon$ for some $\zeta, \delta$ we have $(\epsilon, \zeta, \delta) \in W_{1} \}$ because $\epsilon \in W_{1, 1} \Rightarrow \epsilon \leq \cf(\lambda) \leq \theta$. 

    \item[$\bullet_{2}$] the set $W_{1, 2}$ has cardinality $\leq \theta$, where $W_{1, 2} = \{ \zeta \colon$ for some $\epsilon, \delta$ we have $(\epsilon, \zeta, \delta) \in W_{1}\}$ because $\zeta \in W_{1, 2} \Rightarrow \zeta \in N$ so $\vert W_{1, 2} \vert \leq \Vert N \Vert = \theta$. 

    \item[$\bullet_{3}$] the set $W_{1, 3}$ has cardinality $\leq \theta$, where $W_{1, 3} \coloneqq \{ \delta  \colon$ for some $\epsilon, \zeta$ we have $(\epsilon, \zeta, \delta) \in W_{1} \}$ because $\|N\| = \theta$, and a well ordering of cardinality $\le \theta$ has 
    $\le \theta$ Dedekind cuts and $\delta = \sup(\delta \cap N) > 
    \sup(\delta \cap M_{\alpha})$ for $\alpha < \delta$  
    (see $(*)_4(\beta)$ in choice of $N$).
    
    By $\bullet_{1}$, $\bullet_{2}$ and $\bullet_{3}$, clearly $\boxplus_5$ holds indeed.
\end{enumerate}

Remember, we are trying to show only that for some $\rho \in
\seteq_{\epsilon(*)}$ we have $\eta_{\rho}  =:
\rho ^{[\delta(*)]},M^\rho_{n},N^\rho_{n}$ ($n < \omega $) are as required, we shall prove more,
\mn
\begin{enumerate}
    \item[$\oplus_1$]  if $(\epsilon,\zeta,\delta) \in W_1$ \then \,
    $\Omega_{(\varepsilon,\zeta,\delta)}$ has cardinality $\le \theta$
    where $\Omega_{(\varepsilon,\zeta,\delta)} := \{\nu_\rho:\rho \in
    \Lambda$ and $(\epsilon_\rho,\zeta_\rho,\delta_\rho) = (\varepsilon, \zeta, \delta)\}$.
\end{enumerate}
\mn
as $|W_{1}| \le \theta < \theta^+ = |\seteq_{\epsilon(*)}|$, 
this will be enough. 

So let $y=(\epsilon,\zeta,\delta) \in W_{1}$ we know that $\epsilon  >
\epsilon(*)$ hence $\ga_{\epsilon} \cap \ga_{\epsilon(*)}$ is finite.
Let for $\alpha \in C^\epsilon_{\delta}$:
\[
\gamma [\alpha ] = \min\{\gamma \in C^{\epsilon(*)}_{\delta(*)}: 
\alpha \text{ belongs to } M^*_{\gamma}, 
\text {(equivalently: } C^\epsilon_{\alpha} \in N \cap
M^*_{\gamma})\}.
\]

Now, $\gamma[\alpha]$ is well defined because $C_{\delta}^{\varepsilon} \subseteq N$ (by $(\ast)_{4}$) and $N \subseteq M_{\delta(\ast)} = \bigcup \{ M_{\gamma}^{\ast} \colon \gamma \in C_{\gamma(\ast)}^{\varepsilon(\ast)} \}$ because $C_{\delta(\ast)}^{\varepsilon(\ast)}$ is unbounded in $\delta(\ast)$. Next, 

\begin{fact}\label{7.9N}
    $\langle \gamma[\alpha]:\alpha \in C^\epsilon_{\delta} \rangle$ is a non-decreasing sequence of ordinals which are non-accumulation members of $C^{\epsilon(*)}_{\delta(*)}$, (with limit $\delta(*))$.
\end{fact}

\begin{PROOF}{\ref{7.9N}}
    If $\alpha \in C^\epsilon_\delta$ then $C^\epsilon_\alpha = C_{\delta}^{\varepsilon} \cap \alpha \subseteq N\cap M^*_{\gamma[\alpha]}$ hence $\beta\in C^\epsilon_\alpha \Rightarrow C^\epsilon_\beta = C^\epsilon_\alpha\cap  \beta\subseteq C^\epsilon_\alpha
    \subseteq N\cap M^*_{\gamma[\alpha]} \Rightarrow \gamma[\beta] \le
    \gamma[\alpha]$ so $\beta < \alpha \ \&\ \alpha \in C^\epsilon_\delta
    \, \& \, \beta \in C^\epsilon_\delta \Rightarrow \gamma[\beta] \le
    \gamma [\alpha]$.  Being non-accumulation points is trivial by the definition.
\end{PROOF}

\underline{Continuation of the proof of \ref{7.9}}: 

For $\kappa \in \ga_{\epsilon(*)}$, let:

\begin{itemize}
    \item $\beta^{\epsilon(*)}(\kappa)$ be the supremum of the set
    \[
    \{ \gamma[\alpha]:\alpha \in 
        C^\epsilon_{\delta} \text{ and } \otp(\alpha \cap C^\epsilon_{\delta}) 
        \le \sup(\ga_{\epsilon} \cap \kappa) \text{ and} 
        \otp(\gamma[\alpha] \cap C^{\epsilon(*)}_{\delta(*)})<\kappa\},
    \]

    \item $\gamma^{\epsilon(*)}(\kappa) = \otp(C^{\epsilon(*)}_{\delta(*)} \cap
\beta^{\epsilon(*)}(\kappa))$. 
\end{itemize}

% \begin{equation*}
%     \begin{array}{clcr}
%         \beta^{\epsilon(*)}(\kappa) = \sup\{\gamma[\alpha]:&\alpha \in 
%         C^\epsilon_{\delta} \text{ and } \otp(\alpha \cap C^\epsilon_{\delta}) 
%         \le \sup(\ga_{\epsilon} \cap \kappa) \text{ and} \\
%           &\otp(\gamma[\alpha] \cap C^{\epsilon(*)}_{\delta(*)})<\kappa\}
%     \end{array}
% \end{equation*}

Note: the supremum is taken over a set of $\le \sup(\ga_\epsilon \cap
\kappa)$ ordinals $< \kappa$ but $\ga_\epsilon \cap \kappa$ is a countable (even finite) set of
cardinals $< \kappa,\kappa$ regular uncountable so
$\sup(\ga_\epsilon \cap \kappa) < \kappa$ hence clearly
$\gamma^{\epsilon(*)}(\kappa) < \kappa$.

So $\langle \gamma^{\epsilon(*)}(\kappa):\kappa \in
\ga_{\epsilon(*)}\rangle$ belongs to $\Pi \ga_{\epsilon(*)}$ hence for
some $j(y) < \partial_{\varepsilon(\ast)}$, we have:

\begin{enumerate}
    \item[$\oplus_{1.1}$]   $\rho_{\epsilon(*),j(y)}(\kappa) >
    \gamma^{\epsilon(*)}(\kappa)$ for every large enough $\kappa \in 
    \ga_{\epsilon(*)}$.
\end{enumerate}

For $\kappa \in \ga_{\epsilon}$, let:

\begin{itemize}
    \item $\beta^\epsilon(\kappa)$ be the supremum of the set of the ordinals $\alpha \in C^\epsilon_{\delta}$ such that: 
    \begin{itemize}
        \item for some $\kappa_{1} \in \ga_{\epsilon(*)}, \otp(\gamma[\alpha] \cap C^{\epsilon(*)}_{\delta(*)}) \le \kappa_{1}$, 

        \item $(\forall \beta < \alpha)[\otp(\gamma[\beta] \cap C^{\epsilon(*)}_{\delta(*)}) < \kappa_1],$ and 

        \item $\otp(\alpha \cap C^\epsilon_{\delta})<\kappa$. 
    \end{itemize}

    \item $\gamma^\epsilon(\kappa) = \otp(C^\epsilon_\gamma \cap \beta^\epsilon
    (\kappa))$.
\end{itemize}

% \begin{equation*}
% \begin{array}{clcr}
% \beta^\epsilon(\kappa) = \sup\{\alpha:&\alpha\in C^\epsilon_{\delta}
% \text { and for some } \kappa_{1} \in \ga_{\epsilon(*)},\\
%   &\otp(\gamma[\alpha] \cap C^{\epsilon(*)}_{\delta(*}) \le \kappa_{1}\\
%   &(\forall \beta < \alpha)[\otp(\gamma[\beta] \cap 
% C^{\epsilon(*)}_{\delta(*)}) < \kappa_1], \\
%   &\text{ and } \otp(\alpha \cap C^\epsilon_{\delta})<\kappa\}
% \end{array}
% \end{equation*}

\mn
again, the supremum is taken over a set of $< \kappa$ ordinals $< \kappa$, hence clearly
$\gamma^\epsilon(\kappa) < \kappa$.

So $\langle \gamma^\epsilon(\kappa):\kappa \in \ga_{\epsilon}\rangle$
belongs to $\Pi \ga_{\epsilon}$ hence for some $i(y) < \theta^+$, we have:
$\rho_{\epsilon,i(y)}(\kappa) > \gamma^\epsilon(\kappa)$ for
every large enough $\kappa \in \ga_{\epsilon}$.

Now recall $\seteq_{\epsilon} = \{\rho_{\epsilon,i}:i <
\theta^+\}$ and similarly for $\epsilon(*)$, so
clearly if $i(y) < i < \theta^+ \, \& \, i(y) < j < \theta^+$ then 
$\rho^{[\delta]}_{\epsilon,j}$ cannot
``hurt" $\rho^{[\delta(*)]}_{\epsilon(*),i}$, that is,
$\nu_{\rho_{\varepsilon(*),i}} \in \{\rho^{[\delta]}_{\epsilon,j}:
i(y) < j < \theta^+\}$ so $|\Omega_y| \le |i(y)|$ so $\oplus_1$ holds.

Now we shall show that each $\nu = \rho^{[\delta]}_{\epsilon,j}$ (for $j
\le i(\epsilon))$ can hurt at most $\theta$ (also $\le
2^{\aleph_{0}})$ many $\rho \in \seteq_{\epsilon(*)}$; that is:

\begin{enumerate}
    \item[$\oplus_2$]  if $\nu \in \Omega_y$ then $\Lambda_{y,\nu} =
    \{\rho \in \Lambda:(\varepsilon_\rho,\zeta_\rho,\delta_\rho)=y$ and
    $\nu_\rho = \nu\}$ has cardinality $\le \theta$, recall $\nu_{\rho}$ is from $\boxplus_{5}$. 
\end{enumerate}

Now $\Rang(\rho^{[\delta(*)]})$ has infinite intersection with
\[
B := \{\alpha<\delta(*):\text{ for some } \ell,\nu \rest \ell \in
M^*_{\alpha +1} \setminus M_{\alpha}\}
\]

so let for $\kappa \in \ga_{\epsilon(*)}$:
\[
\beta^*_{\kappa} = \sup\{\otp(C^{\epsilon(*)}_{\delta(*)} \cap
 \alpha): \alpha \in B, \otp(C^{\epsilon(*)}_{\delta(*)} \cap
\alpha) < \kappa\}.
\]

\mn
So for some $i(*) < \theta^+,\beta^*_{\kappa} < \rho_{\epsilon(*),
i(*)}(\kappa)$ for every $\kappa \in \ga_{\epsilon(*)}$ large enough;
so for every $i$, if $i(*) < i < \theta^+$, then 
$\rho_{\epsilon(*),i}$ is not hurt, that is,
$\rho_{\varepsilon(*),i(*)} \notin \Lambda_{y,\nu}$ so $\oplus_2$
holds.

We can conclude

\begin{enumerate}
    \item[$\oplus_3$]  if $y = (\varepsilon,\zeta,\delta) \in W$ then
    $\Lambda_y = \{\rho \in \Lambda:(\varepsilon_\rho,\zeta_\rho,\delta_\rho) =  y\}$ has cardinality $\le \theta$.
\end{enumerate}

[Why?  By $\oplus_1 + \oplus_2$.]

\begin{fact}\label{7.90N}
    \underline{But} we have: 

    \begin{enumerate}
        \item[$\oplus_4$]  $\Lambda$ has cardinality $\le \theta$.
\end{enumerate}
\end{fact}

\begin{PROOF}{\ref{7.90N}}
    As $|W_1| \le \theta$ and $\Lambda = \cup\{\Lambda_y:y \in
    W_1\}$ and each $\Lambda_y$ has cardinality $\le \theta\}$.
\end{PROOF}

So necessarily $\seteq_{\varepsilon(*)} \nsubseteq \Lambda$ and so for some
$\rho \in \seteq_{\varepsilon(*)} \backslash \Lambda$.  Definition
\ref{7.1} is exemplified by $\eta_\rho =
\rho^{[\delta(*)]},\mu^\rho_n,N^\rho_n$ (for $n < \omega$), so we finish.
\end{PROOF}

\centerline{$* \qquad * \qquad *$}
\bigskip

\begin{lemma}\label{7.10}
Suppose $\lambda$ is strong limit, $\lambda = \kappa^{+\omega} > \mu$.  
\Then \, $K^\omega_{\tr}$ has the full 
$(\lambda,\lambda,\mu,\aleph_{0})$-super$^6$-bigness property.
\end{lemma}

\begin{remark} \label{7.10.1}
    We use variants of this proof in case 6 of the proof of \ref{7.14}.
\end{remark}

\begin{PROOF}{\ref{7.10}}
Let $\chi > 2^\lambda$ be large enough.

\Wilog \, $\kappa > \mu$ and $\kappa^\mu = \kappa$.  
Let $\langle C_{\delta}:\delta < \lambda \rangle$ be such that
$C_{\delta}$ is a club of $\delta$ of order type $\cf(\delta)$.  If
$(\kappa^{+n})^{\kappa^{++}} = \kappa^{+n}$ then we choose a function $\cd_{n}$ from $\{M \in {\cH}(\chi) \colon M$ a model, $\|M\| \le \kappa^{++},|\tau(M)| \le
\kappa^{++}$ and $\tau(M) \in {\cH}_{<\kappa^{+3}}(\kappa^{+n})\}$
to $\kappa^{+n}$ such that:

\begin{enumerate}
    \item[$\oplus_{1}$] $\cd_{n}(M_{1}) = \cd_{n}(M_{2}) \text{ iff } M_{1} \cong M_{2} \, \& \, M_{1} \cap \kappa^{+n} = M_{2} \cap \kappa^{+n}. $
\end{enumerate}

As $\lambda$ is strong limit, $2^{\kappa^{++}} < \lambda = \kappa^{+\omega}$ hence $\cd_{n}$ is well defined for every $n$ large enough, so choose $n_{0} > 3$ such that $\cd_{n}$ is well-defined for every $n \geq n_{0}$.  \Wilog \, $\cd_{n}$ is definable in $({\cH}(\chi),\in,<^*_{\chi})$.  We call $\cd_n(M)$ the $n$-code of $M$ or a code of $M$. 

Also for every $n>0$ there are $f_{n},g_{n}$
(definable in $({\cH}(\chi),\in,<^*_{\chi}))$, two place functions from $\kappa^{+n}$ to $\kappa^{+n}$ such that for $\alpha < \kappa^{+n}$ if $\alpha \ge \kappa^{+(n-1)}$ then:

\begin{enumerate}
    \item[$\oplus_{2}$] $\{f_{n}(\alpha,i):i < \kappa^{+(n-1)}\} = \alpha \text{ and } i < \kappa^{+(n-1)} \Rightarrow g_{n}(\alpha, f_{n}(\alpha,i)) = i.$ 
\end{enumerate}

For any $n \ge 3$, let 

\begin{enumerate}
    \item[$\oplus_{3}$] $\cS_n$ is the set of the sets $A$ such that: 

    \begin{itemize}
        \item[$\boxplus_{1}$] $A \subseteq \kappa^{+n}$  and $|A| = \kappa^{++},$
        
        \item[$\bullet_{2}$] $\{ \kappa^{+\ell} + i \colon i < \omega, \ell \leq n \} \cup \kappa^{++}$, $\kappa^{++}  \subseteq A,$ 

        \item[$\bullet_{3}$] letting  $\delta_{\ell}(A) = \sup(A \cap  \kappa^{+\ell})$,  (for  $\ell = 3,...,n), $ we have: $\delta_{\ell}(A) > \kappa^{+(\ell -1)}$ as $\kappa^{+\ell} +1 \in A$, 

        \item[$\bullet_{4}$]  $\cf(\delta_{\ell}(A)) = \aleph_{0}$ 
        
        \item[$\bullet_{5}$] $\ell \leq n \Rightarrow  C_{\delta_{\ell}(A)} \subseteq A$, 

        \item[$\bullet_{6}$] $A$ is the closure of  $\{i:i<\kappa^{++}\} \cup \bigcup\limits_{\ell=3}^{n} C_{\delta_{\ell}(A)}$ under the functions  $f_{\ell}, g_{\ell},  (\ell =3,...,n)$.
    \end{itemize}
\end{enumerate}

Note that if $A \in {\cS}_{n}$, then $n$ can be reconstructed from $A$.

We can prove by induction on $n \ge 3$ that:

\begin{enumerate}
    \item[$\oplus_{4}$] for every $x \in {\cH}(\chi)$
    there is a sequence $\langle M_m:m < \omega \rangle$, such that
    $M_m \prec ({\cH}(\chi),\in,<_{\chi}),\|M_m\| = \kappa^{++},
    \kappa^{++} + 1 \subseteq M_0,x \in M_0,M_m \prec M_{m+1},M_m \in
    M_{m+1}$ (hence $\cd(M_m) \in M_{m+1})$  and
    
    \[
    \bigcup\limits_{m < \omega} M_m \cap \kappa^{+n} \in {\cS}_{n}.
    \]
\end{enumerate}

\mn
Hence for $n \ge n_0$ we know that $\diamondsuit_{{\cS}_{n}}$ holds (see
\cite[4.5(2)=Ld14]{Sh:E62}), in fact 
\mn
\begin{enumerate}
\item[$\boxplus_1$]  for $n \ge n_0$ there is $\bar N_n$ such that:
\sn
\begin{enumerate}
\item[$(a)$]   $\bar N_n = \langle N_{A}:A \in {\cS}_{n}\rangle$,
\sn
\item[$(b)$]  $N_{A}$ a model with universe $A$,
\sn
\item[$(c)$] if $(\alpha)$ then $(\beta) + (\gamma)$, where: 

\begin{enumerate}
    \item[$(\alpha)$] $N$ is a model with universe $\kappa^{+n}$ and vocabulary $\tau(N)$ of cardinality $\le \mu$ included in ${\cH}(\mu^*)$ and satisfying  $<_{\ast}$ is a member of $\tau(N),$ the vocabulary of $N$, $<_*^N = < \rest N$, 

    \item[$(\beta)$] the set $\cS_n[N] = \{A \in {\cS}_{n}:N_{A} = N \rest A\}$ is a stationary subset of $[\kappa^{+n}]^{\le \kappa^{++}}$, 

    \item[$(\gamma)$] $N_{A} = \bigcup\limits_{\ell < \omega} N^\ell_A$, where for each $\ell$, some code $\alpha^\ell_{A}$ of $N^\ell_{A}$,  belongs to $N_{A}$ and $N^\ell_A \prec N_A$. 
\end{enumerate}
\end{enumerate}
\end{enumerate}
\mn
By \cite[1.19=La54]{Sh:E62} there are $\langle N^\eta_{A}:\eta \in 
{\cT}_{A}\rangle$ for $A \in{\cS}_n$ such that:
\mn
\begin{enumerate}
\item[$\boxplus_2$]  $(a) \quad \cT_A \subseteq {}^{\omega
>}(\kappa^{++}),\cT_A$ closed under initial segments, $\langle \rangle \in
\cT_A,\eta \in {\cT}_{A} \Rightarrow$

\hskip25pt $(\exists^{\kappa^{++}}\alpha)
[\eta \char 94 \langle \alpha \rangle \in {\cT}_{A}]$,
\sn
\item[${{}}$]  $(b) \quad$ if $\eta \in {\cT}_A$ then 
$N^\eta_{A} \prec N_{A},\eta \in N^\eta_{A}$,
\sn
\item[${{}}$]  $(c) \quad N^\eta_{A}$ countable,
\sn
\item[${{}}$]  $(d) \quad N^\eta_{A} \cap \kappa = N^{<>}_{A} \cap
\kappa$,
\sn
\item[${{}}$]  $(e) \quad N^\eta_{A} \cap N^\nu_{A} = N^{\eta \cap
\nu}_{A}$,
\sn
\item[${{}}$]  $(f) \quad [\eta \ne \nu \in \cT_A \Rightarrow N^\eta_{A} \ne
N^\nu_{A}]$ and $[\neg(\eta \trianglelefteq \nu) \Rightarrow \eta 
\notin N^\nu_{A}]$,
\sn
\item[${{}}$]  $(g) \quad \{\alpha^\ell_{A}:\ell < \omega\} \cup
\bigcup\limits_{\ell=3}^{n} C_{\delta_{\ell}(A)} \subseteq N^{<>}_{A}$, recalling that  $\alpha_{A}^{\ell}$ is from $\boxplus_{1}$(c)$(\gamma)$, 
\sn
\item[${{}}$]  $(h) \quad \eta \triangleleft \nu \Rightarrow N^\eta_{A} \cap
\kappa^{++}$ is an initial segment of $N^\nu_{A} \cap \kappa^{++}$.
\end{enumerate}
\mn

\begin{enumerate}
    \item[$\boxplus_{3}$] 

    \begin{enumerate}
        \item[(a)] We let $N^\eta_{A} = \bigcup\limits_{\ell < \omega} 
        N^{\eta \rest \ell}_{A}$ when $\eta \in \lim(\cT_A)$.

        \item[(b)] Without loss of
        generality, if $N_{A},N_{B}$ are isomorphic then ${\cT}_A 
        = {\cT}_{B}$ and the (unique) isomorphisms from  
        $N_{A}$ onto $N_{B}$ carry $N^\eta_{A}$ to $N^\eta_{B}$ for each 
        $\eta \in {\cT}_A$.

        \item[(c)] For $\nu \in \lim({\cT}_{A})$, let $\eta^\nu_{A} \in {}^\omega (N^\nu_A)$ just list exactly the members of $N^\nu_{A}$
        and satisfies $\alpha^\ell_{A} = \eta^\nu_{A}(3\ell)$ (for $\ell < \omega$)\footnote{Actually it suffices if it lists $\cup \{C_{\delta_{\ell}(A)}:3 \le \ell  < n\} \cup \{\alpha^\ell_{A}:\ell < \omega\} \cup \{\nu(\ell):\ell < \omega\}$; 
        this change is needed for \ref{7.10A}.}.
    \end{enumerate}
\end{enumerate} 

Let
\[
\langle S^n_{<\gamma_{3},\gamma_{4},\ldots,\gamma_{n}>}:n < \omega
\text{ and } \ell \in \{3,\dotsc,n\} \Rightarrow 
\gamma_{\ell} < \kappa ^{+3}\rangle
\]
be a sequence of stationary subsets of $\{\delta < \kappa^{++}: \cf(\delta) = \aleph_{0}\}$, any two have a bounded intersection (exist, see \cite[4.1=Ld4]{Sh:E62} (which prove more))\footnote{We can assume  $\cup \{S^n_{<\gamma_{\ell}:\ell = 3,
\ldots,n>}:\, \gamma_{\ell} <\kappa^{+2}\}$ 
for $n < \omega$ are pairwise disjoint.}.

\begin{fact}\label{7.10R4}
    We have: 
    \begin{enumerate}
        \item[$\boxplus_{4}$] We can choose $\bar{\cS}$, $\bar{N}$ such that: 

        \begin{enumerate}
            \item[(a)] $\bar{\cS} = \langle \cS_{n} \colon n \in [n_{0}, \omega) \rangle$, where $n_{0}$ was chosen applying $(\ast)_{1}$,

            \item[(b)] $\bar{\cS}_{n} = \langle \cS_{n, \zeta} \colon \zeta < \kappa^{+n} \rangle$ (for $n \geq n_{0}$) is a partition of $\cS_{n}$. 

            \item[(c)] $\bar{N} = \langle \bar{N}^{n} \colon n \in [n_{0}, n \rangle$, where $\bar{N}^{n} = \langle N_{A} \colon A \in \cS_{n} \rangle$, 

            \item[(d)] for each $n, \zeta$ the sequence $\langle N_{A} \colon A \in \cS_{n, \zeta} \rangle$ is a diamond sequence. 
        \end{enumerate}
    \end{enumerate}
\end{fact}

\begin{PROOF}{\ref{7.10R4}}
    E.g. let $P_* \in \cH(\mu^+)$ serve as a unary predicate and
    for every $\zeta < \kappa^{+n}$ let $\cS'_{n,\zeta} = \{A \in \cS_n:P
    \in \tau(N_A)$ and $P^{N_A}_* = \{\zeta\}\}$ and for $A \in
    \cS'_{n,\zeta}$ let $N'_A = N_A \rest (\tau(N_A)) \backslash \{P_*\}$;
    renaming the vocabularies and adding $\cS_n \backslash
    \cup\{\cS'_{n,\zeta}:\zeta < \lambda\}$ to $\cS'_{n,\zeta}$, we can finish proving \ref{7.10R4}. 
\end{PROOF}

\begin{fact}\label{7.10R5}
    If $A \in \cS_{n, \zeta}$ then $\otp(N_{A} \cap \kappa^{+\ell}) < \kappa^{+3}$. 
\end{fact}

\begin{PROOF}{\ref{7.10R5}}
    This holds because $\Vert N_{A} \Vert = \kappa^{++}$ as $A \in \cS_{n}$, see its definition.  
\end{PROOF}

Now, for $\zeta \in [\kappa^{+(n-1)}, \kappa^{+n})$, where $n > n_{0}$, let: 
\begin{enumerate}
       \item[$\oplus$] % 2026-04-21 06:04 \boxplus_{4}$] 
    $I_\zeta  = {}^{\omega >}\lambda \cup \{\eta^\nu_{A}:A \in
    {\cS}_{n,\zeta} \text{ and } \nu \in Y^n_{A,\langle\otp(N_{A}\cap 
    \kappa^{+\ell}):3 \le \ell \le n\rangle}\}.$
\end{enumerate}

where (recall that $\lim(\cT_{A}) \coloneqq \{ \eta \in {\omega}^{\Ord} \colon$ if $n < \omega$ then $\eta \rest n \in \cT_{A}\}$):
\[
Y^n_{A,\gamma_3,\ldots,\gamma_n} = \{\nu:\nu \in \lim({\cT}_A), 
\text{ increasing with limit in } S^n_{<\gamma_3,\ldots,\gamma_n>}\}:
\]

\begin{enumerate}
    \item[$\boxplus_{5}$] We shall prove that the sequence $\langle I_{\zeta}:\kappa^{+n_{0}}
    \le \zeta < \lambda \rangle$ is as required; this suffices. 
\end{enumerate}

\begin{enumerate}
    \item[$\boxplus_{6}$] For this suppose that:  

    \begin{enumerate}
        \item[(a)]  $x \in {\cH}(\chi),$ where $\chi$ regular large enough,

        \item[(b)] $\zeta \in [\kappa^{+n_{0}},\lambda)$,

        \item[(c)] $J_{\zeta}  = \sum\limits\{I_{\xi}:\xi \ne \zeta \text{ and } \xi
        \in [\kappa^{+n_{0}},\lambda)\}$, and 

        \item[(d)] let $n$ be such that $\kappa^{+(n-1)} \le \zeta < \kappa^{+n}$.  Let 
        $M \prec ({\cH}(\chi),\in,<^*_{\chi})$ have cardinality $\kappa^{+n}$ 
        and be such that $\kappa^{+n} + 1 \subseteq
        M,\{x,I_{\zeta},J_{\zeta},\mu\} \in M$ and 
        $\langle C_{\delta}:\delta < \lambda \rangle,\langle \cd_{n},f_{n},g_{n}:n <
        \omega \rangle$ belong to $M$.
    \end{enumerate}
\end{enumerate}

\begin{enumerate}
    \item[$\boxplus_{7}$] 

    \begin{enumerate}
        \item[(a)]  Let $h$ be a one to one function from $|M|$  onto $\kappa^{+n}$,

        \item[(b)] let
        $N^+$ be the model with universe $\kappa^{+n}$ and all relations 
        and functions on $\kappa^{+n}$ definable (with no parameters) in $(M,h)$, 

        \item[(c)] In particular we may use $F,F_1,F_2$ such that 
        $x=\langle y,z\rangle\in M \Rightarrow F^{N^+} (h(y),h(z))=h(x),F^{N^+}_1
        (h(x))=h(y),F^{N^+}_2 (h(x)) = h(z)$.
    \end{enumerate}
\end{enumerate}

So, 

\begin{enumerate}
    \item[$\boxplus_{8}$]

    \begin{enumerate}
        \item[(a)]  For some $A \in {\cS}_{n,\zeta }$ we have $N_{A} \prec N^+$. 

        \item[(b)] We shall show
        that for some  $\nu \in Y^n_{A,\langle\otp(N_{A} \cap \kappa^{+\ell}):3 \le \ell \le n\rangle}$ we have: $\eta^\nu_{A},N^\nu_{A}$, are as required.
    \end{enumerate}
\end{enumerate}

Let $M_{A},M^\nu_{A}$ for $(\nu \in \lim({\cT}_{A}))$ be the Skolem 
Hull of $N_{A},N^\nu_{A}$ respectively in $(M,h)$.  
Note: $|M_{A}| \cap \kappa^{+n} = |N_{A}|,|M^\eta_{A}| \cap
\kappa^{+n} = |N^\eta_{A}|$.  For $\nu \in \lim ({\cT}_{A})$, let
$Z_{\nu} $ be the set of triples $(\xi,B,\rho)$ such that for some $m =
m(\xi)>n_0:\xi \ne \zeta,B \in {\cS}_{m,\xi},\xi \in 
[\kappa^{+(m-1)},\kappa^{+m}),\rho \in \lim ({\cT}_{B})$ and
$<\xi >\char 94 \eta^\rho_{B} \notin M^\nu_{A}$ but
$\{(<\xi >\char94 \eta^\rho_{B}) \rest \ell:\ell < \omega\} \subseteq 
M^\nu_{A}$.
 
We now make some important observations:

\begin{fact}\label{7.10R9}
    We have: 
    \begin{enumerate}
        \item[$(*)_{1}$]  if $(\xi,B,\rho) \in Z_{\nu},\xi \in  [\kappa^{+(n-1)},\kappa^{+n})$ (i.e. $m(\xi) = n$) then $\otp(N_{B} \cap \kappa^{+\ell}) \le \otp(N_{A} \cap \kappa^{+\ell})$ for $\ell \in [3,n]$;  and for at least one $\ell$ the inequality is strict and $B\subseteq A$. 
    \end{enumerate}
\end{fact}

\begin{PROOF}{\ref{7.10R9}}
    As $C_{\delta_{\ell}(B)} \subseteq \Rang(\eta^\rho_{B})$  we have $\bigcup\limits_{\ell=3}^{n} C_{\delta_{\ell} (B)} \subseteq
    N^\nu_{A} \subseteq A$,  hence (see the definition of ${\cS}_{n}$, using the  
    $\langle f_\ell,g_\ell:\ell=3,\ldots,n-1\rangle$ we get $B \subseteq
    A$ so the inequality  ``$\otp(N_{B} \cap \kappa^{+\ell}) \leq \otp(N_{A} \cap \kappa^{+\ell})$'' follows; but necessarily  $B \ne A$ (as
    $\langle \xi \rangle \char 94 \eta^\rho_{B} \in J_\zeta$ 
    and ${\cS}_{n,\xi } \cap {\cS}_{n,\zeta} = \emptyset)$ and if  
    $\neg(\exists \ell)(\delta_{\ell} (B) < \delta_{\ell}(A))$ then we 
    have:  $\kappa^{++} \subseteq B$,  and for $\ell \le$ and $n\ge 3$
    \[
    \sup(B \cap \kappa^{+\ell}) =
    \sup(A \cap \kappa^{+\ell}) = \sup(A \cap B \cap \kappa^{+\ell});
    \]
    
    \mn
    now use the choice of $f_{n},g_{n}$.  You can show, using $B \subseteq A$, by
    induction on  $\ell \le n$  that  $B \cap \kappa^{+\ell} = A \cap
    \kappa^{+\ell}$; for $\ell = n$  we get a contradiction.
\end{PROOF}

Next, we have: 

\begin{fact}\label{7.10R11}
    \begin{enumerate}
        \item[$(*)_{2}$]  if $(\xi,B,\rho) \in Z_{\nu}$ then $\{\delta_{\ell}(B):3
        \le \ell \le m(\xi)\}$  is included in the closure of $|M^\nu_{A}|$ 
        in the order topology, which is a countable set of ordinals; also 
    $B \subseteq M_A$.
\end{enumerate}
\end{fact}

\begin{PROOF}{\ref{7.10R11}}
    Similar argument; for $B \subseteq M_A$ use $\eta^\rho_{B}(3\ell) = \alpha^\ell_{B}$.
\end{PROOF}

\begin{fact}\label{7.10R14}
    \begin{enumerate}
    \item[$(*)_{3}$] So if $Y \subseteq \lim({\cT}_{A})$ is 
    closed with countable density, and no isolated points, then
    for some  $\nu \in Y$ (really for a co-countable set of $\nu$'s):
    \sn
    \begin{enumerate}
    \item[$\otimes$]   $(\xi,B,\rho) \in Z_{\nu} \Rightarrow (\exists k)
    [\{\alpha^\ell_{B}:\ell < \omega\} \subseteq M^{\nu \rest k}_{A}]$.
    \end{enumerate}
    \end{enumerate}
\end{fact}

\begin{PROOF}{\ref{7.10R14}}
    Why?  The point is that  $\{(\xi,B):(\exists \nu \in Y)(\exists \rho) [(\xi,B,\rho) \in Z_{\nu}]\}$  is countable (as for
    each $(\xi,B,\rho) \in Z_{\nu}$ the ordinals $\delta_{\ell}(B),3 \le \ell
    \le m(\xi)$,  are all accumulation points of 
    $\bigcup\limits_{\nu \in Y} M^\nu_{A}$,  which is countable and
    $\langle \delta_\ell (B):3 \le \ell \le m(\xi)\rangle$ 
    determine $B$ hence $\xi$, and for each 
    such $(\xi,B)$  the set of $\nu \in Y$  for which $\otimes$ fails
    is at most a singleton, using clause (e) above and the last clause in the
    definition of $Z_{\nu}$.
\end{PROOF}

Lastly,

\begin{fact}\label{7.10R17}
    We have: 

    \begin{enumerate}
        \item[$(\ast)_{4}$] $C^{\ast}$ is a club of $\kappa^{++}$, where $C^{\ast}$ is the set of $\delta < \kappa^{++}$ such that $\delta \in C'$ where: 

        \begin{itemize}
            \item $\{ \alpha_{B}^{\ell} \colon \ell < \omega \} \subseteq M_{A}^{\nu}$ for some $\nu \in {}^{\omega > } \delta$, and $m < \omega$ and $B \in \cS_{m, \xi}$, 

            \item $C'$ is the $<_{\chi}^{\ast}$-first club disjoint to $S_{\langle \otp(A \cap \kappa^{+\ell}) \colon 4 \leq \ell \leq n \rangle}^{m} \cap S_{\langle \otp(B \cap \kappa^{+ \ell}) \colon 3 \leq \ell \leq m \rangle}^{m}$. 
        \end{itemize}
    \end{enumerate}
\end{fact}

\begin{PROOF}{\ref{7.10R17}}
    Why?  Note that $\kappa^\mu = \kappa$ hence $(\kappa^+)^{\aleph_{0}} =
    \kappa ^+$,  so the number of  possible  $B$'s  for each  $\nu  \in
    {}^{\omega>}(\kappa ^{++})$ is $\le \|M^\nu_{A}\|^{\aleph_{0}} \le 
    \kappa^+$,  and use diagonal intersection.
\end{PROOF}

\begin{enumerate}
    \item[$(*)_5$]  if $\nu \in \lim({\cT}_{A}),\nu$ increasing with limit $\delta
    \in C^* \cap S^n_{\langle \otp(A \cap \kappa^{+\ell}): 3 \le \ell \le 
    n\rangle}$ then
    \[
    (\xi,B,\rho) \in Z_{\nu} \Rightarrow \neg \exists k[\{\alpha^\ell_{B}:
    \ell <\omega \} \subseteq M^{\nu \rest k}_{A}].
    \]
\end{enumerate}

[Why?  Easy by the choice of $C^*$.] 

Together we finish: by $(*)_{4}$, we can find $\delta$ as in $(*)_{5}$
and hence we can find a perfect set $Y \subseteq {\cT}_{A}$ of 
sequences with limit $\delta$;  now $(*)_{3},(*)_{5}$ give
contradictory conclusions (alternatively, see the proof of 
\ref{7.10A}).  
\end{PROOF}

\begin{claim}\label{7.10A}
    In fact in \ref{7.10} we can get (under the assumptions of \ref{7.10}) that
    $K^\omega_{\tr}$ has the full 
    $(\lambda,\lambda,\mu,\mu)$-super$^{6^+}$-bigness
    property (and moreover in \ref{7.3}(F) in clause (ii) there we get 
    ``$\mu + 1 \subseteq  M_{n}$" and $[M_n]^{\aleph_0}
    \subseteq M_n$ which implies (vii) $\Leftrightarrow$ (vii)$^+$ there).
\end{claim}

\begin{PROOF}{\ref{7.10A}}
    For this, we have to make several changes in the proof of \ref{7.10}.
    What more do we prove? we get $\mu+1\subseteq M_0$ and $[M_n]^{\aleph_0} 
    \subseteq M_n$. Without loss of
    generality  $\kappa^\mu = \kappa,\mu^{\aleph_0} = \mu$.
    
    Considering models $N$ with universe $\kappa^{+n}$ 
    we demand that $P_{\oor},<_{\oor}$ belong to $\tau(N)$ where we let 
    $P_{\oor},<_{\oor}$ be fixed one and two place predicates and we demand that
    $<^{N}_{\oor}$ is a well-ordering of the subset $P_{\oor}^N$ of 
    $\kappa^{+n}$.  Parallel restriction applies to $N_A$ for 
    $A\in {\cS}_n$.  Latter having $M$ and $h$,
    we demand $P_{\oor}^{N^+} = \{h(\alpha):\alpha\in M$ an ordinal$\},
    <_{\oor}^{N+} = \{(h(\alpha),h(\beta)):\alpha<\beta$ are ordinals from $M\}$. 
    For any $A\in{\cS}_{n}$, we choose a two place function $g_{A}$ such
    that:
    
    \begin{enumerate}
    \item[$\oplus$]  for every $\alpha \in P^{N_A}_{\oor}$,  for some regular
    \footnote{or $\theta = 1$ or $\theta = 0$, cases which still fit.}
    $\theta \le \kappa^{++}$
    \sn
    \begin{enumerate}
    \item[$(i)$]  $(\forall \alpha < \theta)(\forall \beta < \gamma)[g_{A}(\alpha,\beta) 
    <_{\oor} g_{A}(\alpha,\gamma)]$
    \sn
    \item[$(ii)$]  $(\forall \beta)(\exists \gamma)[\beta <_{\oor} \alpha
    \rightarrow  \gamma < \theta \, \& \,\beta \le_{\oor}
    g_{A}(\alpha,\gamma)]$
    \sn
    \item[$(iii)$]  $(\forall \beta )[\theta \le \beta \Rightarrow
    g_{A}(\alpha ,\beta ) = \alpha]$.
    \end{enumerate}
    \end{enumerate}
    
    Of course we demand that if $N_A \cong N_B,A,B \in {\cS}_n$ then the (unique) isomorphic maps $g_A$ to $g_B$.  
    
    When we choose $M$, we demand (note that: if $\Vert M \Vert^{\aleph_{0}} > \Vert M \Vert$): 
    \[
    [a \subseteq M \, \& \, \|M\|^{|a|} = \|M\| \Rightarrow a \in M].
    \]
    
    \mn
    When we choose  $\langle N^\eta_{A}:\eta \in {\cT}_{A}\rangle$ we
    replace condition (c) in $\boxplus_2$ by 
    \mn
    \begin{enumerate}
    \item[$(c)''$]  $N^\eta_{A}$ has cardinality $\mu$ and include $\mu +1$ and
    \[
    [a \subseteq N^\eta_{A} \, \& \, \|N^\eta_{A}\|^{|a|} = \|N^\eta_{A}\|
    \Rightarrow  a \subseteq  N^\eta_{A}]
    \]
    \end{enumerate}
    \mn
    (the partition theorem on trees still holds) and add, i.e. 
    we now use \cite[1.16=La48]{Sh:E62}
    \mn
    \begin{enumerate}
    \item[$(i)$]  if $\eta \triangleleft \nu$ are from  ${\cT}_{A},
    <^{N_{A}}_{\oor}$ is a well ordering of $P^{N_A}_{\oor}(\subseteq A)$  
    then for any $x \in P^{N_A}_{\oor} \cap N^\eta_A$:
    \sn
    \begin{enumerate}
    \item[$(\alpha)$]  if $\kappa^{++} > \cf(\{y:y \in P^{N_A}_{\oor},
    y <^{N_A}_{\oor} x\},<^{N_A}_{\oor})$ \then
    \[
    N^\eta_A \cap \{y:y \in P^{N_A}_{\oor},y <^{N_A}_{\oor} x\}
    \]
    is an unbounded subset of
    \[
    (\{y:y \in P^{N_A}_{\oor},y \in N^\nu_A,y <^{N_A}_{\oor} x\},<^{N_A}_{\oor})
    \]
    \item[$(\beta)$]  if $\kappa^{++} = \cf(\{y:y \in P^{N_A}_{\oor},
    y <^{N_A}_{\oor} x\},<^{N_A}_{\oor})$ \then \, for any 
    $y \in P^{N_A}_{\oor},y <^{N_A}_{\oor} x$, for some $\alpha <
    \kappa^{++}$ we have:
    $\eta \triangleleft \rho \in {\cT}_A \, \& \, \rho(\ell g(\eta)) >
    \alpha \, \& \, y^* \in N_A \cap P^{N_A}_{\oor}\, \& \, y^*
    <^{N_A}_{\oor} x \, \& \, (\forall z)[z \in  N^\eta_A \, \& \, 
    z <^{N_A}_{\oor} x \rightarrow z <^{N_A}_{\oor} y^*]
    \Rightarrow  y <^{N_{A}}_{{\rm or}} y^*$.
    \end{enumerate}
    \end{enumerate}
    \mn
    Note that as $<^{N_A}_{\oor}$ well order $P^{N_A}_{\oor}$,  this is
    possible --- see \cite[1.16=La48]{Sh:E62} and apply it to $(M_{A},g)$.
    
    But now we cannot demand ``$\eta^\nu_A$ list the members of
    $N^\nu_A$"; so we just require
    \begin{enumerate}
    \item[$\boxplus$]  $(a) \quad \alpha^\ell_A = \eta^\nu_A (3\ell)$,
    \sn
    \item[${{}}$]  $(b) \quad \langle \eta^\nu_A(3\ell +1):\ell < \omega 
    \rangle$ list $\bigcup\limits_{\ell=3}^{n} C_{\delta_{\ell} (A)}$ and 
    \sn
    \item[${{}}$]  $(c) \quad \langle \eta^\nu_A(3\ell +2):\ell <
    \omega\rangle$ is $\langle \nu(\ell):\ell < \omega \rangle$.
    \end{enumerate}
    
    Note that this holds for all $\nu \in \cT_{A}$. 
    
    This, of course, ``kills" $(*)_3$ in the proof of \ref{7.10}.  Now if
    $(\xi,B,\rho) \in Z_{\nu}$,  for $\ell = 3,\ldots,m(\xi)$ define
    $\beta_{\ell}  = \sup[\kappa^{+\ell} \cap \rang(\rho)]$, and define
    $\gamma[\beta_{\ell}] = \min(M^\nu_{A} \cap \lambda \setminus \beta)$, 
    so for some $k(*) < \omega$ we have 
    $\bigwedge\limits_{\ell \in [3,m(\xi)]} \gamma[\beta_{\ell}] \in  
    M^{\nu \rest k(*)}_A$.  So by condition (i) above for each 
    $\ell \in [3,m(\xi)]$,  either $\circledast^1_\ell$ holds or 
    $\circledast^2_\ell$ holds where:
    \mn
    \begin{enumerate} 
    \item[$\circledast^1_\ell$]  $\cf(\gamma[\beta_{\ell}]) < \kappa^{++},
    \sup[\gamma[\beta_{\ell}] \cap M^{\nu \rest k(*)}_A] = 
    \sup[\gamma[\beta_{\ell}] \cap M^{\eta'}_A]$
    whenever $\nu \rest k(*) \triangleleft \eta' \in {\cT}_A \cup 
    \lim({\cT}_{A})$ 
    \sn
    \item[$\circledast^2_\ell$]  $\kappa^{++} = \cf(\gamma[\beta_{\ell}])$ 
    and there is $h_{\gamma[\beta_{\ell}]}:\kappa^{++} \rightarrow \gamma(\beta)$
    increasing continuous with limit $\gamma[\beta_{\ell}]$ such that
    \sn
    \begin{enumerate}
    \item[$\bullet$]  $\nu \rest k(\beta) \triangleleft \eta' \in 
    \lim({\cT}_A) \Rightarrow \sup(N^{\eta'}_A \cap \gamma[\beta_{\ell}])$
    \sn
    \item[$\bullet$]  $\sup(M^{\eta'}_{A} \cap
      \Rang(h_{\gamma[\beta_{\ell}]})) =h(\sup[M^{\eta'}_{A} \cap 
    \kappa^{++}])$.
    \end{enumerate}
    \end{enumerate}
    \mn
    As $\mu \le \kappa$, we can finish easily: we can find a club 
    \[
    C' = \{\delta \in C^*: \text{ if } \nu \in {}^{\omega >}\delta,
    \ell \in [3,\omega) \text{ and } \gamma \in N^\nu_A \text{ then }
    \delta \text{ is closed under } h_\gamma\}.
    \]
    
    of $\kappa^{++}$ and choose $\delta \in C'$.
\end{PROOF}

\begin{theorem}
\label{7.11}
1) If $\lambda > \mu$  \then \, $K^\omega_{\tr}$ has the full
$(\lambda,\lambda,\mu,\aleph_{0})$-super-bigness property and also the
$(2^\lambda,\lambda,\mu,\aleph_{0})$-super bigness property.

\noindent
2) Similarly replacing $\aleph_0$ by $\mu$.
\end{theorem}

\begin{PROOF}{\ref{7.11}}
We prove both parts of \ref{7.11} together. The first phase implies the second by \ref{7.5}(1) hence we concentrate on the
first phrase.  This will follow by combining the previous Lemmas.  
We shall use all the time \ref{7.4}(1) to get ``our super", 
the one from Definition \ref{7.1}, i.e. super$^{4^+}$.  
If $\lambda$ is regular, use \ref{7.6}(1) so assume
$\lambda$ singular; if $(\exists \mu_{1})[\mu \le \mu_{1} =
\mu^{\aleph_{0}}_{1} < \lambda \le 2^{\mu_{1}}]$  use \ref{7.6}(2),
for part (2) note ``(even the full ...") and if
$(\exists \theta)[\theta < \lambda \le \theta^{\aleph_0}]$ let
$\chi$ be minimal such that $\lambda \le \chi^{\aleph_0}$; so $<
\lambda$ hence $\mu + \chi < \lambda$, but $\lambda$ is a limit
cardinal so $\mu^+ + \chi^{++} \le \lambda$ and use \ref{7.8}.  So
assume the last two cases fail, hence  $\lambda$ is singular strong limit.  If
$\cf(\lambda) > \aleph_{0}$ use \ref{7.6}(3), if $\cf(\lambda) =
\aleph_{0},\lambda = \aleph_\delta,\delta$ divisible by $\omega^2$, 
choose $\theta,\mu < \theta < \lambda,\cf(\theta) = \aleph_0$ 
and apply \ref{7.9}, $(\langle {\ga}_{\epsilon}:\epsilon <
\cf(\lambda)\rangle$  exists by \cite[3.22=Lpcf.8]{Sh:E62}). 
The remaining case is $\lambda = \aleph_{\delta} = 
\aleph_{\alpha +\omega}$ strong limit and use \ref{7.10} for part (1),
use \ref{7.10A} for part (2).
\end{PROOF}
\newpage

\section {Applications and generalizations}

Conclusion \ref{7.12}(1) (though not \ref{7.12}(2),(3)) tell
us that unstable and unsuperstable has many models, and the proof 
use only a version of the definition from \cite{Sh:E59}. 
Theorem \ref{7.13} tell us more in this direction 
but the proof of \ref{7.13} in case $\lambda=\lambda(T),
T_1=T$ stable require knowledge of stability theory (and is not used
later), this case appear as end-segment of the 
proof of \ref{7.13}, i.e. starting with
the third paragraph of the proof of \ref{7.13} and with \ref{7.13f}).
We restart in \ref{7.14} resuming our investigations of bigness properties and 
then deal with abelian separable $\primep$-group.

\subsection {The Many pairwise Unembeddable Models} \

\begin{conclusion}
\label{7.12}
1)  If $T \subseteq T_{1}$ are complete first order theories 
and $\lambda > |T_{1}|$  \then \, $\numbIE(\lambda,T_{1},T) = 
2^\lambda$ whenever $T$ is unsuperstable.

\noindent
2) If $\lambda > \mu$ \then \, $K^\omega_{\tr}$ has the full strong
$(\lambda,\lambda,\mu,\aleph_{0})$-bigness property and
$(2^\lambda,\lambda,\mu,\aleph_{0})$-bigness property (see Definition
\cite[2.5(3)=L2.3(3)]{Sh:E59}).

\noindent
3) If $\Phi,\langle \varphi_n(x,\bar y): n< \omega\rangle$ are as in
\cite[1.11=Lb17(2)]{Sh:E59}, and $\lambda > |\tau(\Phi)|$, see \ref{7.12A} \then \,  we can find $I_\alpha\in K^\omega_\tr$ of cardinality $\lambda$ 
for $\alpha < 2^\lambda$ such that letting $M_\alpha = 
\EM_{\tau(T)}(I_\alpha,\Phi)$, for any $\alpha \ne \beta$, 
there is no function from $M_\alpha$ into $M_\beta$ preserving the 
$\pm \varphi_n$.
\end{conclusion}

\begin{remark}\label{7.12A}
    Recall that $\tau(\Phi)$ is the vocabulary of $\Phi$, that is, for a linear order $I$, the Ehrenfeucht-Mostoswki model, $\EM(I, \Phi)$ has the vocabulary $\tau(\Phi)$. Also for $\tau \subseteq \tau(\Phi)$, $\EM_{\tau}(I, \Phi)$ is the $\tau$-reduct of $\EM(I, \Phi)$, recall $\tau= \tau(T)$ is the vocabulary of the theory $T$. 
\end{remark}

\begin{PROOF}{\ref{7.12}}
1) Let $\Phi$ be a template proper for $K^\omega_\tr$ as in
\cite[1.11=Lb17(2)]{Sh:E59}; i.e. $|\tau(\Phi)| = |T_1| +\aleph_0,
\tau_{T_1} \subseteq \tau(\Phi)$, every $\EM(I,\Phi)$ is a model of $T_1$
and for some first order formulas $\varphi_n(x,\bar y_n)$ of 
$\bbL_{\omega,\omega}(\tau_T)$ for $s \in P^I_n,t \in P^I_\omega,
I \in K^\omega_\tr$ we have $\EM(I,\Phi) \vDash \varphi (a_t,\bar a_s)$ 
iff $I \models s \vartriangleleft t$.
By \ref{7.12}(2) (which is proved below) and the definition, the conclusion follows
reading the definition of $\numbIE$ (see \cite[1.4=L1.4new]{Sh:E59})
 and the definition of the bigness property.

\noindent
2) By \ref{7.11} and \ref{7.5}(2).

\noindent
3) Included in the proof above.
\end{PROOF}

\begin{theorem}\label{7.13}
Suppose $T$ is (a first-order, complete) unsuperstable theory and
$\lambda \ge \lambda(T) +  \aleph_{1}$ (see below \ref{7.13a}(1)).

\noindent
1) $T$ has $2^\lambda $ pairwise non-isomorphic strongly
$\aleph_{\epsilon}$-saturated models of cardinality $\lambda$, see
\ref{7.13a}(2),(3). 

\noindent
2) If in addition $T$ is stable or $\lambda > \lambda(T) + \aleph_{0}$,
\then \, $T$ has $2^\lambda,\aleph_{\epsilon}$-saturated (see \ref{7.13a}(21)) models of power
$\lambda$ no one elementarily embeddable into another.

\noindent
3) We can in part (2) weaken the assumption to $\lambda > |T| + 
\aleph_0$ but then have to weaken the conclusion 
to ``strongly $\aleph_0$-homogeneous (see \ref{7.13a}(3) below) 
models of cardinality $\lambda$ (omitting the 
``$\aleph_\epsilon$-saturated"; interesting when $\lambda = |T| + 
\aleph_1,T$ stable).

\noindent
4) If $T \subseteq T_{1},T_{1}$ first order, we can demand 
above that the models are in $\PC_{\tau(T)}(T_1,T)$, that is are 
reducts of models of $T_{1}$, provided that: in \ref{7.13}(1)+(2) we demand  
$\lambda>\lambda (T)+|T_{1}| + \aleph_{0}$, in \ref{7.13}(3) we 
demand $\lambda > |T_{1}| + \aleph_{0}$.
\end{theorem}

\begin{remark}\label{7.13a}
    0) In the notion of ``$M$ is $\aleph_{\varepsilon}$-saturated'' the $\varepsilon$ is not a variable, it try to say ``a little more than being $\aleph_{0}$-saturated''; see \ref{7.13a}(2) (see \cite[IV]{Sh:c}). 
    
    1) $\lambda(T)$ can be defined as the minimal cardinality of
    an $\aleph_\varepsilon$-saturated model of $T$, see (2) below.
    
    \noindent
    2)
    
    \begin{enumerate}
        \item[(a)] $M$ is $\aleph_{\epsilon}$-saturated when it is $\bfF_{\aleph_{0}}^{a}$-saturated, where: 
    
        \item[(b)] A model $M$ is $\bfF_{\kappa}^{a}$-saturated \underline{if} for every $N, M \prec N$ and $A \subseteq M$ of cardinality $< \kappa$ and (finite) $\bar{b} \in N$ there is $\bar{b}' \in M$ realizing $\{ \vartheta(\bar{x}, \bar{b}, \bar{a}) \colon \vartheta(\bar{x}, \bar{y}, \bar{z})$ is a first-order formula and in $N$, $\bar{a} \in {}^{\lg(\bar{z})} A$, and the formula $\vartheta(\bar{x}, \bar{y}, \bar{a})$ is an equivalence relation with finitely many equivalent classes$\}$. 
    \end{enumerate}
    
    \noindent
    3) $M$ is strongly $\aleph_\varepsilon$-saturated \If \, in addition
    it is strongly $\aleph_0$-homogeneous which means that for any $\bar a,
    \bar b \in {}^{\omega>} M$ realizing the same type, there is an
    automorphism of $M$ mapping $\bar a$ to $\bar b$.
    
    \noindent
    4) The restrictions in \ref{7.13} are reasonable as, e.g. by
    \cite{Sh:100}: it is consistent with ZFC that for $T$ the theory of 
    dense linear orders (which is an unstable one) there is $T_1\supseteq T$
    (first order complete theory) of cardinality $\aleph_1$ such
    that for any models $M_1,M_2$ in $\PC(T_1,T)$ of cardinality
    $\aleph_1,M_1$ can be embeddable into $M_2$.
    
    \noindent
    5) Recall ${\gC}^{\eq}$ is extending ${\gC}$ by giving names to equivalence classes, see \cite{Sh:c}. 
    Let us say that $M^{\eq} \prec {\gC}^{\eq}$ is strongly$^+$-$
    \aleph_\varepsilon$-saturated if it is $\aleph_{\varepsilon}$-saturated and: for any finite $A,B \subseteq
    M^{\eq}$ and $(M^{\eq},M^{\eq})$-elementary mapping ${\bf f}$ from
    $\acl(A,M^{\eq})$ onto $\acl(B,M^{\eq})$ there is an automorphism
    ${\bf f}^+$ of $M^{\eq}$ extending ${\bf f}$.
    
    \noindent
    6) The reader may wonder if we can get in \ref{7.13}, models which are strongly$^+$
    $\aleph_\varepsilon$-saturated.
    
    Let $\lambda'(T)$ be the first $\lambda \ge \lambda(T)$ such that 
    for any $M^{\eq},A,B$ as above, the number of ${\bf f}$ as above
    is $\le \lambda$.
    
    Now we can in \ref{7.13} demand the models to be strongly$^+$
    $\aleph_\varepsilon$-saturated if $\lambda\geq \lambda'(T)+ \aleph_1$ 
    (or $\lambda > \lambda'(T) + \aleph_0$, as natural). The proof is 
    essentially the same.
    
    \noindent
    7) In fact the proof indicated in \ref{7.13a}(6) is simpler and gives in some 
    respect more information. We can easily prove:
    
    \begin{enumerate}
    \item[$(*)_1$]  if $A \subseteq {\gC}^{\eq},|A| \le \lambda$ 
    then there is an $(\mathbf F,\cP)$-construction ${\cA}$ 
    (see context \ref{3.3Anew}, Definition \ref{3.3Bnew}),such that:
    
    \begin{enumerate}
    \item[$(i)$] $A_0[{\cA}]=A$,
    \sn
    \item[$(ii)$]  $\mathrm{lng}({\cA})$ is divisible by $\lambda$ and 
    $\cf(\mathrm{lng}({\cA})) \ge \kappa$
    
    \item[$(iii)$]  if $D \in \cP$ and $i < \mathrm{lng}({\cA}),B_1 \subseteq A_i[{\cA}],B_2 \subseteq D$ and $\mathbf f$ is an  element mapping  from $\acl(B_2,{\gC}^{\eq})$ onto $\acl(B_1,{\gC}^{\eq}),|B_1|<\kappa,
    |B_2|<\kappa$ then $\mathrm{lng}({\cA}) = \otp\{\beta < \mathrm{lng}{\cA}$: there is  an elementary mapping $\mathbf f'$ from $D$ onto $D_\beta[{\cA}]$  extending $\mathbf f,B_\beta[{\cA}] = B_1\}$
    \end{enumerate}
    \sn
    \item[$(*)_2$]  if ${\cA}^1_,{\cA}^2$ are as in $(*)_1$, 
    and $A_0 [{\cA^{1}}] = \emptyset = A_0[{\cA}^2]$ then
    $A[{\cA}^1],A[{\cA}^2]$ are isomorphic $\mathbf F_\kappa^{a}$-saturated models (see \ref{7.13a}(21)).  
    
    This replaces \ref{3.3Cnew}-\ref{3.3Jnew}, (but use some of those proofs).
    After that, we can continue as in \ref{3.3Knew}.
    \end{enumerate}
    
    8) For the case $T_1=T,\kappa = \cf(\kappa) \le \kappa_{r}(T)$ 
    we can replace in the proof $\aleph_\epsilon$-saturated by 
    ${\bf F}^a_\kappa$-saturated, etc. 

    9) Recall that for a formula $\varphi$ and statement \textbf{stat}, $\varphi^{\textbf{if(stat)}}$ mean $\varphi$ when the statement is true and mean $\neg \varphi$, the negation of $\varphi$ when the statement is false. 
\end{remark}

\begin{PROOF}{\ref{7.13}}
Let $\tau=\tau_T$. 

First assume $T$ is unstable; note:

\begin{fact}\label{7.13ab}
    There is   a template $\Phi$,  proper for linear orders, $|\tau_\Phi| = \lambda(T)$ such that every model $M$ of the form $\EM_{\tau}(I,\Phi)$ is an $\aleph_{\epsilon}$-saturated model of $T$ and $M \models  \varphi[\bar a_{s}, \bar a_{t}]$ iff $s <_{I} t$  for $s,t \in I$, where $I$ is a linear order.
\end{fact}

\begin{PROOF}{\ref{7.13ab}}
    Apply
    \cite[1.26=L1.24new]{Sh:E59} as follows.  As $T$ is unstable there are
    $\varphi(\bar x,\bar y),\bar a_\ell$ ($\ell< \omega$)
    and $M$ such that $M$ is a model of $T,\bar a_\ell \in M,n = 
    \mathrm{lng}(\bar x) = \mathrm{lng}(\bar y) = \mathrm{lng}(\bar a_\ell)$ and $M \vDash
    \varphi(\bar a_\ell,\bar a_k)^{\textbf{if}(\ell< k)}$ (recall \ref{7.13}(9)).  We can also find a
    vocabulary $\tau_1,\tau \subseteq \tau_1,|\tau_1| = \lambda(T)$
    and $\psi \in \bbL_{|T|^+, \aleph_{0}}(\tau_1)$ such that a model of $T$ is
    $\aleph_\epsilon$-saturated iff it can be expanded to a model of $\psi$.
    
    For every $\lambda$ we can find a strongly $|T|^+$-saturated model
    $M_\lambda$ of $T$ and $\bar a^\lambda_\alpha \in M_\lambda$ such that $M_\lambda \models \varphi(a^\lambda_\alpha,a^\lambda_\beta) \Leftrightarrow (\alpha< \beta)$, hence there is an expansion $M^+_\lambda$ of $M_\lambda$ to
    a model of $\psi$. Lastly, check that \cite[1.26=L1.24new]{Sh:E59} gives the desired conclusion.
\end{PROOF}

\underline{Continuing the proof of \ref{7.13}}: 

Now part (1) (of \ref{7.13}) holds by \cite[3.19=L3.9]{Sh:E59} (with $M_I$
being $\EM(I,\Phi)$, it is as required in \cite[3.19=L3.9A]{Sh:E59} by
\cite[3.8=L3.4]{Sh:E59}).

Also part (2) (of \ref{7.13}) holds by
\ref{3.3Knew} (interpreting $I \in K^\omega_{\tr}$ as a linear order as
in \cite[2.4=L2.2]{Sh:E59}) noting that we have $\lambda > \lambda(T) +
\aleph_0 = |\tau_\Phi|$ as we are assuming $T$ is unstable.  The proof
of part (3) is similar, replacing $\tau_\Phi$ by $\tau'$ of cardinality
$\lambda + \aleph_1,|T_1| + \aleph_1$.
Lastly for part (4) \wilog \, every model 
$\EM_\tau(I,\Phi)$ is a reduct of a model of $T_1$, so we are done
by \ref{7.12}(3).

So without loss of generality, $T$ is stable.  As $T$ is unsuperstable, by 
\cite[Proof of 2.1]{Sh:225} (or a proof similar to the first paragraph), 

\begin{enumerate}
    \item[$(\ast)$]  There is a template $\Phi$  proper for $K^\omega_{\tr},|\tau_{\Phi}| =
    \lambda(T)$ as in \cite[1.11(2)=L1.8(2)]{Sh:E59} such that 
    every $\EM_\tau(I,\Phi)$
    is strongly $\aleph_\epsilon$-saturated.  
\end{enumerate}

If $\lambda > \lambda(T)$, 
note that \ref{7.13}(1) follows by \ref{7.12}(3) and \ref{7.13}(3) 
by decreasing $\tau_\Phi$.

In all those proofs we can restrict ourselves to models of $T$ which are reducts of models of $T_{1}$, i.e. demand
that for any suitable $I \in K_{\tr}^{\omega}$, the model $\EM(I,\Phi)$ is a model of $T_1$
so part (4) follows.  We are left with part (2) the case 
$T$ is stable, and the proof is restricted to elementary
classes: the proof needs some knowledge of forking, but it is not used later, so a reader can skip it. We also use the notation of \cite{Sh:c}.

\begin{enumerate}
    \item[$\boxplus_{1}$] Let $\varphi_{n}(\bar x,\bar y_{n})$ (for $n < \omega$), $\bar a_{\eta}(\eta
    \in {}^{\omega \ge}\lambda)$  witness unsuperstability, i.e. be as
    in \cite[Ch.III,\S3]{Sh:a}, so there is $\langle \bar a_{\eta}:
    \eta \in {}^{\omega >}\lambda \rangle$ which is a non-forking tree
    (that is $\eta \in {}^{\omega >}\lambda \Rightarrow \tp(\bar a_\eta,
    \cup\{\bar a_\nu:\neg(\eta \trianglelefteq \nu),\nu \in 
    {}^{\omega>}\lambda\})$ does not fork over $\cup\{\bar a_{\eta \rest \ell}: 
    \ell < \mathrm{lng}(\eta)\}$), and for $\eta \in {}^\omega \lambda,
    \tp(\bar a_{\eta},\cup \{\bar a_{\nu}:\nu  \in {}^{\omega >}\lambda\})$  
    does not fork over $\bigcup\limits_{\ell < \omega} \bar a_{\eta \rest
    \ell}$ and $\tp(\bar a_\eta,\bigcup\limits_{\ell < k}
    \bar a_{\eta \rest \ell})$ forks over $\bigcup\limits_{\ell < k}
    \bar a_{\eta \rest \ell}$ for $k < \omega$. 
\end{enumerate}

Let 
$I \subseteq {}^{\omega \ge}\lambda$ be closed under initial segments, $|I| =
\lambda$ and we shall construct a model $M_I$.  We work in $\gC^{\eq}$, so
without loss of generality $\bar a_\eta = \langle a_\eta \rangle$ so the
$a_\eta$'s are pairwise distinct.

By induction on $\alpha < \lambda^{+}$ we choose $(\bar A^\alpha,\bar f^\alpha) \in
\mathbf K_\alpha$ where
\mn
\begin{enumerate}
\item[$\boxplus_{2}$]  $(\bar A,\bar f) \in \mathbf K_\alpha$ \Iff \, $\bar
  A = \langle A_{i}:i \le \alpha
 \rangle$ and $\bar f = \langle f^i_{c,d}:c,d \in A_{i},\tp(c,\emptyset) = 
\tp(d,\emptyset)$ and $i \le \alpha\rangle$ satisfies:
\sn
\begin{enumerate}
\item[$(A)$]   $\bar A = \langle A_{i}:i \le \alpha \rangle$ is increasing
continuous: $|A_i| = \lambda,A_i \subseteq \gC$
\sn
\item[$(B)$]   $f^i_{c,d}$ is an elementary mapping, $f^i_{c,d}(c) = d,
f^i_{d,c} = (f^i_{c,d})^{-1}$, and for $c,d \in A_j$ the sequence $\langle
f^i_{c,d} \colon i \in [j, \alpha] \rangle$ is increasing continuous, and:
if $c,d \in A_{i}$, but $\bigwedge\limits_{j < i} [\{c,d\} \nsubseteq
  A_j]$ then $\Dom(f^i_{c,d}) = \{c\}$
\sn
\item[$(C)$]  for each $i$:  either
\sn
\item[${{}}$]  $(i) \quad A_{i+1} = A_{i} \cup \{a_i\},\tp(a_i,A_i)$
does not fork over some finite subset 

\hskip30pt $B_i$ of $A_i$
\\
or
\sn
\item[${{}}$]  $(ii) \quad$ for some $\bfc(i), \bfd(i) \in A_i$, (such that
$\tp(\bfc(i),\emptyset) = \tp(\bfd(i),\emptyset))$

\hskip30pt  we have:
\[
A_{i+1} = A_{i} \cup f^{i+1}_{\bfc(i), \bfd(i)}(A_{i})
\]

and
\[
(\exists j < i)[\Rang(f^i_{\bfc(i), \bfd(i)}) = A_{j}] \vee [\Dom(f^i_{\bfc(i), \bfd(i)})
= \{\bfc(i)\}].
\]
\sn
\item[$(D)$]  for every $c,d \in A_{i+1}$ such that $\tp(c,\emptyset) =
\tp(d,\emptyset)$:

\item[${{}}$]  $(i) \quad$ if $\{c,d\}$ is not a subset of $A_i$, then $\Dom
(f^{i+1}_{c,d}) =\{c\}$

\item[${{}}$]  $(ii) \quad$ if $c,d \in A_i$, case (i) of (C) holds or case
(ii) of (C) holds but 

\hskip30pt $\langle c,d \rangle \notin
\{\langle \bfd(i), \bfd(i)\rangle,\langle \bfd(i), \bfc(i)\rangle\}$, \then \, 
$f^{i+1}_{c,d} = f^i_{c,d}$
\sn
\item[${{}}$]  $(iii) \quad$ if  $c = \bfc(i), d = \bfd(i)$ and case (ii) of 
(C) holds, then $\tp(f^{i+1}_{\bfc(i), \bfd(i)}(A_{i}),A_{i})$ does not fork
over $\Rang(f^i_{\bfc(i), \bfd(i)})$ and $\Dom(f^{i+1}_{\bfc(i), \bfd(i)}) =
A_{i}$ and recall $f^{i+1}_{d,c}=(f^{i+1}_{c,d})^{-1}$
\sn
\item[(E)]   $A_{0} = \cup\{{\bar a}_{\eta}:\eta \in I\}$.
\end{enumerate}
\end{enumerate}
Note that we can prove by induction on $\alpha$ that for any such
construction $(\bar A,\bar f) \in \mathbf K$:

\begin{enumerate}
\item[$(*)$]  If $\Dom( f^i_{c,d}) \ne \{c\}$, then 

\begin{enumerate}
\item[$(i)$]  $(\exists \delta \le i)[\Dom(f^i_{c,d}) = A_{\delta} =
\Rang(f^i_{c,d})]$ so $\delta$ is a limit ordinal or

\item[$(ii)$]  $(\exists \epsilon < \zeta \le i)[\Dom(f^i_{c,d}) =
A_{\zeta} \, \& \, \Rang(f^i_{c,d}) = A_{\epsilon} \cup (A_{\zeta+1} \setminus
A_{\zeta})]$ or
\sn
\item[$(iii)$]   $(\exists \epsilon < \zeta \le i)[\Rang(f^i_{c,d}) = A_{\zeta}
\, \& \, \Dom(f^i_{c,d}) = A_{\epsilon} \cup (A_{\zeta +1} \setminus
A_{\zeta})]$.
\end{enumerate}
\end{enumerate}

We can clearly find $\alpha < \lambda^+$ and $(\bar A,\bar f) \in
\mathbf K_\alpha$, i.e. $A_{i},$ (for $i \leq \alpha$)  $f^i_{c,d}$
(for $i < \alpha$) satisfying (A) - (E) such that:
\mn
\begin{enumerate}
\item[$(**)$]  $(i) \quad$ for every finite $B \subseteq A_{\alpha}$ 
and $b \in \gC$,
\sn
\sn
\begin{enumerate}
\item[${{}}$]  $\bullet \quad$ if $\lambda \ge \lambda(T)$ 
then $\stp(b,B)$ is realized by some $a \in A_{\alpha}$, moreover for 
$\lambda$ ordinals $i < \alpha$ clause (i) of (C) holds, 
$B = B_i \subseteq A_i$ and $a_i$ realizes $\stp(b,B)$,
\sn
\item[${{}}$]  $\bullet \quad$ if $|T| \le \lambda < \lambda(T)$ if 
$\bar a$ list $B$ and $\models \varphi[b,\bar a]$ then for $\lambda$ 
ordinals $i<\alpha,\models \varphi[a_i,\bar a]$ and $B_i=B$
\end{enumerate}
\sn
\item[${{}}$]  $(ii) \quad$ for every $c,d \in A_{\alpha},
\Dom(f^\alpha_{c,d}) = A_{\alpha} = \Rang(f^\alpha_{c,d})$.
\end{enumerate}
\mn
This is easy by reasonable bookkeeping and clause (C) above.  Hence
$A_{\alpha}$ is the universe of a strongly 
$\aleph_{\epsilon}$-saturated model if $\lambda \ge \lambda(T)$, 
and strongly $\aleph_0$-homogeneous 
(in both cases of model cardinality $\lambda$), if $\lambda <
\lambda(T)$ (remember we work in $\gC^{\eq}$).  We call it $M_{I}$
(and should  have written $\alpha_{I} < \lambda^+,A^I_{i}$, etc).

This is close to \cite[Sh.IV,5.13, pg.~213 + \S3]{Sh:c}. Well, we have constructed the models, but we still need to show the non-embeddability. This is proved just before \ref{7.13f}, i.e., the end of the sub-section, which deals with the context \ref{3.3Anew} and use \ref{3.3A1new}, \ref{xyz.1}, and in particular \ref{7.13D}. 

In \ref{3.3Anew} we can restrict ourselves to Pos. 1.
So till we finish the proof of \ref{7.13} we adapt the context 
\ref{3.3Anew}, and for notational simplicity only assume 
$\lambda \ge \lambda(T)$, (otherwise Claim \ref{3.3Jnew} has to be revised).
\end{PROOF}

\begin{context}\label{3.3Anew}
$T$ is a stable (first-order) theory, $A, B, C, D$ denote subsets of the monster $\gC = \gC_{T}$ of cardinality $< \kappa$, $\gC$ is $\kappa$-saturated. 

\noindent
\underline{Pos. 1}, ${\bf F} =
{\bf F}^f_{\aleph_0},\kappa = \aleph_0,{\cP} = {\cP}_I = \{D_I\}$
where $D_I = \{a_\eta:\eta \in I\}$ for some $I$, $\langle a_\eta:
\eta\in I\rangle$ as

\begin{enumerate}
    \item[$\boxplus_{2}$] In the proof of \ref{7.12} above, $\lambda \geq |D_I| + \lambda(T)$. 
\end{enumerate}
\noindent
\underline{Pos. 2}:  $T$ is a stable theory, 
${\bf F}={\bf F}^f_\kappa$, see \cite[IV, 3.14]{Sh:c} and
$\kappa = \kappa_r(T)$, so a regular cardinal, ${\cP}$ a family of sets ($\subseteq {\gC} = \gC_{T}$) and 
$\lambda = \lambda^{<\kappa} + \lambda (T) \ge \sup \limits_{D \in {\cP}}
|D|$, and we shall assume

\begin{itemize}
    \item if $|B| \le \lambda,$ then
    \[\lambda \ge |\{\tp(\bar d,B):
    \lg(\bar d) < \kappa \text{ and } \Rang(\bar d) \cup B \text{ is }
    \bF\text{-atomic over } B\}|
    \]
\end{itemize}

\mn
(recall we say $A'$ is ${\bf F}$-atomic over $A$ if for every finite
$\bar d \subseteq A'$ we have $\tp(\bar d,A) \in {\bf F}(B)$ for some
$B \subseteq A$ of cardinality $< \kappa$).

\noindent
\underline{Pos 3}: $T$ and $F$ are as in \emph{Pos 2} but $\cP = \{ (B, D) \}$, where $B \subseteq D$.

\noindent
\underline{Pos 4}: As in \emph{Pos 3}, but $T$ is just a singleton (no used below).  

Now we define the relevant constructions and prove that  
the demands parallel to non-forking calculus hold.
\end{context}

\noindent
We can work in Pos 2 because
\begin{observation}
\label{3.3A1new}

\noindent
1) If Pos 1, then Pos 2.

\noindent
2) If Pos 2, then Pos 3 when we identify $D \in \cP$ with $(\emptyset, D)$. 
\end{observation}

\begin{definition}\label{3.3Bnew}
1) We say ${\cA} = \langle (A_\alpha,D_\alpha,B_\alpha):\alpha < 
\alpha_*\rangle$ is an $(\bF,{\cP})$-construction (we may  omit $(\mathbf
F,\cP)$ when clear from the context) \when:
\mn
\begin{enumerate}
\item[$(a)$]  $\langle A_\alpha \colon \alpha < \alpha_{\ast} \rangle$ is increasing continuous (and we stipulate
$A_{\alpha_{\star}} = \bigcup\limits_{\alpha < \alpha_{\ast}}  (A_\alpha \cup D_\alpha))$
\sn
\item[$(b)$]  $A_{\alpha+1} = A_\alpha \cup D_\alpha$
\sn
\item[$(c)$]  $B_\alpha \subseteq A_\alpha \cap D_\alpha$ 
\sn
\item[$(d)$] for every finite $\bar d \subseteq D_\alpha$
(or just $\bar d \subseteq D_\alpha \setminus B_\alpha$)
we have $\tp(\bar d,A_\alpha) \in \bF(B_\alpha)$
\sn
\item[$(e)$] moreover, for some $B = B_{\alpha}[\bar{d}, \cA] \subseteq B_{\alpha}$, we have $\tp(\bar{d}, A_{\alpha}) \in \bfF(B)$ and $\vert B_{\alpha} \vert < \kappa \Rightarrow B = B_{\alpha}$. 

\item[$(f)$]  for each $\alpha$ one of the following occurs: 

\begin{enumerate}
    \item[$\bullet_{1}$] $D_\alpha$     has cardinality $<\kappa$, 

    \item[$\bullet_{2}$] For \emph{Pos 2}: for some $D'_\alpha \in {\cP},D_\alpha \cong D'_\alpha$ 
    which means that there is an elementary mapping $h_\alpha$ from $D'_{\alpha}$ 
    onto $D_\alpha$ (where elementary mapping means in the sense of $\gC$ of course),

    For Pos 4: the pair $(D_{\alpha}, B_{\alpha})$ is isomorphic to some pair from $\cP$. 

    \item[$\bullet_{3}$] $D_{\alpha} \cong D_{\alpha}'$ for some $D_{\alpha}' \subseteq A_{\alpha}$ and $h_{\alpha}$ is an elementary mapping from $D_{\alpha}'$ onto $D_{\alpha}$. 
\end{enumerate}

\end{enumerate}
\mn
2)  For a construction ${\cA}$ as above we let 
$\alpha_* = \mathrm{lng}({\cA}),A_\alpha[{\cA}] = A_\alpha$ for $\alpha \le 
\alpha_*,D_\alpha = D_\alpha[{\cA}],B_\alpha = B_\alpha[{\cA}]$ 
and $A[{\cA}] = A_{\alpha_*}$.

\noindent
2A) We say that $\cA'$ is a \emph{successor} of $\cA$, \underline{when} (both are $(\cF, \cP)$-constructions) and: 

\begin{enumerate}
    \item[(a)] $\mathrm{lng}(\cA') = \mathrm{lng}(\cA)$,

    \item[(b)] $\cA \unlhd \cA'$.  
\end{enumerate}

\noindent
3)  We can replace $\alpha^*$ by any well ordering. We may replace
$D_\alpha[{\cA}]$ by (or add to it) $D'_\alpha[{\cA}] \in \cP$ and
$h_\alpha[{\cA}]$ from clause (e) if $|D_\alpha| \ge \kappa$.

\noindent
4) For $\alpha < \mathrm{lng}({\cA})$ we let $w_\alpha[{\cA}] = 
( \{\beta < \alpha \colon B_\alpha[{\cA}] \cap (A_{\beta+1}[{\cA}] \setminus 
A_\beta[{\cA}]) \ne \emptyset\})$ so $w_\alpha[{\cA}]$ has cardinality 
$< \kappa$ by clause (c) of Part (1) so $w_0 =\emptyset$.

\noindent
5) We call ${\cA}$ \emph{standard} \when:

\begin{enumerate}
    \item[(a)]  $\beta \in w_\alpha [{\cA}] 
    \Rightarrow w_\beta [{\cA}] \subseteq w_\alpha [{\cA}] \, \& \, 
    B_\beta[{\cA}] \subseteq B_\alpha[{\bar \cA}]$, 

    \item[(b)] if $\beta < \alpha_{\ast}$, $B_{\beta}[\cA]$ is of cardinality $< \kappa$ and $\beta \in w_{\alpha}[\cA]$, \underline{then} $B_{\beta}[\cA] \subseteq B_{\alpha}[\cA]$, 

    \item[(c)] if $\beta \in w_{\alpha}[\cA]$ and $\bar{b} \in {}^{\omega >}(D_{\beta}[\cA] \cap B_{\alpha}[\cA])$, \underline{then} $\tp(\bar{b}, A_{\beta})$ do not fork over $(B[\cA] \cap A_{\beta})$. (Equivalently, for every $\alpha < \mathrm{lng}(\cA)$, the pair $(w_{\alpha}, B_{\alpha}[\cA])$ is $\cA$-closed, see below.)
\end{enumerate}

\noindent
6) We say that the pair $w, B$ is $\cA$-closed, \underline{when}: 

\begin{enumerate}
    \item[$(a)$] $w \subseteq \mathrm{lng}(\cA)$ and $\beta \in w \Rightarrow w_{\beta}[\cA] \subseteq w$, 

    \item[$(b)$] $B \subseteq A[\cA]$, 

    \item[$(c)$] if $\beta < \mathrm{lng}(\cA)$ and $B \cap (A_{\beta + 1}[\cA] \setminus A_{\beta}[\cA]) \neq \emptyset$, then $\beta \in w$, 

    \item[$(d)$] if $\beta < \alpha_{\ast}$, $B_{\beta}[\cA]$ is of cardinality $< \kappa$  and $\beta \in w$, \underline{then} $B_{\beta}[\cA] \subseteq B$, 

    \item[$(e)$] if $\beta \in w$ and $\bar{b} \in {}^{\omega >}(D_{\beta}[\cA] \cap B_{\alpha}[\cA])$, \underline{then} $\tp(\bar{b}, A_{\beta})$ do not fork over $(B_{\alpha}[\cA] \cap B_{\alpha}[\cA])$, \underline{then} $\tp(\bar{b}, A_{\beta})$. 
\end{enumerate}

\noindent
6A) We say that $(w, B)$ is $(\cA, \kappa)$-closed if in addition $B$ is of cardinality $< \kappa$. 

\noindent
7)  For $\beta \le \mathrm{lng}(\cA)$ let ${\cA} \restriction \beta$
be defined naturally such that $\mathrm{lng}(\cA(\beta)) = \beta$.

\noindent
8) For $b \in A[{\cA}] = A_{\alpha_{\ast}}[\cA]$, let $\alpha(b)=\alpha(b,{\cA})$ be the 
$\beta$ such that $b \in A_{\beta+1}[{\cA}] \setminus A_\beta[{\cA}]$ 
but for $b \in A_0[{\cA}]$ we stipulate $\alpha(b) = -1$.

\noindent
9) For $b \in A[\cA]$ let 
$w_b [{\cA}] = w_{\alpha(b)}[\cA]$ (where we stipulate $w_{-1}[\cA]=\emptyset$, and
for a sequence $\bar b = \langle b_i:i < \ell g(\bar b)\rangle$ we let
$w_{\bar b}[{\cA}] = \bigcup \{w_{b_i}[\cA]:i < \ell g(\bar b)\}$ and 
$B_{\bar b}[{\cA}]=\cup\{ \{ b_\ell\} \cup B_{\alpha(b_\ell)}[{\cA}]:\ell < 
\mathrm{lng}(\bar b)\}$.

\noindent
10)  We may omit ${\cA}$ when clear from the context.
\end{definition}

\begin{fact}
\label{3.3Cnew}
1) For any $w \subseteq \mathrm{lng}(\cA)$ and $B \subseteq A[\cA]$, there is an $\cA$-closed pair $(w', B')$ such that $w \subseteq w'$, $B \subseteq B'$ and $w', B'$ are of cardinality $< \kappa + \vert w \vert^{+} + \vert B \vert^{+}$. 

\noindent
2) For any $(\bF, {\cP})$-construction ${\cA}$ there is a standard
$(\bF,{\cP})$-construction ${\cA}'$ such that:
\begin{enumerate}
\item[$(a)$]  $\mathrm{lng}({\cA}') = \mathrm{lng}({\cA})$
\sn
\item[$(b)$]  $A_\alpha[{\cA}'] = A_\alpha[{\cA}]$
\sn
\item[$(c)$]  $D_\alpha[{\cA}'] = D_\alpha[{\cA}]$ (and
$D'_\alpha[{\cA}'] = D'_\alpha[{\cA}],h_\alpha[{\cA}'] =
  h_\alpha[{\cA}]$)
\sn
\item[$(d)$]  $B_\alpha[{\cA}'] \supseteq B_\alpha[{\cA}]$
\sn
\item[$(e)$]  $w_\alpha[{\cA}'] \supseteq w_\alpha[{\cA}]$.
\end{enumerate}
\end{fact}

\begin{PROOF}{\ref{3.3Cnew}}
    1) Straightforward, choose $B_\alpha$, $w_\alpha$
    by introduction on $\alpha \leq \alpha
    _{\ast}$ for $\cA \rest \alpha$, $w \cap \alpha$, $B \cap \cA_{\alpha}[\cA]$ recalling that $\kappa$ is regular by
    \ref{3.3Anew}.

    2) Follows. 
\end{PROOF}

\begin{claim}\label{3.3Dnew}
Assume that:
\mn
\begin{enumerate}
\item[$(a)$]  ${\cA}$ is a standard $(\bF,{\cP})$-construction. 

\end{enumerate}  % 2026-04-21 05:58 

\noindent  % 2026-04-21 05:59 
1) Assume in addition that: 

\sn
\begin{enumerate}  % 2026-04-21 05:58 
\item[$(b)$]  $\pi$ is a one-to-one function from $\alpha = \mathrm{lng}({\cA})$ onto
the ordinal $\alpha'$,
\sn
\item[$(c)$]  if $\beta\in w_\alpha[{\cA}]$ then $\pi(\beta)< \pi(\alpha)$, and

\item[$(d)$] If $\beta < \mathrm{lng}(\cA)$, $\vert B_{\alpha} \vert \geq \kappa$ and $\bar{b} \subseteq D_{\beta}[\cA] \setminus B_{\beta}[\cA]$, \underline{then}
\[
\bar{b} \subseteq \bigcup \{ D_{\gamma} \colon \gamma < \beta \wedge \pi(\gamma) < \pi(p) \}.
\]
\end{enumerate}
\mn
\Then \, there is a standard $(\bF,{\cP})$-construction ${\cA}'$ such
that:
\mn
\begin{enumerate}
\item[$(i)$]  $\mathrm{lng}({\cA}') = \alpha'$
\sn
\item[$(ii)$]  $D_\alpha[{\cA}] = D_{\pi(\alpha)}[{\cA}']$
\sn
\item[$(iii)$]  $w_{\pi(\alpha)}[{\cA}'] = \{\pi(\beta):\beta \in
w_\alpha[{\cA}]\}$ and $B_{\pi(\alpha)}[{\cA}'] = B_\alpha[{\cA}]$
\sn
\item[$(iv)$]  $A_0[{\cA}'] = A_0[{\cA}]$
\sn
\item[$(v)$]  $A_{\pi(\alpha)}[{\cA}'] = A_0[{\cA}'] \cup \bigcup
\{D_\beta[{\cA}'] \colon  \pi(\beta) < \pi(\alpha)\}$.
\end{enumerate}

\noindent
2) If $(w', B')$ and $(w', B'')$ are $\cA$-closed, then $(w' \cap w'', B' \cap B'')$ is $\cA$-closed.

% \mn
% 2) Assume that
% $w_2 \subseteq u_2 \subseteq \mathrm{lng}({\cA})$, and $u_1,u_2$ are
% ${\cA}$-closed and $u=u_1 \cap u_2,v=v_1 \cap v_2$ \then \, 
% for any finite 
% \[
% \bar d \subseteq \bigcup_{\beta\in v_2} D_\beta \cup 
% \bigcup_{\gamma\in u_2} B_\gamma
% \]

% the type
% \[
% \tp(\bar d,\bigcup_{\beta\in v_1} D_\beta \cup 
% \bigcup_{\gamma \in u_1} B_\gamma)
% \]

% belongs to
% \[
% \bF[\bigcup\limits_{\beta\in v} D_\beta \cup 
% \bigcup_{\gamma\in u} B_\gamma].
% \]

\noindent
3) Assume also that   $B \subseteq A_{\lg(\cA)}[{\cA}] $ and $|B|< \kappa$.

\mn
\Then \, there is a $(\bF,{\cP})$-construction ${\cA}'$ satisfying:
\mn
\begin{enumerate}
\item[$(\alpha)$]  ${\cA}' = \langle A'_\alpha,D'_\alpha,B'_\alpha:
\alpha < 1 + \mathrm{lng}({\cA}')\rangle$, we use ordinal addition,
\sn
\item[$(\beta)$]  $A'_0 = A_0[{\cA}]$
\sn
\item[$(\gamma)$]  $A'_1 = A'_0 \cup B$
\sn
\item[$(\delta)$]  $D'_0 = B$
\sn
\item[$(\varepsilon)$]  $A'_{1+\alpha} = A'_{1} \cup A_\alpha [{\cA}]$
\sn
\item[$(\zeta)$]  $D'_{1+\alpha} = D_\alpha$
\sn
\item[$(\eta)$]  $B'_{1+\alpha} \supseteq B_\alpha$.
\end{enumerate}
\mn
4) In part (3), if for some ${\cA}$-closed $u\subseteq \mathrm{lng}({\cA})$ 
we have $\cup \{B_\alpha:\alpha\in u\} \subseteq B \subseteq \cup\{
D_\alpha:\alpha\in u\}$ \then \, we can let $B'_{1+\alpha}=B_\alpha\cup B$. 
\end{claim} 

\begin{PROOF}{\ref{3.3Dnew}}
    For \ref{3.3Dnew}, recall that in Definition \ref{3.3Bnew}(1)(d) the set $B_{\alpha}[\bar{b}, \cA]$ is part of $\cA$ and then see \ref{3.3Bnew}(6)(d). This helps in particular in \ref{3.3Dnew}(2). For the others, recall that the proof of \cite[Ch.IV 3.3,3.2,pg.176]{Sh:c}, (of course, 
    we can strengthen \ref{3.3Dnew}(1),(3)); [e.g. for part (4) show by 
    induction on $\alpha \le \mathrm{lng}({\cA})$ then $\bar d \subseteq B
    \Rightarrow \tp(\bar d,A_\alpha[{\cA}]) \in {\bf F}(B \cap
    A_\alpha[{\cA}])$; for part (3), just find $B' \supseteq B$ which is 
    as in part (4); part (2) can be proved by induction on $\mathrm{lng}({\cA})$].
\end{PROOF}

\begin{definition}\label{3.3Enew}
1)  We say $({\cA},\bar f)$ is a automorphic
    $(\bF,{\cP})$-construction \when \,:
\mn
\begin{enumerate}
\item[$(a)$] ${\cA}$ is a standard $(\bF,{\cP})$-construction
\sn
\item[$(b)$]  $A_0[{\cA}]= \emptyset$
\sn
\item[$(c)$]  $\mathrm{lng}({\cA}) < \lambda^+$ 
\sn
\item[$(d)$] $\bar f = \langle f_{i, g}: i \le \mathrm{lng}({\cA}),g \in 
{\cG}_{A_i[{\cA}]}\rangle$ where ${\cG}_A$ is a set of elementary mappings
from a subset of $A$ into a subset of $A$ such that $g \in \cG_{A_{i}[\cA]}$ implies $g^{-1} \in \cG_{A_{i}[\cA]}$, 
\sn
\item[$(e)$]  $f_{i,g}$ is an elementary mapping with domain and range
$\subseteq A_i[{\cA}]$
\sn
\item[$(f)$] $f_{i,g}$ is increasing continuous with $i,f_{i,g^{-1}}
= (f_{i, g})^{-1}$
\sn
\item[$(g)$]  if $g \in {\cG}_{A_i[{\cA}]} \setminus \bigcup\limits_{j<i}
{\cG}_{A_j[{\cA}]}$ then $f_{i,g} = g$
% \sn
% \item[$(h)$]  for each $i \le \mathrm{lng}({\cA})$ and $g \in {\cG}_{A_i[{\cA}]}$,
% for some ${\cA}$-closed set $w$
% \[
% \Dom(f_{i,g}) = \bigcup \{D_\beta:\beta \in w\}.
% \]

\item[$(h)$] Let $\cG_{A_{< i}[\cA]} = \bigcup_{j < i} \cG_{A_{j}[\cA]}$ for $i \leq \mathrm{lng}(\cA)$. 
\end{enumerate}
\mn
2)  The cardinality of $({\cA},\bar f)$ written $\card({\cA},
\bar f)$ is the one of ${\cA}$, i.e. $|\mathrm{lng}({\cA})| + |A[{\cA}]|$.

\noindent
3)  For $\beta \le \mathrm{lng}({\cA})$ let $({\cA},\bar f)
\restriction \beta$ be defined naturally.
\end{definition}

\begin{fact}\label{3.3F}
    Assume that $({\cA},\bar f)$ is an automorphic 
    $(\bF,{\cP})$-construction. Let $\alpha = \mathrm{lng}({\cA})$, and $B \subseteq D$, $B^{\ast} \subseteq A[\cA]$, $B_{1} \subseteq A[\cA]$ and $g$ an elementary mapping from $B_{1}$ into $\cA[A]$. Then:

    1) If $\vert D^{\ast} \vert < \kappa$, \underline{then} there is some $\cA'$ such that: 

    \begin{enumerate}
        \item[$(a)$] $\cA'$ is a successor of $\cA$, 

        \item[$(b)$] $B_{\alpha}[\cA'] = B$, 

        \item[$(c)$] $D_{\alpha}[\cA_{\alpha}]$ is isomorphic over $B$ to $D^{\ast}$.

        \item[$(d)$] $\cG_{A_{\alpha}[\cA ' ]} % 2026-04-21 05:55 ]'  } 
           = \cG_{A_{< \alpha}[\cA]}$, 

        \item[$(e)$] $f_{\alpha, g}[\cA'] = f_{\beta, g}[\cA]$, when $\beta < \alpha$, $g \in \cG_{A_{\beta[\cA]}}$.  
    \end{enumerate}

    \noindent
    2) If $D' \in \cP$, $D^{\ast} \cong D'$ and $\vert B \vert < \kappa$, \underline{then} there is $\cA'$ such that (a)-(e) above holds. 

    \noindent
    3) If $g \notin \cG_{A_{< \alpha}[\cA]}$, $\vert D^{\ast} \vert < \kappa$, \underline{then} there is $\cA'$ such that (a), (b), (c), and (e) as in first part holds and

    \begin{enumerate}
        \item[(d)] $\cG_{A_{\alpha}[\cA]'} = \cG_{A_{< \alpha}[\cA]} \cup \{ g, g^{-1} \}$. 
    \end{enumerate}

    \noindent
    4) A special case of part (3) is that $u \subseteq \alpha$ has no last member, $g_{\beta} \in \cG_{A_{\beta}[\cA]}$ for $\beta \in u$ is $\subseteq$-increasing with $\beta$ and $g \coloneqq \{ g_{\beta} \colon \beta \in u \}$.

    \noindent
    5) Another special case of part (3) is $\beta < \alpha$, $g_{\beta} \in \cG_{A_{\beta}[\cA]}$ and $g \in \cG_{A_{\beta}[\cA]}$ extend $g_{\beta}$ and has domain $A_{\alpha}[\cA]$.

    % the quadruple $(w, B, D, g)$ satisfies one of the following:

    % \begin{enumerate}
    %     \item[(A)] 

    %     \begin{enumerate}
    %         \item[(a)] $B \subseteq A[\cA]$, $A \subseteq D'$, $(w, B)$ is $\cA$-closed and $g$ is an elementary map from $B$ into $\cA[\cA]$. 

    %         \item[(b)] \blueq{  \begin{enumerate}
    % \item[$(a)$]  $\mathrm{lng}({\cA}') = \mathrm{lng}({\cA}) + 1 = \alpha+1$
    % \sn
    % \item[$(b)$]  $\card({\cA}',\bar f') \le \card({\cA},\bar f) + |D'|+1$
    % \sn
    % \item[$(c)$]  $B_\alpha[{\cA}']=B$, and $w_{\alpha}[\cA^{\inc}] = w$ 
    % \sn
    % \item[$(d)$]  $D_\alpha[\cA'] = D$
    % \sn
    % \item[$(e)$]  ${\bar{\cA}}'\restriction \alpha = {\cA}$
    % \sn
    % \item[$(f)$]  $f_{\alpha +1,g}[{\cA}'] = f_{\alpha, g}[{\cA}]$
    % for $g \in {\cG}_{A_\alpha}[{\cA}']$.
    % \end{enumerate}}
    %     \end{enumerate}

    %     \item[(B)] One of the following hold: 

    %     \begin{enumerate}
    %         \item[(a)] $\vert D' \vert < \kappa$. 

    %         \item[(b)] $D' \in \cP$. 

    %         \item[(c)] $D' \subseteq \cA[\cA]$. 
    %     \end{enumerate}
    % \end{enumerate}
\end{fact}

\begin{PROOF}{\ref{3.3F}}
Straightforward (by the existence of non-forking extensions), see more details in \ref{3.3H}. 
\end{PROOF}

\begin{definition}
\label{3.3Gnew}
For automorphic $(\bF,{\cP})$-constructions $({\cA}^1,\bar f^1)$,
$({\cA}^2,\bar f^2)$ let $({\cA}^1,\bar f^1) \le ({\cA}^2,
\bar f^2)$ means: $\bar{\cA}^1 \trianglelefteq \bar{\cA}^2$ and 
$\bar f^1 = \langle f^2_{i,g}: i \le \mathrm{lng}({\cA}^1),g \in 
{\cG}_{A_i[{\cA}^1]}\rangle$.
\end{definition}

\begin{claim}
\label{3.3H}
If $({\cA}, \bar f)$ is an automorphic $(\bF,{\cP})$-construction,
$i \le \mathrm{lng}({\cA}),g_* \in {\cG}_{A_i}[{\cA}]$ \then \,  for
some automorphic $(\bF,{\cP})$-construction $({\cA}',\bar f')$ we
have:
\mn
\begin{enumerate}
\item[$(a)$]  $({\cA},\bar f) \le ({\cA}',\bar f')$
\sn
\item[$(b)$]  $\card({\cA}',\bar f') \le \card({\cA},\bar f) +
  \aleph_0$
\sn
\item[$(c)$]  $\Dom((f_{j, g_{\ast}})[\cA']) = A_i[\cA]$ where $j = \mathrm{lng}({\cA}')$.
\end{enumerate}
\end{claim}

\begin{PROOF}{\ref{3.3H}}
Let ${\cA}^0 = {\cA} \rest i$, then by \ref{3.3Dnew} 
we can find a standard $(\bF,{\cP})$-construction ${\cA}^1$ and 
$j_1 \le \mathrm{lng}({\cA}^1)$ such that $A_0[{\cA}^1] = A_0[\cA^0],
A[\cA^1] = A[\cA^0]$, and $A_{j_1}[\cA^1]$ is $\Dom(f_{i,g_*}[{\cA}])$.
We can find an elementary mapping $h$ such that: $\Dom(h) =
A[{\cA}^0],h$ extends $f_{i,g_*}$, and for every $\beta 
\in [j_1,\mathrm{lng}({\cA}^1)]$, we have

\[
\bar d \subseteq h(D_\beta[{\cA}^1]) \Rightarrow \tp(\bar d,
A[{\cA}]\cup h(A_\beta[{\cA}^1])) \in \bF(h(B_\beta)).
\]

\mn
Now we define the automorphic $(\bF,{\cP})$-construction ${\cA}':
\mathrm{lng}({\cA})' = \mathrm{lng}({\cA})+ (\mathrm{lng}({\cA}^1)-j_1)$, and 
${\cA}' \restriction \mathrm{lng}({\cA}) = {\cA},D_{\mathrm{lng}({\cA})+\zeta}[{\cA}']
 = h(D_{j_1+\zeta}[{\cA}^1]),B_{\mathrm{lng}({\cA}) + \zeta}[{\cA}'] =
h(B_{j_1+\zeta}[{\cA}^j])$.  Define $\bar f' = \langle f'_{\alpha,g}:
\alpha \le \mathrm{lng}({\cA}'),g \in {\cG}_{A_\alpha[{\cA}']}\rangle$ as follows: for 
$\alpha \le \mathrm{lng}({\cA}'),g \in {\cG}_{A_\alpha[{\cA}']}$ we let:
\mn
\begin{enumerate}
\item[$(\alpha)$]  if $\alpha \le \mathrm{lng}({\cA})$, then $f'_{\alpha,g} 
= f_{\alpha,g}$

\item[$(\beta)$]  if $\alpha \ge \mathrm{lng}({\cA})$ and $g \notin 
{\cG}_{A_{\mathrm{lng}({\cA})}}[\cA]$ then $f'_{\alpha,g} = g$

\item[$(\gamma)$]  if $g \in {\cG}_{A_{\mathrm{lng}({\cA})}}[\cA]$, and 
$g \ne g_*,g^{-1}_*$ then $f'_{\alpha,g} = f_{\mathrm{lng}({\cA}),g}$
\sn
\item[$(\delta)$]  if $g=g_*$ and $\alpha < \mathrm{lng}({\cA}')$ then let
$f'_{\alpha,g}$ be $f_{\mathrm{lng}({\cA}),g}$
\sn
\item[$(\epsilon)$]  if $g=g_*$ and $\alpha = \mathrm{lng}({\cA}')$ then let 
$f'_{\alpha,g}$ be $h$. 
\end{enumerate}
\mn
Now check.
\end{PROOF}

\begin{claim}\label{3.3Inew}
$\delta < \lambda^+$ is a limit ordinal and if 
$\langle({\cA}^\zeta,\bar f^\zeta):\zeta<\delta\rangle$ is
increasing (sequence of automorphic $(\bF,{\cP})$-constructions),
\then \, it has a $\lub({\cA}^\delta,\bar f^\delta)$ i.e.
\[
\zeta < \delta \Rightarrow ({\cA}^\zeta,\bar f^\zeta) \le 
({\cA}^\delta,\bar f^\delta)
\]
\[
\mathrm{lng}({\cA}^\delta) = \bigcup\limits_{\zeta< \delta} \mathrm{lng}({\cA}^\zeta)
\]
\[
\card({\cA}^\delta) \le |\delta| + \sup\limits_{\zeta< \delta}
\card({\cA}^\zeta).
\]
\end{claim}

\begin{PROOF}{\ref{3.3Inew}}
    Straightforward.
\end{PROOF}

\begin{claim}
\label{3.3Jnew}
For every $\theta = \cf(\theta) \in [\kappa,\lambda]$
there is an automorphic $(\bF,{\cP})$-construction ${\cA}$ of
cardinality $\lambda$ such  that $\cf(\mathrm{lng}({\cA}))=\theta$ and
\mn
\begin{enumerate}
\item[$\otimes_1$]  for $g \in
  {\cG}_{A[{\cA}]},f_{\mathrm{lng}({\cA}),g}[{\cA}]$ is 
an automorphism of $A[{\cA}]$ 
\sn
\item[$\otimes_2$]  if $B \subseteq A[{\cA}],|B|<\kappa,B \subseteq B'$
 and $|B'|< \kappa$ or $B'$ is isomorphic to some $B'' \in {\cP}$
 \then \,

\begin{equation*} 
\begin{array}{clcr}
\mathrm{lng}({\cA}) = \otp\{\alpha:& \text{ there is an elementary mapping } h 
\text{ from } B' \text{ onto } D_\alpha\\
   & \text{ which is the identity on } B, \text{ and } B_\alpha = B\}.
\end{array}
\end{equation*}
\end{enumerate}
\end{claim}

\begin{PROOF}{\ref{3.3Jnew}}
By bookkeeping and the assumptions (in \ref{3.3Anew}) on $\lambda$.
\end{PROOF}

\begin{claim}
\label{3.3Lnew}
Suppose:
\mn
\begin{enumerate}
\item[$(a)$]  ${\cA}$ is an $(\bF,{\cP})$-construction,
\sn
\item[$(b)$]  $\chi^*$ large enough and $N_1\prec N_2 \prec
  ({\cH}(\chi^{\ast}),\in)$
\sn
\item[$(c)$]  ${\cA} \in N_1$ and the monster model ${\gC}$ belongs
to $N_1$ and $N_1 \cap \kappa$ is an ordinal (possibly $\kappa$ itself, if
$\kappa=\aleph_0$ this is necessarily the case)
\sn
\item[$(d)$]  $\bar b \in {}^{\omega>}(A[{\cA}])$ and 
$w_{\bar b}[{\cA}] \cap N_2 \subseteq N_1$ (on $w_{\bar b}$ see 
Definition \ref{3.3Bnew}(9))
\sn
\item[$(e)$]  if $\alpha \in w_{\bar b}\cap N_1$ then 
$\tp(\bar b \rest D_\alpha[{\cA}],N_2 \cap A[{\cA}])$ does not fork 
over $N_1 \cap A[{\cA}]$ where for $\bar b = \langle b_\ell:
\ell < n\rangle$ we let $\bar b \rest D_\alpha =
\langle b_\ell:\ell<n,b_\ell\in D_\alpha\rangle$.
\end{enumerate}
\mn
\Then \, $\tp(\bar b,A[{\cA}] \cap N_2)$ does not fork over  $A[\cA] \cap N_1$.
\end{claim}

\begin{PROOF}{\ref{3.3Lnew}}
By \ref{3.3Dnew}.
\end{PROOF}

\noindent
Now to complete the proof of \ref{7.13} we turn back to the model
$M_I$ we have constructed before \ref{3.3Anew}.

\begin{fact}
\label{3.3Knew}
For the context \ref{3.3Anew}(Pos 1), (so $I \in K^\omega_{\tr},I \subseteq
{}^{\omega\ge}\lambda$ is closed under initial segments of cardinality
$\le \lambda)$, letting $\kappa = \aleph_0,\cP = \{D_I\}$ 
(see \ref{3.3Anew}(Pos 1)) for some ${\cA} = {\cA}^I$ we have
\mn
\begin{enumerate}
\item[$(A)$]  ${\cA}$ is a standard $(\bF,{\cP}_I)$-construction
$A_1[{\cA}] = D_I$ and $A_0[\cA] = \emptyset$
\sn
\item[$(B)$]  $A[{\cA}^I]$ is strongly $\aleph_\epsilon$-saturated of
cardinality $\lambda$
\sn
\item[$(C)$]  $A[{\cA}^I]$ is equal to the model $M_I$ constructed
during the beginning of the proof of \ref{7.12}.
\end{enumerate}
\end{fact}

\begin{remark}
\label{3.14new3}
We do not actually use clause (C), as we can just let $M_I$ be 
the model with universe $A[{\cA}^I]$.
\end{remark}

\begin{PROOF}{\ref{3.3Anew}}
    Straightforward for clause (C) recall Definition \ref{3.3Enew}, Claim
    \ref{3.3H} (or just use the model constructed in \ref{3.3Jnew}).
\end{PROOF}

\begin{fact}
\label{7.13A}
If $\chi$ is regular large enough, $\cA^I \in \cH(\chi),{\cA}^{I} \in N_{1}
\prec N_{2} \prec ({\cH}(\chi),\in,<^*_{\chi}),N_\ell \cap 
\kappa_r (T) \in \kappa_r (T)+1,\bar b \in M_{I}$, and 
$w_{\bar b}[{\cA}^I] \cap N_{2} \subseteq N_{1}$ and 
$\alpha \in w_{\bar b}[{\cA}] \Rightarrow \tp(\bar b \rest
D_\alpha[\cA^I],N_2 \cap M_I)$ does not fork over $N_1\cap
M_I$. \Then \, $\tp(\bar b,N_{2} \cap M_{I})$ does not fork over 
$N_{1} \cap M_{I}$.
\end{fact}

\begin{PROOF}{\ref{7.13A}}
By \ref{3.3Lnew}.
\end{PROOF}

\noindent
For the rest for simplicity assume $\kappa=\aleph_0$. 
\begin{fact}
\label{7.13B}
If $\chi$ is regular, $I \in K^\omega_{\tr},{\cA}^{I} \in {\cH}(\chi),
{\cA}^{I} \in  N_{1} \prec N_{2} \prec ({\cH}(\chi),\in,<^*_{\chi}),
N_\ell\cap \kappa_r(T)\in \kappa_r(T)+1$ 
and $\eta \in P^I_{\omega},n < \omega,\eta \rest n \in N_{1},\eta(n)
\in N_{2} \setminus N_{1}$ and ${\gC}\in N_1$ \then \, 
$\tp(\bar a_{\eta},N_{2} \cap M_{I})$ fork over $N_{1} \cap M_{I}$.
\end{fact}

\begin{PROOF}{\ref{7.13B}}
Let $\cA = \cA^{I}$ be as in \ref{3.3Knew} and let 
$A_i^I = A_i[{\cA}^I]$, for $i \le \alpha_I = \mathrm{lng}({\cA}^I)$, 
and recall $A^I_1 = \{a_\eta:\eta\in I\}$.
For $\bar c \subseteq N_\ell \cap M_I$, clearly $\tp(\bar c,A^I_1)$ does
not fork over $\bigcup\{B_\gamma \cap A^I_1:\gamma \in w_{\bar c}\}
\cup (\bar c \cap A^I_1) \subseteq N_\ell\cap A^I_1$, hence $\tp(A^I_1,
N_\ell \cap M_I)$ does not fork over $N_\ell\cap A^I_1$ recalling
$a_\eta \in A^I_1$ we have hence
$\tp(a_\eta,N_\ell \cap M_I)$ does not fork over $N_\ell \cap A^I_0$.

But $\tp(\bar a_{\eta},\{\bar a_{\nu}:\nu \in I,
\neg(\eta \rest n \triangleleft \nu)\})$
does not fork over $\{\bar a_{\nu}:\nu \trianglelefteq \eta \rest n\}$,
(why? as $\langle \bar a_\eta:\eta\in I \rangle$ is a non-forking
tree).  Now the set $\{\bar a_\nu:\nu \trianglelefteq \eta\rest n\}$
is  a subset of $N_1$ hence $\tp(\bar a_{\eta},\{\bar a_{\nu}:a_{\nu} 
\in N_1$ and $\neg(\eta \rest n \triangleleft \nu)\})$ 
does not fork over $\{\bar a_{\nu}:\nu \trianglelefteq \eta \rest n\}$
so by transitivity and previous the  sentence, $\tp(\bar a_{\eta},M_{I} \cap
N_{1})$ does not fork over $\{a_{\eta \rest m}:m \le n\}$.

On the other hand $\tp(\bar a_{\eta},M_{I} \cap  N_{2})$ forks over
it (otherwise $\tp(a_{\eta},\{\bar a_{\eta \rest \ell}:\ell \le n + 1\})$
does not fork over $\{a_{\eta \rest \ell}:\ell \le n\}$, contradiction), so the
conclusion follows. 
\end{PROOF}

\begin{fact}
\label{7.13C}
If $I$ is super unembeddable into $J$ \then \, $M_{I}$ is not
isomorphic to $M_{J}$.
\end{fact}

\begin{PROOF}{\ref{7.13C}}
Straightforward by the definition and Facts \ref{7.13A},
\ref{7.13B}, but we give some details.  \Wilog \, $T$ is countable so 
$\kappa_r(T) = \aleph_1$, (justified in the proof of \ref{7.13D} below).

Let $f$ be an isomorphism from $M_{I}$ onto $M_{J}$
and $\chi$ be regular large enough. We can find $\langle 
M_{n},N_{n}:n < \omega \rangle $  an increasing sequence of
elementary submodels of $({\cH}(\chi),\in,<^*_{\chi})$ and $\eta$ as
in Definition \ref{7.1} such that ${\cA}^{I},{\cA}^{J},f$ belongs to
$N_{0}$.

By Fact \ref{7.13A}, $\tp(f(\bar a_{\eta}),M_{J} \cap N_{n})$ does not fork
over $M_{J} \cap M_{n}$ for every $n$ large enough.  
By Fact \ref{7.13B}, $\tp(\bar a_{\eta},M_{I} \cap N_{n})$ forks
over $M_{I} \cap  M_{n}$.  As $f$ maps $M_{I} \cap N_{n}$ onto 
$M_{J} \cap N_{n}$ and $M_{I} \cap M_{n}$ onto $M_{J} \cap 
M_{n}$ and $\bar a_{\eta}$ to $f(\bar a_{\eta})$ we finish.
\end{PROOF}

\begin{fact}
\label{7.13D}
If $I$ is super unembeddable into $J$ \then \, $M_{I}$ is not
elementarily embeddable into $M_{J}$.
\end{fact}

\begin{PROOF}{\ref{7.13D}}
    Let $\tau_{0}$ be a countable sub-vocabulary of $\tau(T)$ such that
    for $\eta \in P^I_{\omega},n < \omega$ we have $\tp(a_{\eta},
    \{a_{\eta \rest \ell}:\ell < n\})$ forks over $\{a_{\eta \rest \ell}: 
    \ell < n\}$ even in the $\tau_{0}$-reduct of $M_{I}$.  Suppose
    $f$ is an elementary embedding of $M_{I}$ into $M_{J}$ (or just of their
    $\tau_{0}$-reduct) and we shall get a contradiction.  Modulo the proof of
    \ref{7.13C}, it suffice to prove:
    \mn
    \begin{enumerate}
    \item[$(*)_1$]  if $\tp(\bar a,A)$ does not fork over $B \subseteq A$
      in $\gC$ then $\tp_{\bbL(\tau_0)}(\bar a,A)$ does not fork over $B$
      in $\gC \rest A$.
    \end{enumerate}
    \mn
    [Why?  By character by ranks, \cite[Ch.III]{Sh:c}.]

    \begin{fact}\label{xyz.1}
    We have: 
        \begin{enumerate}
            \item[$(*)_2$] if  ${\cA}^{I}$, ${\cA}^{J},f \in N \prec ({\cH}(\chi), \in,<^*_{\chi})$ \then \, $\tp_{\bbL(\tau_0)}(N \cap M_{J},\Rang(f))$  does not fork over $N \cap \Rang(f)$ in $M_{J} \rest \tau_{0}$.
        \end{enumerate}
    \end{fact}

    \begin{PROOF}{\ref{xyz.1}}
        Why $(*)_2$ holds?  As $T$ is stable and $\tau_0$ is countable for 
        every $\bar c \in M_{J}$ there is a countable $B^*_{\bar c} \subseteq
        \Rang(f)$ such that $\tp_{\bbL(\tau_0)}(\bar c,\Rang(f))$ does not fork over
        $B^*_{\bar c}$ in $M_{J} \rest \tau_{0}$.  As $\tau_{0}$ is countable, $T$
        stable,  clearly $\bar c \in N \cap M_J \Rightarrow B^*_{\bar c} \subseteq N \cap \Rang(f)$.  So for $\bar c \in N \cap M_J$ the type  $\tp_{\bbL(\tau_0)}(\bar c,\Rang (f))$ does not fork over  $N \cap \Rang(f)$, as required. 
    \end{PROOF}
    So we have finished proving~\ref{7.13D}. 
\end{PROOF}

\begin{PROOF}{\ref{7.13}}

\noindent
\underline{Continuation of the  Proof of \ref{7.13}}: 

Let $\langle I_\alpha:\alpha < 2^\lambda\rangle$
exemplify that $K^\omega_{\tr}$ has the $(2^\lambda,\lambda,\mu,
\aleph_0)$-bigness property and let $M_\alpha = M_{I_\alpha}$. Now
apply \ref{3.3Knew} and \ref{7.13D}. 
\end{PROOF}

\begin{remark}
\label{7.13f}
In \ref{7.13C}. \ref{7.13D} weaker versions of  unembeddable suffice.
\end{remark}
\bigskip

\subsection {On Generalizations and Abelian $p$-groups} \

Having finished our digression to stability theory, we look at a strengthening
of \ref{7.11}, which will be used in \ref{7.15}.

\begin{theorem}
\label{7.14}
If $\lambda > \mu + \aleph_1$ and $\mu \ge \kappa$ 
\then \, $K^\omega_{\tr}$ has the full
$(\lambda,\lambda,\mu,\kappa)$-super$^+$-bigness property which means
that in the Definition \ref{7.2} we replace super by super$^+$ which
means that we define ``$I \in K_{\tr}^{\omega}$ is $(\mu, \kappa)$-super-unembeddable into $J \in K_{\tr}^{\omega}$'' as in Definition \ref{7.1} but replace $(*)$ there by:
\mn
\begin{enumerate}
\item[$(*)^+$]   like $(*)$ of Definition \ref{7.1} adding
\sn
\begin{enumerate}
\item[$(v)^+$] for each  $n, \eta \rest n \in M_{n}$ and $\eta \rest 
(n+1) \in N_{n} \setminus M_{n}$
\sn
\item[$(vii)$]   if  $\nu \in P^J_{\omega}$ is in the closure of $M_n 
\cap I$, (i.e. $\{\nu \rest \ell:\ell < \omega\} \subseteq M_n)$ 
then $\nu \notin N_{n} \setminus M_{n}$
\sn
\item[$(viii)$]   there is $M \prec ({\cH}(\chi),\in,<^*_{\chi})$ 
such that: $\bigcup\limits_{n < \omega} M_{n} \subseteq M$ and 
$\eta \notin M$,  but for each  $n$  we have:
\[
\nu \in P^J_{\omega} \, \& \, \bigwedge\limits_{\ell < \omega} 
\nu \rest \ell \in M_n \Rightarrow \nu \in M.
\]
\end{enumerate}
\end{enumerate}
\end{theorem}

\begin{remark}
\label{7.14a}
Compare with \ref{7.7} here + \ref{7.6}(2) here.
\end{remark}

\begin{PROOF}{\ref{7.14}}
The proof is done by cases, so to enlighten the reader, we first list them.

If $\lambda$ is regular $> \aleph_1$:  by case 1

If $\lambda$ is singular and $(\exists \chi < \lambda)[\chi < \lambda 
\le \chi^{\aleph_{0}}]$:  by case 2;

Case 3, see below. 

If neither case 1 nor case 2 but $(\exists \chi)[\mu \le \chi^{\aleph_{0}} <
\lambda \le 2^\chi]$:  by case 4;

So we are left with $\lambda$ strong limit singular.  

If $\lambda = \aleph_{\alpha +\omega }$: by case 6;

If $\cf(\lambda) > 2^{\aleph_{0}}$: by case 5; 

In the remaining case let $\theta = \mu^{+\omega}$, so necessarily
$\theta < \lambda$, hence for some increasing sequence 
$\langle \lambda_{n}:n < \omega \rangle$ of regular cardinals with limit
$\theta < \lambda,\cf(\Pi \lambda_{n}/\finite) = \theta^+,\lambda_{n}
> \mu$ (exist see \cite[3.22=Lpcf.8]{Sh:E62}), now for $\epsilon <
2^{\aleph_{0}},\ga_{\epsilon}$ be an infinite subset of 
$\{\lambda_{n}:n < \omega\}$ such that $\epsilon \ne \zeta
\Rightarrow |\ga_{\epsilon} \cap \ga_{\zeta}| < \aleph_{0}$.

So case 3 applies.

\noindent
\underline{Case 1}:  $\lambda$ regular $>\aleph_1$. (In fact also the
requirements from Def. \ref{7.3}(G$^+$) of ``super$^{7^+}$" hold.)

Use the proof of \ref{7.6}(1) with minor changes:

Choosing $\bar C$, by \ref{7.8B} we can add the demands:
\mn
\begin{enumerate}
\item[$(c)$]  for any $\zeta < \lambda$, for every club $E$  of 
$\lambda$ we have $\{\delta \in S_{\zeta}:C_{\delta} \subseteq E\}$  
is stationary
\sn
\item[$(d)$]   $\alpha \in C_{\delta} \Rightarrow \cf(\alpha) >
\aleph_{0}$.
\end{enumerate}
\mn
Choosing $\delta(*) \in E$ we demand also $C_{\delta(\ast)} \subseteq E$, and
let $m_\ell=2\ell$.

So the condition for super$^{7^+}$ (Def \ref{7.3}(G$^+$) (hence from Def.\ref{7.1}) holds. Clause (v)$^+$ holds by the choice of $\eta_\delta,m_\ell,M_\ell,N_\ell$.  Clause (vii) holds by clause (d), i.e. $\cf(\eta_\delta(m))>\aleph_0,
\eta_\delta(m)\in E$. 

Lastly, clause (viii) is exemplified by $M = M^*_{\delta(*)}$.

\noindent
\underline{Case 2}:   There is $\chi,\chi < \lambda \le 
\chi^{\aleph_{0}}$ and $\lambda$ is singular.

Just Claim \ref{7.8} applies; i.e. the proof of \ref{7.8} but by
\ref{7.8F}(2) we can  choose there $C_{\delta}$ such that
\mn
\begin{enumerate}
\item[$(*)$]  $\alpha \in C_{\delta} \, \& \, \alpha > \sup(C_{\delta}
\cap \alpha) \Rightarrow \cf(\alpha) > \aleph_{0}$.
\end{enumerate}
\mn
The proof gives also $(v)^+,(vii),(viii)$ and even
\mn
\begin{enumerate}
\item[$(vii)^{++}$]  if $\eta \in P^J_{\omega},\{\eta \rest \ell:
\ell < \omega\} \subseteq M_{n}$ then $\eta \in M_{n}$.
\end{enumerate}
\mn 
[Why?  By $(*)$ above or see Case 3's proof; note that if 
$\eta = \langle i \rangle \mathop{\otimes}\limits_\lambda
\nu$ (or $\eta = \langle i \rangle \char94 \nu)$ and $\nu \in I_{\rho_i}$
then necessarily $i\in M_n$ hence $I_{\rho_i} \in M_n$.]
\mn
\begin{enumerate}
\item[$(ii)^+$]  $\mu \subseteq M_{n}$.
\end{enumerate}
\medskip

\noindent
\underline{Case 3}:  $\lambda$ singular, and for some $\theta,\lambda > \theta
\ge \mu  + \aleph_{1},\cf (\theta) = \aleph_{0}$ and there is a sequence 
$\langle \ga_{\epsilon}:\epsilon < \cf (\lambda) \rangle$  as in \ref{7.9}.

The proof of \ref{7.9}\ gives  (v)$^+$ trivially.  Again (as in
the proof of \ref{7.8})
\[
[\eta \in P^J_{\omega} \, \& \, \bigwedge\limits_{\ell} \eta \rest 
\ell \in  M^*_{\alpha} \Rightarrow \eta \in M^*_{\alpha +1}]
\]

hence
\[
[\cf(\alpha) > \aleph_{0} \, \& \, \eta \in P^J_{\omega} \, \& \,
\bigwedge\limits_{\ell} \eta \rest \ell \in M^*_{\alpha} \Rightarrow \eta
\in M^*_\alpha].
\]

hence clause (vii) holds. 

Lastly,  it follows that $M^*_{\delta(*)}$  satisfies the 
requirement in clause (viii).
\medskip

\noindent
\underline{Case 4}:   There is $\chi,\mu \le \chi < \lambda \le 2^\chi$ and:
$\lambda$ is singular or at least $(\chi^{\aleph_0})^+<\lambda$.

Like the proof of \ref{7.6}(2).
\medskip

\noindent
\underline{Case 5}:  $\lambda$ is strong limit singular $\cf(\lambda) 
> 2^{\aleph_{0}}$.

By the proof of \ref{7.6}(3) using models $N_{\eta} $ of cardinality $2^{\aleph_{0}}$, (i.e. choose $\kappa=2^{\aleph_{0}}$);  and demand $[N_\eta]^{\aleph_0} \subseteq N_\eta$ and using \cite[1.16=La45]{Sh:E62}.  Alternatively, in its proof notice that by 
thinning $\cT'$ a bit more we can get: let $h_{\ast} \in N_{\langle\rangle}$ be a one to one function from $\lambda$ onto
${}^\omega \lambda$, then:
\mn
\begin{enumerate}
\item[$(*)$] if $k < \lg(\eta) \, \& \, \eta \in {\cT} \, \& \,
\alpha \in N_{\eta} \, \& \, \bigwedge\limits_{\ell < \omega} 
(h_{\ast}(\alpha))(\ell) \in  N_{\eta \rest k}$ then $\alpha \in 
N_{\eta \rest k}$.
\end{enumerate}
\mn
The point is: \wilog \, $k+1 = \lg(\eta)$ and for each $\nu \in {\cT}$  
of length $k$,
\[
|\{\eta:\nu \triangleleft \eta \in {\cT} \, \& \,
\ell g(\eta) = k+1 \text{  and }(*) \text{ fails for } \eta,k\}| \le
\kappa^{\aleph_{0}}.
\]

\mn
For (viii), $\bigcup\limits_{n<\omega} M_n$ is as required 
by clause (ix) of Subfact \ref{7.6C}.  Note that $(v)^+$ is satisfied 
by the proof of \ref{7.6}.
\medskip

\noindent
\underline{Case 6}:  $\lambda$ strong limit, $\lambda =
\aleph_{\alpha +\omega}$.

The proof of \ref{7.10} or even better \ref{7.10A} 
gives this, too (for $(viii)^+$ the suitable ``initial segment" of  
$M_{A}$ can serve as $M$).
\medskip

\noindent
\underline{Case 7}:  $\lambda = \sum\limits_{i<\cf(\mu)} \mu_i > \mu$, where $\mu_{i}$ is
increasing with $i$, $\cf(\mu_i)=\aleph_0,p(\mu_i)>\mu_i^+$ see
\cite[3.22=Lpcf.8]{Sh:E62}.  

By the proof of \ref{7.7}.
\end{PROOF}
\bigskip

\noindent
We now turn to separable reduced $\primep$-groups, continuing
\cite[2.11=Lh5]{Sh:E59}, but the proof apply to more cases.

\begin{claim}\label{7.15}
1) We can define for every 
$I \in K^\omega_{\tr}$ and prime $\primep$, a separable
reduced abelian $\primep$-group $\modG^a_{I}$ such that:
\mn
\begin{enumerate}
\item[$(*)_{0}$] $\modG^a_{I}$ has cardinality $|I| + 2^{\aleph_{0}}$
\sn
\item[$(*)_{1}$]  if $I,J \in K^\omega_{\tr},I$ is 
$(2^{\aleph_0},2^{\aleph_0})$-super$^+$-unembeddable into $J$ 
(see \ref{7.14}) \then \, $\modG^a_{I}$ is not embeddable
into $\modG^a_{J}$ (i.e. there is no mono-morphism from  
$\modG^a_{I}$ into $\modG^a_{J}$;  we do not require purity).
\end{enumerate}
\mn
2) For $\lambda > 2^{\aleph_{0}}$ and prime $\primep$ there is a 
family of $2^\lambda$ separable reduced abelian $\primep$-groups, each 
of power $\lambda$,  no one embeddable into another.
\end{claim}

\begin{PROOF}{\ref{7.15}}
Part (2) follows from part (1) and \ref{7.14}.

\noindent
1)  \underline{Stage A}:  On the 
definition of ``super$^+$ unembeddable'' see \ref{7.14}. 
We choose a family $\{f_{\alpha}:\alpha < \alpha^*\}$ with $\alpha^*
\le 2^{\aleph_0}$ such that:
\mn
\begin{enumerate}
\item[$\boxdot$]  $(a) \quad f_{\alpha} \in {}^\omega \omega$,
\sn
\item[${{}}$]  $(b) \quad f_{\alpha}$ is (strictly) increasing, 
$f_{\alpha}(0) = 0$,
\sn
\item[${{}}$]   $(c) \quad$ if $h_{1}$ is a function from 
${}^{\omega >}\omega$ to $\omega$, \then \, for some $\alpha$,  
for infinitely many 

\hskip30pt  $n,f_{\alpha}(n) > h_{1}(f_{\alpha} \rest n)$,
\sn
\item[${{}}$]  $(d) \quad 
\alpha \ne \beta \Rightarrow  f_{\alpha} \ne f_{\beta}$,
\sn
\item[${{}}$]  $(e) \quad \langle 
f_\alpha(n+1) - f_\alpha(n):n < \omega\rangle$ goes
to infinity (for convenience).
\end{enumerate}
\mn
Obviously, there is such a sequence  with $\alpha^* = 2^{\aleph_{0}}$.

For any $I \in K^\omega_{\tr}$ let $\modG^a_{I}$ be the abelian group
generated by
\[
\{x_{\eta,\rho}:\eta \in P^I_{n},\rho \in {}^{n+1}\omega 
\text{ and } n<\omega\} \cup \{y^n_{\eta,\alpha}:\eta \in
P^I_{\omega},\alpha < \alpha^* \text{ and } n<\omega\}
\]

freely except the equations:
\mn
\begin{enumerate}
\item[${{}}$]  $\primep^{\rho(n)}x_{\eta,\rho} = 0$ for $\eta \in P^I_{n},
\rho \in {}^{n+1}\omega,n < \omega$
\sn
\item[${{}}$]  $(\primep^{f_{\alpha}(n+1) - f_{\alpha}(n)} 
y^{n+1}_{\eta,\alpha}) = y^n_{\eta,\alpha} - x_{\eta \rest n,
f_{\alpha} \rest (n+1)}$ for $\eta \in P^I_{\omega},\alpha < \alpha^*,
n < \omega$
\end{enumerate}
\mn
so actually
\[
y^n_{\eta,\alpha} = \sum\{\primep^{f_{\alpha}(\ell)-f_{\alpha}(n)}
x_{\eta \rest \ell,f_{\alpha} \rest (\ell +1)}:\ell \in
\Dom(f_{\alpha}),\ell \ge n\}
\]

\mn
Recall $\modG^a_I$ is a separable reduced 
abelian $\primep$-group (see \cite{Fu}) and:
\mn
\begin{enumerate}
\item[$\odot$]  in $\modG^a_I,\|-\|_\primep$ is a norm \underline{where} $\|x\|_{\primep} = \inf\{2^{-n}:x$ is divisible by $\primep^n$ in $\modG^a_I\}$.

Now, 

\item[$(*)_0$]   for any $n < \omega,\eta \in P^I_{n}$, and $\rho \in
{}^{n+1}\omega$, there is a projection ${\bf h} = {\bf h}^I_{\eta,\rho}$ of
$\modG^a_{I}$ (i.e. an endomorphism of this group which is the identity 
on its range) defined as follows:
\sn
\begin{enumerate}
\item[${{}}$]  $(\alpha) \quad$ if $m < \omega,\nu \in P^I_{m},\varrho \in
{}^{m+1}\omega$ then
\[
(\nu,\varrho) \ne (\eta,\rho) \Rightarrow {\bf h}(x_{\nu,{\varrho}}) = 0
\]
and 
\[
(\nu,\varrho) = (\eta,\rho) \Rightarrow {\bf h}(x_{\nu,\varrho}) = 
x_{\nu,\varrho}
\]
\sn
\item[${{}}$]  $(\beta) \quad$ if $\nu \in P^I_{\omega},\alpha < 
\alpha^*,m  < \omega$ then:
\[
(\nu \rest n,f_{\alpha} \rest (n+1)) \ne (\eta,\rho) \Rightarrow
{\bf h}(y^m_{\nu,\alpha}) = 0,
\]
\item[${{}}$]  $(\gamma) \quad m > n \, \& \, (\nu \rest n,
f_{\alpha} \rest (n +1)) = (\eta,\rho) \Rightarrow {\bf h}(y^m_{\nu,\alpha}) = 0$,
\sn
\item[${{}}$]  $(\delta) \quad m \le n \, \& \,(\nu \rest n,f_{\alpha} 
\rest (n + 1)) = (\eta,\rho) \Rightarrow {\bf h}(y^m_{\nu,\alpha}) 
= \primep^{f_{\alpha} (n)-f_{\alpha}(m)}x_{\eta,\rho}$.
\end{enumerate}
\end{enumerate}
\mn
Also note:
\mn
\begin{enumerate}
\item[$(*)_1$]  if $I \in K^\omega_{\tr}$ for every $z \in \modG^a_I$ and $m$
there is $z'\in \modG^a_I$ such that
\sn
\begin{enumerate}
\item[$(a)$]  $z - z'$ is divisible by $\numbp^m$ in $\modG^a_I$
\sn
\item[$(b)$]  $z' \in \sum\{\bbZ x_{\eta,\rho}$: for some $n < \omega$
 we have: $\eta \in P^I_n$ and $\rho \in {}^{n+1}\omega\}$.
\end{enumerate}
\end{enumerate}
\mn
\underline{Stage B}:  For proving the claim toward contradiction 
we assume:
\mn
\begin{enumerate}
\item[$\boxplus$]   $I \in K^\omega_{\tr}$ is super$^+$-unembeddable into  $J
\in K^\omega_{\tr}$, (i.e. as in \ref{7.14}) but ${\bf g}$ is an embedding of
$\modG^a_{I}$ into $\modG^a_{J}$.
\end{enumerate}
\mn
Let $\chi$ be large enough and let
$\eta \in P^I_\omega,\langle M_{n},N_{n}:n < \omega \rangle$ and $M$ be as guaranteed in \ref{7.14}, and the following belongs to $M_{0}$:

${\bf g},I,J,\modG^a_I,\modG^a_J$
and the functions $(\eta,\rho) \mapsto x_{\eta,\rho},(\eta,\alpha,n) 
\mapsto y^n_{\eta,\alpha}$ and so $(\eta,\rho) \mapsto {\bf h}^I_{\eta,\rho},
(\eta,\rho) \mapsto {\bf h}^J_{\eta,\rho}$  
and $\{\alpha:\alpha < \alpha^*\}\subseteq M_0$.

Remember $\eta \rest (n+1) \in N_{n}\setminus M_{n}$ (by $(v)^+$
there).  For $\iota < \omega,\rho \in {}^{\iota +2}\omega$ let
\[
k_{\rho} \coloneqq {\bf n}(\numbp^\iota \mathbf{g}(x_{\eta \rest (\iota +1),\rho}),
\modG^a_{J} \cap M_{\ell})
\]

\mn
where for $y \in \modG^a_{J}$ and $\modG \subseteq \modG^a_{J}$ we
let:
\[
{\bf n}(y,\modG) = \sup \{k: \text{ for some } z \in \modG, y - z 
\text{ is divisible in } \modG^a_{J} \text{ by } \numbp^k\}.
\]

\noindent
\underline{Stage C}:

Now, 

\begin{fact}\label{7.15V}
    We have: 
    
    \begin{enumerate}
    \item[$\otimes$]  $k_{\rho} < \omega$ \when \, $\rho \in {}^{\ell +2}
      \omega,\ell < \omega$.
    \end{enumerate}
\end{fact}

\begin{PROOF}{\ref{7.15V}}
    Otherwise, we can let
    \begin{enumerate}
    \item[$(*)_2$] $\modG^a_J \vDash {\bf g}(x_{\eta \rest (\ell +1),\rho}) =
    \sum\limits_{(\nu,\rho) \in u_1} a_{\nu,\rho} x_{\nu,\rho} + 
    \sum\limits_{(\eta ,\beta) \in u_2} b_{\eta,\beta}
    y^{m(\eta,\beta)}_{\eta,\beta}$ with
    \sn
    \begin{enumerate}
    \item[$(a)$]   $u_{1} \subseteq \{(\nu,\rho):\nu \in P^J_{k}$ and $\rho \in
    {}^{k+1}\omega \text{ for some } k < \omega\}$,
    \sn
    \item[$(b)$]   $u_{2} \subseteq \{(\nu,\beta):\nu \in P^J_{\omega}$
     and $\beta < \alpha^*\}$,
    \sn
    \item[$(c)$]  $a_{\nu,\rho},b_{\nu, \beta} \in \bbZ$,
    \sn
    \item[$(d)$]  $m(\nu,\beta) < \omega$,
    \sn
    \item[$(e)$]  $\modG_J \models `` a_\nu,x_{\nu,\rho} \ne 0$,
    and $b_{\nu,\beta} y^\ell_{\nu,\beta} \ne 0$ (in $\modG^a_{J})$
    \sn
    \item[$(f)$]  $u_1,u_2$ are finite.
    \end{enumerate}
    \end{enumerate}

    By the way $\modG^a_J$ was defined we can replace $y^{\bfn(\nu,\beta)}_{\nu,
    \beta}$ by $\numbp^{f_\beta(\bfn(\nu,\beta)+1) - f_\beta(\bfn(\nu,\beta))} y^{\bfn(\nu,
     \beta)+1}_{\nu,\beta} + x_{\nu \rest (\bfn(\nu,\beta)+1),\beta}$ and
     repeat this, hence using clause (e) of $\boxdot$, \wilog \, for some $m_0<
     m_1< \omega$ large enough:
    \mn
    \begin{enumerate}
    \item[$(*)_3$]  $(a) \quad (\eta_1,\beta_1) \in u_2 \, \& \, 
    (\eta_2,\beta_2) \in u_2 \, \& \, \eta_1 \ne \eta_2 \Rightarrow 
    \eta_1\rest m_0 \ne \eta_2 \rest m_0$
    \sn
    \item[${{}}$]  $(b) \quad (\eta,\beta_1) \in u_2 \, \&
    \,(\eta,\beta_2) \in u_2 \, \& \, \beta_1 \ne \beta_2 \Rightarrow  
    f_{\beta_1} \rest m_0 \ne f_{\beta_2} \rest m_0$
    \sn
    \item[${{}}$]   $(c) \quad$ if $(\eta,\beta) \in u_2$ then
    \sn
    \begin{enumerate}
    \item[${{}}$]  $(\alpha) \quad \bfn(\eta,\beta)> m_0$
    \sn
    \item[${{}}$]  $(\beta) \quad$ if $m_0 \le m < \bfn(\eta,\beta)$ 
    then $(\eta \rest m,f_\beta \rest (m+1)) \in u_1$ and 
    
    \hskip30pt $\numbp^{f_\beta(m(\eta,\beta))-f_\beta(m)} 
    a_{\eta \rest m,f_\beta \rest (m+1)} = b_{\eta, \beta}$
    \sn
    \item[${{}}$]  $(\gamma) \quad |a_{\eta \rest m_0,f_\beta
      \rest(m_0+1)}| < m_1$
    \sn
    \item[${{}}$]  $(\delta) \quad b_{\eta,\beta} 
    y^{\bfn(\eta,\beta)}_{\eta,\beta}$ is divisible by $\numbp^{m_1}$ in
    $\modG^a_J$.
    \end{enumerate}
    \sn
    \item[${{}}$]   $(d) \quad$ if ($\nu,\rho) \in u_1$ then $a_{\nu,\rho}
    x_{\nu,\rho}$ is not divisible by $\numbp^{m_1}$ in $\modG^a_J$.
    \end{enumerate}
    
    So, using $(*)_{0} + (*)_1 + (*)_2$ in  $\modG^a_J$ and our assumption
    toward contradiction that $k_\rho = \omega$, necessarily $u_{1} \in
    M_{\ell}$, hence $(\nu,\rho) \in u_1 \Rightarrow a_{\nu,\rho} 
    x_{\nu,\rho} \in M_\ell$.  Repeating this increasing $m_1$ (hence the
    $\bfn(\eta,\beta)$'s) we get also $(\nu,\beta  \in  u_{2} \Rightarrow 
    \bigwedge\limits_{i < \omega} \nu \rest i \in M_\ell$, 
    hence by clause (vii) of \ref{7.14} we have $(\nu,\beta) \in 
    u_2 \Rightarrow \nu \in M_\ell \Rightarrow y^m_{\nu,\beta} \in 
    M_\ell \Rightarrow b_{\nu,\beta} y^m_{\nu,\beta} \in M_\ell$. 
    Together by $(*)_2$ in $\otimes$ we have $\mathbf g(x_{\eta \rest
    (\ell+1),\rho}) \in M_\ell$, but $\mathbf g \in M_0$ is one to one, hence $\eta \rest (\ell+1)\in M_n$, contradiction.  So really $k_\rho <
    \omega$,  i.e. $\otimes$ holds.
\end{PROOF}

\noindent
\underline{Stage D}:

By the choice of  $\langle f_{\alpha}:\alpha < \alpha^* \rangle$ for some
$\alpha < \alpha^*$,  for infinitely many $\ell < \omega$  we have:
$f_{\alpha}(\ell +1) > k_{f_{\alpha} \rest (\ell +1)}$.

Now in $\modG^a_{I}$ for each $m < \omega,y^0_{\eta ,\alpha } -
\sum\limits_{n<m} \numbp^{f_{\alpha}(n)}x_{\eta \rest n,f_{\alpha} 
\rest (n+1)}$ is divisible by $\numbp^{f_{\alpha}(m)}$ hence for each  
$m < \omega$:
\mn
\begin{enumerate}
\item[$(*)_{4}$]  ${\bf g}(y^0_{\eta,\alpha}) - 
\sum\limits_{n<m} \numbp^{f_{\alpha}(n)}{\bf g}(x_{\eta \rest n,
f_{\alpha} \rest (m+1)})$ is divisible by $\numbp^{f_{\alpha}(m)}$ 
in $\modG^a_J$.
\end{enumerate}
\mn
Now ${\bf g}(y^0_{\eta,\alpha})$ has, for some $n(0)$, the form
\[
\sum \{b_{\eta,\alpha} y^{n(0)}_{\eta,\alpha}:\eta \in \seteq_{0},
\alpha \in u_{0}\} + \sum \{a_{\eta,\rho} x_{\eta,\rho}:(\eta,{\rho}) 
\in \seteq_{1}\}
\]

\mn
where:

$u_{0}$ a finite subset of $\alpha^*$

$\seteq_{0}$ a finite subset of $P^J_{\omega}$

$\seteq_{1}$ a finite subset of $\bigcup\limits_{n < \omega} 
(P^J_n \times {}^{n+1}\omega)$

$b_{\eta,\alpha},a_{\eta,\rho} \in \bbZ$.

\noindent
Let
\[
\seteq'_{1} = \{\eta:\text{ for some } \rho \in {}^{\omega >}\omega 
\text{ we have }(\eta,{\rho}) \in \seteq_{1}\}.
\]

\mn
We can find $n(1) < \omega$ such that:
\mn
\begin{enumerate}
\item[$(\alpha)$]   $n(1) > n(0)$
\sn
\item[$(\beta)$]   $\seteq'_{1} \cap (\bigcup\limits_{n < \omega}
  M_{n}) \subseteq M_{n(1)}$
\sn
\item[$(\gamma)$]  $\eta \in \seteq_{0} \, \& \, n \ge n(1) \Rightarrow
\{\eta \rest \ell:\ell < \omega\} \cap N_{n} \subseteq M_{n}$.
\end{enumerate}
\mn
(For $(\gamma)$ use clause (vi) of $(*)$ of \ref{7.14}, i.e. of \ref{7.1}.)

So by the choice of $\alpha$ we can find $\ell$ such that:
\[
n(1) < \ell  < \omega
\]
\[
f_{\alpha} (\ell +1) > k_{f_{\alpha} \rest (\ell +1)}.
\]

\mn
Now by $(*)_{4}$
\[
{\bf n}({\bf g}(y^0_{\eta,\alpha}),\modG^a_{J} \cap N_\ell) \ge
f_{\alpha}(\ell + 1)
\]

\mn
as exemplified by $\sum\limits_{i \le \ell} \numbp^{f_{\alpha}(i)}{\bf g}
(x_{\eta \rest i,f_{\alpha} \rest (i+1)})$.

Now if
\[
{\bf n}({\bf g}(y^0_{\eta,\alpha}),\modG^a_{J} \cap M_{\ell})
\text{ is } \ge f_{\alpha}(\ell + 1)
\]

\mn
then we get (use again $(*)_{4}$)
\[
{\bf n}(\sum\limits_{i\le\ell} \numbp^{f_{\alpha}(i)}{\bf g}
(x_{\eta \rest i,f_{\alpha} \rest (i+1)}),\modG^a_{J} \cap M_{\ell}) 
\text{ is } \ge f_\alpha(\ell+1)
\]

\mn
but for $i < \ell$
\[
{\bf g}(x_{\eta \rest i,f_{\alpha} \rest (i+1)}) \in
M_{i} \text{ (as } \eta \rest i,\mathbf{g} \in M_{i})
\]

\mn
so we get
\[
{\bf n}(\numbp^{f_{\alpha} (\ell)}{\bf g}(x_{\eta \rest \ell,f_{\alpha} \rest
(\ell +1)}),\modG^a_{J} \cap M_{\ell}) \text{ is } \ge f_{\alpha}(\ell+1) >
k_{f_{\alpha} \rest (\ell +1)}.
\]

\mn
But this contradicts the definition of $k_{f_{\alpha} \rest (\ell +1)}$.

So
\[
{\bf n}({\bf g}(y^0_{\eta,\alpha}),\modG^a_{J} \cap M_{\ell}) <
f_{\alpha}(\ell + 1) \le {\bf n}({\bf g}(y^0_{\eta,\alpha}),\modG^a_{J} \cap
N_{\ell}).
\]

\mn
But this contradicts $\ell > n(1)$.    
\end{PROOF}

\begin{remark} \label{7.15A}
    Really, the proof of \ref{7.15} is a kind of simple black box:
    we attach to every $\eta \in P^{I_\zeta}_\omega$, a first order 
    theory $T_{\eta}$ such that:
    
    if $I = I_\zeta,J = \sum\limits_{\xi \ne \zeta} I_{\xi},$ and $\chi^*$ is 
    regular large enough, $x \in {\cH}(\chi^*)$, then we can find $\eta,
    \langle M_{\eta},N_{n}:n < \omega \rangle$ as in \ref{7.1},
    \ref{7.14} and $T_{\eta}$ is the first order theory of 
    \newline
    $(\bigcup\limits_{n}
    M_{n},M_{m},N_{m},\eta \rest m)_{m <\omega}$. 
    We need of course $\kappa \ge 2^{\aleph_0}$.
    \end{remark}
    
    \begin{remark}
    \label{7.15B}
    1)  We could have used in the proof only ($(*)$ of Def.\ref{7.1} and)
    (vii) of \ref{7.14}; but as we have used also $(v)^+$ from \ref{7.14}
    we can add:
    \mn
    \begin{enumerate}
    \item[$(*)$]  $\alpha^* = {\gb} = \min\{|{\cF}|:\cF$ is a set of
      functions from $\omega$ to $\omega$ such that for every  
    $g \in {}^\omega \omega$ for some  $f \in {\cF}$ we have 
    $(\exists^\infty n)[g(n) < f(n)]\}$.
    \end{enumerate}
    \mn
    Hence we can improve \ref{7.15} in two ways:
    \mn
    \begin{enumerate}
    \item[$(\alpha)$]  we can omit (viii) in \ref{7.14} and add $|\modG^a_{I}|
    = |I| + {\gb}$
    \end{enumerate}
    \mn
    \underline{or}
    \mn
    \begin{enumerate}
    \item[$(\beta)$]  we can weaken the ``super$^+$" assumption and omit
    $(v)^+,(viii)$ from \ref{7.14}.
    \end{enumerate}
\end{remark}

\noindent
Of course (assuming less, getting less)
\begin{conclusion}
\label{7.15C}
If $\lambda > \aleph_0$ then there are $2^\lambda$ separable reduced
abelian $\numbp$-groups of cardinality $\lambda$ no one purely embeddable into
another.
\end{conclusion}

\begin{PROOF}{\ref{7.15C}}
    By \ref{7.11}. In detail, by \ref{7.11} there is $\langle I_\alpha: \alpha <
    2^\lambda\rangle$ such that $\alpha \ne \beta$ implies $I_\alpha$ is
    $(\aleph_0, \aleph_0)$-super unembeddable into $I_\beta$. But $\alpha
    \neq \beta$ implies $I_\alpha$ is strongly
    $\varphi_{\tr}$-unembeddable into $I_\alpha$. Now $\modG_{I_\alpha}$ is
    defined in \cite[2.11=Lh5]{Sh:E59}. By
    \cite[2.13=h8]{Sh:E59} we have $\modG_{I_\alpha}$ is a separable
    reduced abelian $\numbp$-group. Lastly,  ``$\modG_{I_\alpha}$ not purely
    embeddable into $\modG_{I_\beta}$'' by \cite[2.14 = Lh11]{Sh:E59}.
\end{PROOF}

\newpage

\bibliographystyle{amsalpha}
\bibliography{shlhetal}

\end{document}